\documentclass[11pt,a4paper]{amsart}
\usepackage{amssymb,amsthm}
\usepackage{bm}

\usepackage[truedimen,margin=30truemm]{geometry}

\usepackage[pdftex,colorlinks=true,bookmarks=true,citecolor=blue,linkcolor=blue,pdfauthor={},pdftitle={}]{hyperref}

\theoremstyle{plain}
\newtheorem{theorem}{Theorem}[section]
\newtheorem{lemma}[theorem]{Lemma}
\newtheorem{proposition}[theorem]{Proposition}
\newtheorem{corollary}[theorem]{Corollary}

\theoremstyle{definition}
\newtheorem{definition}[theorem]{Definition}
\newtheorem{assumption}[theorem]{Assumption}

\theoremstyle{remark}
\newtheorem{remark}[theorem]{Remark}

\numberwithin{equation}{section}

\begin{document}

\title[Cahn--Hilliard equation in a curved thin domain]{Thin-film limit of the Cahn--Hilliard equation in a curved thin domain}

\author[T.-H. Miura]{Tatsu-Hiko Miura}
\address{Graduate School of Science and Technology, Hirosaki University, 3, Bunkyo-cho, Hirosaki-shi, Aomori, 036-8561, Japan}
\email{thmiura623@hirosaki-u.ac.jp}



\begin{abstract}
  We consider the Cahn--Hilliard equation with Neumann boundary conditions in a three-dimensional curved thin domain around a given closed surface.
  When the thickness of the curved thin domain tends to zero, we show that the weighted average in the thin direction of a weak solution to the thin-domain problem converges on the limit surface in an appropriate sense.
  Moreover, we rigorously derive a limit problem, which is the surface Cahn--Hilliard equation with weighted Laplacian, by characterizing the limit function as a unique weak solution to the limit problem.
  The proof is based on a detailed analysis of the weighted average and the use of Sobolev inequalities and elliptic regularity estimates on the curved thin domain with constants explicitly depending on the thickness.
  This is the first result on a rigorous thin-film limit of nonlinear fourth order equations in general curved thin domains.
\end{abstract}

\maketitle

\section{Introduction} \label{S:Intro}

\subsection{Problem settings and main results} \label{SS:Int_Main}
Let $\Gamma$ be a given closed $C^5$ surface in $\mathbb{R}^3$, and let $\bm{\nu}$ be the unit outward normal vector field of $\Gamma$.
Also, let $g_0$ and $g_1$ be $C^3$ functions on $\Gamma$ such that $g:=g_1-g_0\geq c$ on $\Gamma$ with some constant $c>0$ (note that we do not make any assumptions on the signs of $g_0$ and $g_1$).
For a sufficiently small $\varepsilon>0$, we define a curved thin domain in $\mathbb{R}^3$ by
\begin{align} \label{E:Def_CTD}
    \Omega_\varepsilon := \{y+r\bm{\nu}(y,t) \mid y\in\Gamma, \, \varepsilon g_0(y)<r<\varepsilon g_1(y)\}
\end{align}
and consider the Cahn--Hilliard equation with Neumann boundary conditions
\begin{align} \label{E:CH_CTD}
  \left\{
  \begin{aligned}
    \partial_tu^\varepsilon &= \Delta w^\varepsilon, \quad w^\varepsilon = -\Delta u^\varepsilon+F'(u^\varepsilon) \quad\text{in}\quad \Omega_\varepsilon\times(0,\infty), \\
    \partial_{\nu_\varepsilon}u^\varepsilon &= 0, \quad \partial_{\nu_\varepsilon}w^\varepsilon = 0 \quad\text{on}\quad \partial\Omega_\varepsilon\times(0,\infty), \\
    u^\varepsilon|_{t=0} &= u_0^\varepsilon \quad\text{in}\quad \Omega_\varepsilon,
  \end{aligned}
  \right.
\end{align}
where $\partial_{\nu_\varepsilon}$ is the outer normal derivative on $\partial\Omega_\varepsilon$.

The Cahn--Hilliard equation is a diffuse interface model of phase separation of binary mixtures, such as the spinodal decomposition of binary alloys \cite{CahHil58,Ell89,Nov08}.
Here, $u^\varepsilon$ is the order parameter which usually represents the concentration of one of the two components or the difference of concentrations of the two components.
Also, $w^\varepsilon$ is the chemical potential which is the functional derivative of the Ginzburg--Landau free energy
\begin{align*}
  E_\varepsilon(u^\varepsilon) := \int_{\Omega_\varepsilon}\left(\frac{|\nabla u^\varepsilon|^2}{2}+F(u^\varepsilon)\right)\,dx,
\end{align*}
where $F$ is the homogeneous free energy potential which is usually assume to have a double-well structure.
The Neumann boundary conditions are imposed to make sure that the conservation of mass and energy holds:
\begin{align*}
  \frac{d}{dt}\int_{\Omega_\varepsilon}u^\varepsilon(t)\,dx = 0, \quad \frac{d}{dt}E_\varepsilon(u^\varepsilon(t))+\int_{\Omega_\varepsilon}|\nabla w^\varepsilon(t)|^2\,dx = 0, \quad t>0.
\end{align*}
It is also known (see e.g. \cite{Fif00,Wu22}) that the Cahn--Hilliard equation is the gradient flow of $E_\varepsilon$ under the mass conservation with respect to the inner product
\begin{align*}
  (u_1,u_2)_{\mathcal{H}_\varepsilon} := \langle u_1,L_\varepsilon u_2\rangle_{H^1(\Omega_\varepsilon)} = \int_{\Omega_\varepsilon}(\nabla L_\varepsilon u_1)\cdot(\nabla L_\varepsilon u_2)\,dx, \quad u_1,u_2\in[H^1(\Omega_\varepsilon)]'.
\end{align*}
Here, $[H^1(\Omega_\varepsilon)]'$ is the dual space of $H^1(\Omega_\varepsilon)$ and $\langle\cdot,\cdot\rangle_{H^1(\Omega_\varepsilon)}$ is the duality product between $[H^1(\Omega_\varepsilon)]'$ and $H^1(\Omega_\varepsilon)$.
Also, $L_\varepsilon$ is the inverse of the negative Neumann Laplacian on $\Omega_\varepsilon$.

In this paper, we assume that $F$ is a regular potential and make the following assumption on $F$, which is satisfied by the standard double-well polynomial potentials
\begin{align*}
  F(z) = \frac{1}{4}(z^2-1)^2, \quad F(z) = \frac{1}{4}z^4-\frac{1}{2}z^2.
\end{align*}

\begin{assumption} \label{A:Poten}
  We assume that $F\in C^3(\mathbb{R})$ and it satisfies
  \begin{align} \label{E:Poten}
    F(z) \geq -C_0, \quad F''(z) \geq -C_2, \quad |F'''(z)| \leq C_3(|z|+1) \quad\text{for all}\quad z\in\mathbb{R},
  \end{align}
  where $C_0$, $C_2$, and $C_3$ are nonnegative constants independent of $z$.
\end{assumption}

\begin{remark} \label{R:Po_Grow}
  When the last inequality of \eqref{E:Poten} holds, for each $z\in\mathbb{R}$ we have
  \begin{align} \label{E:Int_F2nd}
    |F''(z)| = \left|F''(0)+\int_0^zF'''(s)\,ds\right| \leq |F''(0)|+C_2\left(\frac{1}{2}|z|^2+|z|\right) \leq c(|z|^2+1)
  \end{align}
  by Young's inequality, where $c>0$ is a constant independent of $z$.
  Similarly,
  \begin{align} \label{E:Int_Fzero}
    |F'(z)| \leq c(|z|^3+1), \quad |F(z)| \leq c(|z|^4+1), \quad z\in\mathbb{R}.
  \end{align}
  Hence, in general, $F$ may grow polynomially up to the fourth order.
\end{remark}

We consider a weak solution to \eqref{E:CH_CTD} in the $L^2$-class (see Definitions \ref{D:CHT_WeSo} and \ref{D:CHT_GlW} for the definition).
For $u_0^\varepsilon\in H^1(\Omega_\varepsilon)$, we can get the global-in-time existence and uniqueness of a weak solution by the Galerkin and energy methods (see Sections \ref{S:CHT_Weak} and \ref{S:Galer}).
The purpose of this paper is to study the thin-film limit of \eqref{E:CH_CTD}.
More precisely, we prove the convergence of a weak solution to \eqref{E:CH_CTD} as $\varepsilon\to0$ in an appropriate sense and rigorously derive a limit problem on $\Gamma$ as an equation satisfied by the limit of the weak solution to \eqref{E:CH_CTD}.
We also compare solutions to \eqref{E:CH_CTD} and the limit problem explicitly in terms of $\varepsilon$.

To state the main results, we fix some notations (see Sections \ref{S:Pre} and \ref{S:Ave} for details).
Let $\nabla_\Gamma$ and $\mathrm{div}_\Gamma$ be the tangential gradient and the surface divergence on $\Gamma$, respectively.
Also, let $\mathcal{H}^2$ be the two-dimensional Hausdorff measure.
We set
\begin{align*}
  A_gv := \frac{1}{g}\mathrm{div}_\Gamma(g\nabla_\Gamma v) \quad\text{on}\quad \Gamma, \quad E_g(v) := \int_\Gamma g\left(\frac{|\nabla_\Gamma v|^2}{2}+F(v)\right)\,d\mathcal{H}^2
\end{align*}
for a function $v$ on $\Gamma$, and call $A_g$ the weighted Laplacian and $E_g$ the weighted Ginzburg--Landau free energy.
Let $[H^1(\Gamma)]'$ be the dual space of $H^1(\Gamma)$, and let
\begin{align*}
  \mathcal{E}_T(\Gamma) := \{v\in L^2(0,T;H^1(\Gamma)) \mid \partial_tv\in L^2(0,T;[H^1(\Gamma)]')\}, \quad T>0.
\end{align*}
For a function $\varphi$ on $\Omega_\varepsilon$, we define the weighted average of $\varphi$ in the thin direction by
\begin{align*}
  \mathcal{M}_\varepsilon\varphi(y) := \frac{1}{\varepsilon g(y)}\int_{\varepsilon g_0(y)}^{\varepsilon g_1(y)}\varphi(y+r\bm{\nu}(y))J(y,r)\,dr, \quad y\in\Gamma.
\end{align*}
Here, $J(y,r):=\{1-r\kappa_1(y)\}\{1-r\kappa_2(y)\}$ with principal curvatures $\kappa_1$ and $\kappa_2$ of $\Gamma$, which is the Jacobian appearing in the change of variables formula
\begin{align} \label{E:Int_CoV}
  \int_{\Omega_\varepsilon}\varphi(x)\,dx = \int_\Gamma\int_{\varepsilon g_0(y)}^{\varepsilon g_1(y)}\varphi(y+r\bm{\nu}(y))J(y,r)\,dr\,d\mathcal{H}^2(y).
\end{align}
Now, let us state the main results.
The first one is a rigorous derivation of a limit problem of \eqref{E:CH_CTD} as $\varepsilon\to0$ by convergence of a solution and characterization of the limit.

\begin{theorem} \label{T:TFL_Weak}
  Let $u_0^\varepsilon\in H^1(\Omega_\varepsilon)$ and Assumption \ref{A:Poten} be satisfied.
  Also, let $(u^\varepsilon,w^\varepsilon)$ be the unique weak solution to \eqref{E:CH_CTD}.
  Suppose that the following conditions are satisfied:
  \begin{itemize}
    \item[(a)] there exist constants $c>0$, $\alpha\in[0,6/25]$, and $\varepsilon_0\in(0,1)$ such that
    \begin{align}
      \varepsilon^{-1}\|u_0^\varepsilon\|_{L^2(\Omega_\varepsilon)}^2+\varepsilon^{-1}|E_\varepsilon(u_0^\varepsilon)| &\leq c\varepsilon^{-2\alpha}, \label{E:TFu0_CTD} \\
      \|\mathcal{M}_\varepsilon u_0^\varepsilon\|_{L^2(\Gamma)}^2+|E_g(\mathcal{M}_\varepsilon u_0^\varepsilon)| &\leq c \label{E:TFu0_Sur}
    \end{align}
    for all $\varepsilon\in(0,\varepsilon_0)$,
    \item[(b)] there exists a function $v_0\in H^1(\Gamma)$ such that
    \begin{align*}
      \lim_{\varepsilon\to0}\mathcal{M}_\varepsilon u_0^\varepsilon = v_0 \quad\text{weakly in}\quad L^2(\Gamma).
    \end{align*}
  \end{itemize}
  Then, there exist functions $v$ and $\mu$ on $\Gamma\times(0,\infty)$ such that
  \begin{align*}
    v \in \mathcal{E}_T(\Gamma)\cap L^\infty(0,T;H^1(\Gamma)), \quad \mu\in L^2(0,T;H^1(\Gamma))
  \end{align*}
  and
  \begin{align} \label{E:TFL_Weak}
    \begin{alignedat}{3}
      \lim_{\varepsilon\to0}\mathcal{M}_\varepsilon u^\varepsilon &= v &\quad &\text{weakly-$\ast$ in} &\quad &L^\infty(0,T;H^1(\Gamma)), \\
      \lim_{\varepsilon\to0}\mathcal{M}_\varepsilon w^\varepsilon &= \mu &\quad &\text{weakly in} &\quad &L^2(0,T;H^1(\Gamma))
    \end{alignedat}
  \end{align}
  for all $T>0$, and $(v,\mu)$ is a unique global weak solution to the limit problem
  \begin{align} \label{E:CH_Lim}
    \left\{
    \begin{aligned}
      \partial_tv &= A_g\mu, \quad \mu = -A_gv+F'(v) \quad\text{on}\quad \Gamma\times(0,\infty), \\
      v|_{t=0} & = v_0 \quad\text{on}\quad \Gamma.
    \end{aligned}
    \right.
  \end{align}
\end{theorem}

We refer to Definitions \ref{D:CLi_WeSo} and \ref{D:CLi_GlW} for the definition of a weak solution to \eqref{E:CH_Lim}.
As in the case of the thin-domain problem \eqref{E:CH_CTD}, the global-in-time existence and uniqueness of a weak solution to \eqref{E:CH_Lim} can be shown by the Galerkin and energy methods for the general initial data $v_0\in H^1(\Gamma)$.
However, that result is also obtained as a consequence of Theorem \ref{T:TFL_Weak} (see Corollary \ref{C:CLi_Ex}).

\begin{remark} \label{R:TFL_LP}
  If $g\equiv1$, then $A_g$ reduces to the canonical Laplace--Beltrami operator on $\Gamma$.
  Thus, the limit problem \eqref{E:CH_Lim} can be seen as the surface Cahn--Hilliard equation with weighted Laplacian.
  For a (smooth) solution to \eqref{E:CH_Lim}, we see by integration by parts on $\Gamma$ (see \eqref{E:IbP_Ag}) that the conservation of weighted mass and energy holds:
  \begin{align*}
    \frac{d}{dt}\int_\Gamma gv(t)\,d\mathcal{H}^2 = 0, \quad \frac{d}{dt}E_g(v(t))+\int_\Gamma g|\nabla_\Gamma\mu|^2\,d\mathcal{H}^2 = 0, \quad t>0.
  \end{align*}
  Note that $\Gamma$ and $g$ are independent of time.
  Also, as in the case of flat domains with $g\equiv1$ (see e.g. \cite{Fif00,Wu22}), the limit problem \eqref{E:CH_Lim} can be seen as the gradient flow of $E_g$ under the weighted mass conservation with respect to the weighted inner product
  \begin{align*}
    (v_1,v_2)_{\mathcal{H}_g} := \langle v_1,gL_g v_2\rangle_{H^1(\Gamma)} = \int_\Gamma g(\nabla_\Gamma L_gv_1)\cdot(\nabla_\Gamma L_gv_2)\,d\mathcal{H}^2, \quad v_1,v_2\in[H^1(\Gamma)]'.
  \end{align*}
  Here, $\langle\cdot,\cdot\rangle_{H^1(\Gamma)}$ is the duality product between $[H^1(\Gamma)]'$ and $H^1(\Gamma)$.
  Also, $L_g$ is the inverse of $-A_g$ (see Lemma \ref{L:InA_Sur} for the precise definition).
\end{remark}

\begin{remark} \label{R:TFL_Ini}
  We take the absolute values of $E_\varepsilon(u_0^\varepsilon)$ and $E_g(\mathcal{M}_\varepsilon u_0^\varepsilon)$ in \eqref{E:TFu0_CTD} and \eqref{E:TFu0_Sur}, since they may get negative under Assumption \ref{A:Poten}.
  Moreover, in \eqref{E:TFu0_CTD}, we divide the integrals over $\Omega_\varepsilon$ by $\varepsilon$, the scale of the thickness of $\Omega_\varepsilon$.
  It seems to be more natural that the scaled energy of $u_0^\varepsilon$ on $\Omega_\varepsilon$ is assumed to be uniformly bounded with respect to $\varepsilon$, but it is allowed in \eqref{E:TFu0_CTD} that the scaled energy grows as $\varepsilon\to0$ of order $\varepsilon^{-2\alpha}$.
  However, in \eqref{E:TFu0_Sur}, the energy of the weighted average $\mathcal{M}_\varepsilon u_0^\varepsilon$ on $\Gamma$ is assumed to be bounded uniformly in $\varepsilon$.
\end{remark}

\begin{remark} \label{R:TFL_IWC}
  Actually, the condition (b) follows from (a), if we replace the limit $\varepsilon\to0$ by that of a subsequence.
  Indeed, for $\zeta\in H^1(\Gamma)$, we see by $g>0$ on $\Gamma$ and \eqref{E:Poten} that
  \begin{align*}
    \frac{1}{2}\int_\Gamma g|\nabla_\Gamma\zeta|^2\,d\mathcal{H}^2 = E_g(\zeta)-\int_\Gamma gF(\zeta)\,d\mathcal{H}^2 \leq E_g(\zeta)+C_0\int_\Gamma g\,d\mathcal{H}^2.
  \end{align*}
  By this inequality, \eqref{E:TFu0_Sur}, and $g\geq c$ on $\Gamma$, we find that $\mathcal{M}_\varepsilon u_0^\varepsilon$ is bounded in $H^1(\Gamma)$ and thus converges weakly in $H^1(\Gamma)$ up to a subsequence.
  Here, we impose the condition (b) in order to get the convergence \eqref{E:TFL_Weak} of the full sequence.
  Also, for the proof, it is sufficient to assume the weak convergence in $L^2(\Gamma)$, not in $H^1(\Gamma)$.
\end{remark}

Next, we estimate the difference of $\mathcal{M}_\varepsilon u^\varepsilon$ and $v$ on $\Gamma$ explicitly in terms of $\varepsilon$.

\begin{theorem} \label{T:DiE_Sur}
  Let Assumption \ref{A:Poten} be satisfied, and let $u_0^\varepsilon\in H^1(\Omega_\varepsilon)$ and $(u^\varepsilon,w^\varepsilon)$ be the unique weak solution to \eqref{E:CH_CTD}.
  Also, let $v_0\in H^1(\Gamma)$ and $(v,\mu)$ be the unique weak solution to \eqref{E:CH_Lim}.
  Suppose that the condition (a) of Theorem \ref{T:TFL_Weak} holds and
  \begin{align} \label{E:DES_Ini}
    \int_\Gamma g\mathcal{M}_\varepsilon u_0^\varepsilon\,d\mathcal{H}^2 = \int_\Gamma gv_0\,d\mathcal{H}^2.
  \end{align}
  Then, there exists a constant $c>0$ independent of $\varepsilon$ such that
  \begin{align} \label{E:DiE_Sur}
    \|\mathcal{M}_\varepsilon u^\varepsilon-v\|_{L^2(0,T;H^1(\Gamma))} \leq ce^{cT}\Bigl\{\|\mathcal{M}_\varepsilon u_0^\varepsilon-v_0\|_{L^2(\Gamma)}+\varepsilon^{1-3\alpha}(1+T)^3\Bigr\}
  \end{align}
  for all $T>0$.
  In particular, if $\mathcal{M}_\varepsilon u_0^\varepsilon\to v_0$ strongly in $L^2(\Gamma)$ as $\varepsilon\to0$, then
  \begin{align*}
    \lim_{\varepsilon\to0}\mathcal{M}_\varepsilon u^\varepsilon = v \quad\text{strongly in}\quad L^2(0,T;H^1(\Gamma)).
  \end{align*}
\end{theorem}

In Theorem \ref{T:DES_Pre}, we give a more precise version of \eqref{E:DiE_Sur} in terms of the inverse $L_g$ of $-A_g$ (see Lemma \ref{L:InA_Sur} for the definition of $L_g$).
Note that the condition \eqref{E:DES_Ini} is preserved in time and required to apply $L_g$ to the difference of $\mathcal{M}_\varepsilon u^\varepsilon$ and $v$.

We also have a difference estimate for $u^\varepsilon$ and $v$ on $\Omega_\varepsilon$.
Let $\mathbf{P}$ be the orthogonal projection onto the tangent plane of $\Gamma$.
For a function $\eta$ on $\Gamma$, we write $\bar{\eta}$ for the constant extension of $\eta$ in the normal direction of $\Gamma$ (see Section \ref{SS:Pr_Sur} for details).

\begin{theorem} \label{T:DiE_CTD}
  Under the assumptions of Theorem \ref{T:DiE_Sur}, we have
  \begin{multline} \label{E:DiE_CTD}
    \varepsilon^{-1/2}\Bigl(\|u^\varepsilon-\bar{v}\|_{L^2(0,T;L^2(\Omega_\varepsilon))}+\left\|\overline{\mathbf{P}}\nabla u^\varepsilon-\overline{\nabla_\Gamma v}\right\|_{L^2(0,T;L^2(\Omega_\varepsilon))}\Bigr) \\
    \leq ce^{cT}\Bigl\{\|\mathcal{M}_\varepsilon u_0^\varepsilon-v_0\|_{L^2(\Gamma)}+\varepsilon^{1-3\alpha}(1+T)^3\Bigr\}
  \end{multline}
  for all $T>0$, where $c>0$ is a constant independent of $\varepsilon$ and $T$.
  Moreover,
  \begin{align} \label{E:DiE_ND}
    \varepsilon^{-1/2}\|\partial_\nu u^\varepsilon\|_{L^2(0,T;L^2(\Omega_\varepsilon))} \leq c\varepsilon^{1-\alpha}(1+T) \quad\text{for all}\quad T>0,
  \end{align}
  where $\partial_\nu u^\varepsilon:=\bar{\bm{\nu}}\cdot\nabla u^\varepsilon$ is the derivative of $u^\varepsilon$ in the normal direction of $\Gamma$.
\end{theorem}

Note that the $L^2(\Omega_\varepsilon)$-norm is divided by $\varepsilon^{1/2}$ in \eqref{E:DiE_CTD}, since it involves the square root of the thickness of $\Omega_\varepsilon$.
If $\alpha=0$ in \eqref{E:TFu0_CTD}, i.e. the scaled energy of $u_0^\varepsilon$ on $\Omega_\varepsilon$ is uniformly bounded, then the estimate \eqref{E:DiE_CTD} shows that $v$ approximates $u^\varepsilon$ of order $\varepsilon$.

\begin{remark} \label{R:Po_Exam}
  The above main results are obtained under Assumption \ref{A:Poten}, which is satisfied by the double-well polynomial potentials
 \begin{align*}
  F(z) = \frac{1}{4}(z^2-1)^2, \quad F(z) = \frac{1}{4}z^4-\frac{1}{2}z^2.
  \end{align*}
  These potentials are often used as an approximation of the logarithmic potential
  \begin{align*}
    F_{\textrm{log}}(z) = \frac{\theta}{2}\{(1-z)\log(1-z)+(1+z)\log(1+z)\}+\frac{\theta_c}{2}(1-z^2), \quad z\in(-1,1).
  \end{align*}
  Here, $\theta$ and $\theta_c$ are positive constants that represent the absolute temperature of a mixture and a critical temperature of phase separation, respectively.
  Since $F_{\log}$ is a thermodynamically relevant potential \cite{CahHil58}, it is meaningful to extend the main results of this paper to the case of $F_{\log}$.
  This is the subject of a future work.
\end{remark}

\subsection{Outline of proof} \label{SS:Int_Pro}
We give the proofs of Theorems \ref{T:TFL_Weak}, \ref{T:DiE_Sur}, and \ref{T:DiE_CTD} in Section \ref{S:TFL}.
Here, we explain the outline and main idea of the proofs.

To prove the main results, we proceed basically in the same way as in the case of second order equations \cite{Miu17,Miu20_03,Miu24_GL}:
\begin{enumerate}
  \item For the weak solution $(u^\varepsilon,w^\varepsilon)$ to the thin-domain problem \eqref{E:CH_CTD}, let
  \begin{align*}
    v^\varepsilon := \mathcal{M}_\varepsilon u^\varepsilon, \quad \mu^\varepsilon := \mathcal{M}_\varepsilon w^\varepsilon \quad\text{on}\quad \Gamma\times(0,\infty)
  \end{align*}
  be the weighted average.
  We first derive a weak form satisfied by $v^\varepsilon$ and $\mu^\varepsilon$ from a weak form of \eqref{E:CH_CTD} by using the change of variables formula \eqref{E:Int_CoV}.
  \item Next, we differentiate in time the energy functional
  \begin{align} \label{E:Int_wGLE}
    E_g(v^\varepsilon(t)) = \int_\Gamma g\left(\frac{|\nabla_\Gamma v^\varepsilon(t)|^2}{2}+F(v^\varepsilon(t))\right)\,d\mathcal{H}^2
  \end{align}
  and use relations of $v^\varepsilon$ and $\mu^\varepsilon$ including their weak form to derive energy estimates for $v^\varepsilon$ and $\mu^\varepsilon$ which are uniform with respect to $\varepsilon$.
  We also show some uniform estimates for $\partial_tv^\varepsilon$ and $F(v^\varepsilon)$ by using estimates for $v^\varepsilon$ and $\mu^\varepsilon$.
  \item By the uniform estimates derived above, we have the convergence of (subsequences of) $v^\varepsilon$ and $\mu^\varepsilon$ on $\Gamma$ as $\varepsilon\to0$ in an appropriate sense.
  Letting $\varepsilon\to0$ in the weak form of $v^\varepsilon$ and $\mu^\varepsilon$, we find that the pair $(v,\mu)$ of the limit functions is indeed a weak solution to the limit problem \eqref{E:CH_Lim}.
  This establishes Theorem \ref{T:TFL_Weak}, when combined with uniqueness of a weak solution to \eqref{E:CH_Lim} (see Proposition \ref{P:CLi_Uni}).
  \item We subtract the weak form of $v$ and $\mu$ from that of $v^\varepsilon$ and $\mu^\varepsilon$, and then apply an energy method to obtain the difference estimate \eqref{E:DiE_Sur} on $\Gamma$ in Theorem \ref{T:DiE_Sur}.
  \item Combining \eqref{E:DiE_Sur} and estimates for the difference of $u^\varepsilon$ and $v^\varepsilon$ (see Lemmas \ref{L:ADif_Lp} and \ref{L:ATG_L2D}), we get the difference estimate \eqref{E:DiE_CTD} on $\Omega_\varepsilon$ in Theorem \ref{T:DiE_CTD}.
  The estimate \eqref{E:DiE_ND} just follows from an estimate for $\partial_\nu u^\varepsilon$ in $L^2(\Omega_\varepsilon)$ given in Lemma \ref{L:ND_CTD} and an energy estimate for $u^\varepsilon$.
\end{enumerate}
The main effort is devoted to the first two steps, especially to analysis of errors appearing in the process of the average.
In the step (i), we take a test function on $\Gamma$, substitute its constant extension in the normal direction of $\Gamma$ for the weak form of \eqref{E:CH_CTD}, and rewrite the weak form by using the change of variables formula \eqref{E:Int_CoV}.
Then, we get the weak form of $v^\varepsilon$ and $\mu^\varepsilon$ which is close to the weak form of the limit problem \eqref{E:CH_Lim}, but it involves error terms since the weighted average does not commute with the gradient and $F'$.
In order to derive uniform energy estimates for $v^\varepsilon$ and $\mu^\varepsilon$ and show that the weak form of $v^\varepsilon$ and $\mu^\varepsilon$ actually converges to the weak form of \eqref{E:CH_Lim} as $\varepsilon\to0$, we need to estimate the error terms explicitly in terms of $\varepsilon$ and show that they are sufficiently small as $\varepsilon\to0$.
To this end, we carefully analyze the weighted average and prove several useful estimates with constants explicitly depending on $\varepsilon$ in Section \ref{S:Ave}, and combine them with energy estimates for the weak solution $(u^\varepsilon,w^\varepsilon)$ to the thin-domain problem \eqref{E:CH_CTD} given in Section \ref{S:CHT_Weak}.

In the case of the second order equations \cite{Miu17,Miu20_03,Miu24_GL}, it was enough to take the average of a weak form and apply the above idea in order to derive energy estimates for an averaged solution, since we could get them just by testing the averaged solution itself to the averaged weak form.
In the step (ii), however, we cannot derive energy estimates for $v^\varepsilon$ and $\mu^\varepsilon$ only by using their weak form.
Indeed, for the energy functional \eqref{E:Int_wGLE}, we have
\begin{align*}
  \frac{d}{dt}E_g(v^\varepsilon) = \langle\partial_tv^\varepsilon,g\tilde{\mu}^\varepsilon\rangle_{H^1(\Gamma)}, \quad \tilde{\mu}^\varepsilon := -A_gv^\varepsilon+F'(v^\varepsilon),
\end{align*}
but the integral equality for $\partial_tv^\varepsilon$ obtained in the step (i) is
\begin{align} \label{E:Int_WFve}
  \int_0^T\langle\partial_tv^\varepsilon,g\eta\rangle_{H^1(\Gamma)}\,dt+\int_0^T(g\nabla_\Gamma\mu^\varepsilon,\nabla_\Gamma\eta)_{L^2(\Gamma)}\,dt = R_\varepsilon(\eta;T)
\end{align}
for $T>0$ and a test function $\eta$, where $R_\varepsilon(\eta;T)$ is an error term.
Thus, to derive an energy estimate from the above relations, we need to further relate $\tilde{\mu}^\varepsilon$ with $\mu^\varepsilon$ in a strong form.
However, such an relation cannot be obtained from the weak form of $v^\varepsilon$ and $\mu^\varepsilon$ since it involves error terms.
Hence, in addition to the average of the weak form, we also take the weighted average of the strong form
\begin{align*}
  w^\varepsilon = -\Delta u^\varepsilon+F'(u^\varepsilon) \quad\text{in}\quad \Omega_\varepsilon\times(0,\infty)
\end{align*}
to derive the relation of $v^\varepsilon$ and $\mu^\varepsilon$ in the strong form
\begin{align*}
  \mu^\varepsilon = -A_gv^\varepsilon+F'(v^\varepsilon)+\zeta^\varepsilon = \tilde{\mu}^\varepsilon+\zeta^\varepsilon \quad\text{on}\quad \Gamma\times(0,\infty),
\end{align*}
and show that the error term $\zeta^\varepsilon$ and $\nabla_\Gamma\zeta^\varepsilon$ are sufficiently small as $\varepsilon\to0$ in the $L^2$ space (note that $\nabla_\Gamma\zeta^\varepsilon$ appears when we use \eqref{E:Int_WFve}).
This requires the most effort.

In fact, the error term is defined as $\zeta^\varepsilon:=-\zeta_\Delta^\varepsilon+\zeta_{F'}^\varepsilon$ with
\begin{align*}
   \zeta_\Delta^\varepsilon &:= \mathcal{M}_\varepsilon(\Delta u^\varepsilon)-A_gv^\varepsilon = \mathcal{M}_\varepsilon(\Delta u^\varepsilon)-A_g\mathcal{M}_\varepsilon u^\varepsilon, \\
  \zeta_{F'}^\varepsilon &:= \mathcal{M}_\varepsilon(F'(u^\varepsilon))-F'(v^\varepsilon) = \mathcal{M}_\varepsilon(F'(u^\varepsilon))-F'(\mathcal{M}_\varepsilon u^\varepsilon),
\end{align*}
so we estimate them separately.
To estimate $\zeta_\Delta^\varepsilon$, we rewrite it in a more convenient form by calculating $A_g\mathcal{M}_\varepsilon u^\varepsilon$.
However, it is too hard to compute the second order derivatives of $\mathcal{M}_\varepsilon u^\varepsilon$, since its tangential gradient already has a complicated expression
\begin{align} \label{E:Int_AvGr}
  \nabla_\Gamma\mathcal{M}_\varepsilon u^\varepsilon = \mathcal{M}_\varepsilon(\mathbf{B}\nabla u^\varepsilon)+\mathcal{M}_\varepsilon((\partial_\nu u^\varepsilon+u^\varepsilon f_J)\mathbf{b}_\varepsilon)+\mathcal{M}_\varepsilon(u^\varepsilon\mathbf{b}_J) \quad\text{on}\quad \Gamma,
\end{align}
where $\mathbf{B}$, $f_J$, $\mathbf{b}_\varepsilon$, and $\mathbf{b}_J$ are functions on $\Omega_\varepsilon$ involving $g_0$, $g_1$, and the curvatures of $\Gamma$ (see Lemma \ref{L:Ave_TGr} for details).
Hence, instead, we consider the weak form
\begin{align*}
  \int_\Gamma g(A_g\mathcal{M}_\varepsilon u^\varepsilon)\eta\,d\mathcal{H}^2 = -\int_\Gamma g\nabla_\Gamma\mathcal{M}_\varepsilon u^\varepsilon\cdot\nabla_\Gamma\eta\,d\mathcal{H}^2
\end{align*}
for a test function $\eta$ on $\Gamma$, and rewrite the right-hand side in terms of $\mathcal{M}_\varepsilon(\Delta u^\varepsilon)$ and error terms by using \eqref{E:Int_AvGr}, the equality
\begin{align*}
  \varepsilon\int_\Gamma g(y)\mathcal{M}_\varepsilon\varphi(y)\zeta(y)\,d\mathcal{H}^2(y) = \int_{\Omega_\varepsilon}\varphi(x)\bar{\zeta}(x)\,dx
\end{align*}
for functions $\varphi$ on $\Omega_\varepsilon$ and $\zeta$ on $\Gamma$, and integration by parts.
This approach enables us to avoid computing the second order derivatives of $\mathcal{M}_\varepsilon u^\varepsilon$.
We also note that, when we estimate the $H^1(\Gamma)$-norm of $\zeta_\Delta^\varepsilon$, we need to assume that the third-order derivatives of $g_0$, $g_1$, and the curvatures of $\Gamma$ are bounded on $\Gamma$.
Thus, we require the $C^5$-regularity of $\Gamma$ and the $C^3$-regularity of $g_0$ and $g_1$.
For details, we refer to Lemma \ref{L:Ave_Lap}.

Another task is to estimate $\zeta_{F'}^\varepsilon$.
We split $\zeta_{F'}^\varepsilon=K_1+K_2$ into
\begin{align*}
  K_1(y) &:= \frac{1}{\varepsilon g(y)}\int_{\varepsilon g_1(y)}^{\varepsilon g_2(y)}F'\Bigl(u^\varepsilon(y+r\bm{\nu}(y))\Bigr)\{J(y,r)-1\}\,dr, \\
  K_2(y) &:= \frac{1}{\varepsilon g(y)}\int_{\varepsilon g_1(y)}^{\varepsilon g_2(y)}\Bigl\{F'\Bigl(u^\varepsilon(y+r\bm{\nu}(y))\Bigr)-F'\Bigl(\mathcal{M}_\varepsilon u^\varepsilon(y)\Bigr)\Bigr\}\,dr
\end{align*}
for $y\in\Gamma$.
To $K_1$, we use $|J-1|\leq c\varepsilon$ and the bound \eqref{E:Int_Fzero} of $F'$.
Also, we use the mean value theorem for $F'$ and the bound \eqref{E:Int_F2nd} of $F''$ to $K_2$, and then apply the inequality
\begin{align*}
  |u^\varepsilon(y+r\bm{\nu}(y))-\mathcal{M}_\varepsilon u^\varepsilon(y)| \leq c\varepsilon[\mathcal{M}_\varepsilon(|u^\varepsilon|+|\partial_\nu u^\varepsilon|)](y)
\end{align*}
given in Lemma \ref{L:ADif_Lp}.
After that, we take the $L^2(\Gamma)$-norm, use H\"{o}lder's inequality and the estimate for the weighted average (see Lemma \ref{L:Ave_Lp})
\begin{align*}
  \|\mathcal{M}_\varepsilon\varphi\|_{L^p(\Gamma)} \leq c\varepsilon^{-1/p}\|\varphi\|_{L^p(\Omega_\varepsilon)}, \quad \varphi\in L^p(\Omega_\varepsilon), \quad p\in[1,\infty]
\end{align*}
and apply the Sobolev inequality (see Lemma \ref{L:Sob_CTD})
\begin{align*}
  \varepsilon^{-1/p}\|\varphi\|_{L^p(\Omega_\varepsilon)} \leq c\varepsilon^{-1/2}\|\varphi\|_{H^1(\Omega_\varepsilon)}, \quad \varphi\in H^1(\Omega_\varepsilon), \quad p\in[2,6]
\end{align*}
to get a good estimate for $\zeta_{F'}^\varepsilon$, which gives the smallness of $\zeta_{F'}^\varepsilon$ when combined with energy estimates for $u^\varepsilon$.
We also estimate $\nabla_\Gamma\zeta_{F'}^\varepsilon$ in a similar way, but calculations become more involved.
In particular, we make use of the interpolation inequality (see Lemma \ref{L:Ipl_CTD})
\begin{align*}
  \varepsilon^{-1/p}\|\varphi\|_{L^p(\Omega_\varepsilon)} \leq c\varepsilon^{-1/2}\|\varphi\|_{L^2(\Omega_\varepsilon)}^{1-\sigma}\|\varphi\|_{H^2(\Omega_\varepsilon)}^\sigma, \quad \sigma = \frac{3}{4}-\frac{3}{2p}
\end{align*}
for $\varphi\in H^2(\Omega_\varepsilon)$ and $p\in[2,\infty]$ besides the Sobolev inequality in order to get an estimate for $\nabla_\Gamma\zeta_{F'}^\varepsilon$ which can be properly combined with energy estimates for $u^\varepsilon$.
For details, we refer to Section \ref{SS:Ave_NL}.

Let us also give a remark on constants appearing in inequalities on $\Omega_\varepsilon$.
In the study of the thin-film limit problems, we need to clarify how the constants depend on $\varepsilon$ like the above Sobolev and interpolation inequalities.
As one may easily imagine, this requires long and careful discussions, especially when we consider curved thin domains around surfaces.
In particular, when we use energy estimates for $u^\varepsilon$ in the analysis of error terms, we need to combine them with the uniform elliptic regularity estimate (see Lemma \ref{L:UER_CTD})
\begin{align*}
  \|u^\varepsilon\|_{H^{2+k}(\Omega_\varepsilon)} \leq c\Bigl(\|\Delta u^\varepsilon\|_{H^k(\Omega_\varepsilon)}+\|u^\varepsilon\|_{L^2(\Omega_\varepsilon)}\Bigr), \quad k=0,1.
\end{align*}
This is highly nontrivial since we need to confirm that the constant $c$ is independent of $\varepsilon$, and it is proved by a careful re-examination of a localization argument and the method of difference quotient.
We refer to Section \ref{S:Uni_ER} for details.

\subsection{Literature overview} \label{SS:Int_Lit}
In the study of partial differential equations (PDEs) in thin domains, a main subject is to find limit problems as the thickness of thin domains tends to zero and to compare the thin-domain and limit problems in terms of the thickness.
Such a thin-film limit problem is important in view of dimension reduction as well as of  modelling of various phenomena in thin sets like tubes and membranes and in lower dimensional sets like curves and surfaces.

PDEs in thin domains have been studied extensively in the case of second order equations.
Since the works by Hale and Raugel \cite{HalRau92_DH,HalRau92_RD}, many authors have studied second order PDEs in flat thin domains around lower dimensional domains.
Also, some works deal with a thin L-shaped domain \cite{HalRau95}, flat thin domains with holes \cite{PriRyb01} and moving boundaries \cite{PerSil13}, and thin tubes and networks \cite{Kos00,Kos02,Yan90}.
Curved thin domains around surfaces and manifolds appear in the shell theory \cite{Cia00} and in the study of reaction-diffusion equations \cite{PrRiRy02,PriRyb03_Cur}, the Navier--Stokes equations \cite{TemZia97,Miu20_03,Miu21_02,Miu22_01,Miu24_SNS}, the Ginzburg--Landau heat flow \cite{Miu24_GL}, and the eigenvalue problem for the Laplacian \cite{Sch96,Kre14,JimKur16,Yac18}.
Moreover, diffusion equations in moving thin domains around moving surfaces were studied in \cite{EllSti09,ElStStWe11,Miu17,MiGiLi18,Miu23}.
We also refer to \cite{Rau95} for other examples of thin domains.

Compared to the second order case, the thin-domain problem of higher order equations has been less studied.
There are a few works on the biharmonic equation in a thin rectangular \cite{ArFeLa17}, a thin T-shaped domain \cite{GaPaPi16}, and a curved thin domain around a closed curve \cite{FerPro23}.
Also, in \cite{Abo96}, the Cahn--Hilliard equation in a thin annulus around a circle was analyzed.
These works deal with thin domains in $\mathbb{R}^2$.
A higher dimensional case was recently studied in \cite{CaJaWo24}, where the authors analyzed the Stokes--Cahn--Hilliard equations in a flat thin layer in $\mathbb{R}^2$ or $\mathbb{R}^3$.
In this paper, we deal with a curved thin domain in $\mathbb{R}^3$ around a general closed surface.
This paper gives the first result on a rigorous thin-film limit of nonlinear fourth order equations in general curved thin domains.

Let us also mention the literature on the limit problem \eqref{E:CH_Lim}.
As mentioned in Remark \ref{R:TFL_LP}, the limit problem \eqref{E:CH_Lim} is the surface Cahn--Hilliard equation.
It is used to describe dealloying of a binary alloy \cite{EAKDS01,EilEll08} and phase separation in cell membranes \cite{MPKHWJ12,MMRH13,GKRR16,YuQuMaOl19,ZhWaQuOlMa21}.
In view of applications, there is a growing interest in the analysis of the surface Cahn--Hilliard equation like other surface PDEs (see e.g. \cite{DziEll13_AN}).
We refer to \cite{RatVoi06,DuJuTi11,EllRan15,OCoSti16,Rat16,AZVR19,ZTLHMS19,YuQuOl20,CaeEll21,CaElGrPo23} and the references cited therein for the mathematical and numerical studies.
Here, we note that the surface Cahn--Hilliard equation is usually introduced as a surface version of the planar equation, not as a thin-film limit of the bulk equation.
Our results verify the characterization of the surface Cahn--Hilliard equation as a thin-film limit and justify the use of the surface Cahn--Hilliard equation as a model of phase separation in curved thin layers like cell membranes.

Lastly, we give comments on the directions of further studies.
In the case of flat domains, several authors analyzed the Cahn--Hilliard equation with degenerate mobility \cite{Yin92,EllGar96,DaiDu16} and logarithmic potential \cite{BloEll91,EllLuk91}.
Also, the Cahn--Hilliard equation on moving surfaces was introduced and studied in the recent works \cite{ZTLHMS19,YuQuOl20,CaeEll21,CaElGrPo23}.
It is the subject of future works to extend our results to these cases.
We also note that the surface Cahn--Hilliard equation appears as (a part of) the dynamic boundary conditions proposed in \cite{GoMiSc11,LiuWu19}.
It is an interesting problem to derive the Cahn--Hilliard equation with dynamic boundary conditions as a thin-film limit of a two-phase Cahn--Hilliard problem posed on a bulk domain and the thin tubular neighborhood of the boundary.
In the recent work \cite{GiLaRy25}, an abstract framework for singular limits of gradient flows on Hilbert spaces varying with some parameters was developed and applied to the parabolic $p$-Laplace equation in a flat thin domain.
Since the thin-domain and limit problems \eqref{E:CH_CTD} and \eqref{E:CH_Lim} have the gradient flow structures, it may be possible to re-examine our results in the framework of \cite{GiLaRy25}.
This is also a future topic.

\subsection{Organization of the paper} \label{SS:Int_Org}
The rest of this paper is organized as follows.
Section \ref{S:Pre} provides notations and basic results on a closed surface and a curved thin domain.
In Section \ref{S:CHT_Weak}, we give the definition of a weak solution to the thin-domain problem \eqref{E:CH_CTD} and show basic properties and estimates of a weak solution.
Section \ref{S:Ave} is devoted to analysis of the weighted average operator in the thin direction.
We study the thin-film limit problem of \eqref{E:CH_CTD} and establish the main results in Section \ref{S:TFL}.
In Section \ref{S:Uni_ER}, we prove uniform elliptic regularity estimates on the curved thin domain with constants independent of $\varepsilon$.
Section \ref{S:Pf_Aux} gives the proofs of some auxiliary lemmas.
In Section \ref{S:Galer}, we briefly explain the outline of construction of a weak solution to \eqref{E:CH_CTD} by the Galerkin method.

\section{Preliminaries} \label{S:Pre}
We fix notations and give basic results on surfaces and thin domains.

\subsection{Notations used throughout the paper} \label{SS:Pr_Nota}
In what follows, we always write $c$ for a general positive constant independent of $\varepsilon$ and $t$.

We fix a coordinate system of $\mathbb{R}^3$.
Under the fixed coordinate system, let $x_i$ be the $i$-th component of $x\in\mathbb{R}^3$ for $i=1,2,3$.
We write a vector $\mathbf{a}\in\mathbb{R}^3$ and a matrix $\mathbf{A}\in\mathbb{R}^{3\times 3}$ as
\begin{align*}
  \mathbf{a} =
  \begin{pmatrix}
    a_1 \\ a_2 \\ a_3
  \end{pmatrix}
  = (a_1,a_2,a_3)^{\mathrm{T}}, \quad \mathbf{A} = (A_{ij})_{i,j} =
  \begin{pmatrix}
    A_{11} & A_{12} & A_{13} \\
    A_{21} & A_{22} & A_{23} \\
    A_{31} & A_{32} & A_{33}
  \end{pmatrix}.
\end{align*}
Let $|\mathbf{a}|$, $\mathbf{a}\cdot\mathbf{b}$, and $\mathbf{a}\otimes\mathbf{b}:=(a_ib_j)_{i,j}$ be the norm, the inner product, and the tensor product of $\mathbf{a},\mathbf{b}\in\mathbb{R}^3$, respectively.
We write $\mathbf{I}_3$ and $\mathbf{A}^{\mathrm{T}}$ for the $3\times 3$ identity matrix and the transpose of $\mathbf{A}\in\mathbb{R}^{3\times3}$, respectively.
Also, we denote by
\begin{align*}
  \mathbf{A}:\mathbf{B} := \mathrm{tr}[\mathbf{A}^{\mathrm{T}}\mathbf{B}], \quad |\mathbf{A}| := \sqrt{\mathbf{A}:\mathbf{A}}, \quad \mathbf{A},\mathbf{B}\in\mathbb{R}^{3\times3}
\end{align*}
the inner product and the Frobenius norm of matrices, respectively.

For a function $\varphi$ and a vector field $\mathbf{u}$ on an open set in $\mathbb{R}^3$, we write
\begin{align*}
  \nabla\varphi := (\partial_1\varphi,\partial_2\varphi,\partial_3\varphi)^{\mathrm{T}}, \quad \nabla^2\varphi := (\partial_i\partial_j\varphi)_{i,j}, \quad \nabla\mathbf{u} := (\partial_i\mathrm{u}_j)_{i,j},
\end{align*}
where $\mathbf{u}=(\mathrm{u}_1,\mathrm{u}_2,\mathrm{u}_3)^T$.
We also use the notations
\begin{align*}
  |\nabla^0\varphi| := |\varphi|, \quad |\nabla^k\varphi| := (\textstyle\sum_{i_1,\dots,i_k=1}^3|\partial_{i_1}\cdots\partial_{i_k}\varphi|^2)^{1/2}, \quad k\geq1
\end{align*}
including the case where $\varphi$ is vector-valued and matrix-valued.

When $\mathcal{X}$ is a Banach space, we write $[\mathcal{X}]'$ for the dual space of $\mathcal{X}$.
Also, we denote by $\langle\cdot,\cdot\rangle_{\mathcal{X}}$ the duality product between $[\mathcal{X}]'$ and $\mathcal{X}$.

\subsection{Closed surface} \label{SS:Pr_Sur}
Let $\Gamma$ be a closed (i.e. compact and without boundary), connected, and oriented surface in $\mathbb{R}^3$.
We assume that $\Gamma$ is of class $C^5$, and it is the boundary of a bounded domain $\Omega$ in $\mathbb{R}^3$.
Moreover, we denote by $\bm{\nu}$ the unit outward normal vector field of $\Gamma$ which points from $\Omega$ into $\mathbb{R}^3\setminus\Omega$.
Then, $\bm{\nu}$ is of class $C^4$ on $\Gamma$.

Let $d$ be the signed distance function from $\Gamma$ increasing in the direction of $\bm{\nu}$.
Also, let $\kappa_1$ and $\kappa_2$ be the principal curvatures of $\Gamma$.
Then, $\kappa_1$ and $\kappa_2$ are continuous and thus bounded on $\Gamma$.
Hence, $\Gamma$ has a tubular neighborhood
\begin{align*}
  \mathcal{N}_\delta := \{x\in\mathbb{R}^3 \mid -\delta<d(x)<\delta\}
\end{align*}
of radius $\delta>0$ such that for each $x\in\Gamma$ there exists a unique $\pi(x)\in\Gamma$ satisfying
\begin{align} \label{E:Fermi}
  x = \pi(x)+d(x)\bm{\nu}(\pi(x)), \quad \nabla d(x) = \bm{\nu}(\pi(x)),
\end{align}
and $d$ and $\pi$ are of class $C^5$ and $C^4$ on $\overline{\mathcal{N}_\delta}$, respectively (see e.g. \cite[Section 14.6]{GilTru01}).
Taking $\delta>0$ sufficiently small, we may also assume that
\begin{align} \label{E:Cur_Bd}
  c^{-1} \leq 1-r\kappa_k(y) \leq c \quad\text{for all}\quad y\in\Gamma, \, r\in[-\delta,\delta], \, k=1,2.
\end{align}
Let $\mathbf{P}:=\mathbf{I}_3-\bm{\nu}\otimes\bm{\nu}$ be the orthogonal projection onto the tangent plane of $\Gamma$, which is of class $C^4$ on $\Gamma$.
We define the tangential gradient of a function $\eta$ on $\Gamma$ by
\begin{align*}
  \nabla_\Gamma\eta = (\underline{D}_1\eta,\underline{D}_2\eta,\underline{D}_3\eta)^{\mathrm{T}} := \mathbf{P}\nabla\tilde{\eta} \quad\text{on}\quad \Gamma.
\end{align*}
Here, $\tilde{\eta}$ is an extension of $\eta$ to $\mathcal{N}_\delta$, and the value of $\nabla_\Gamma\eta$ is independent of the choice of $\tilde{\eta}$.
It immediately follows from the above definition that
\begin{align} \label{E:TGr_Tan}
  \mathbf{P}\nabla_\Gamma\eta = \nabla_\Gamma\eta, \quad \nabla_\Gamma\eta\cdot\bm{\nu} = 0 \quad\text{on}\quad \Gamma.
\end{align}
Let $\bar{\eta}:=\eta\circ\pi$ be the constant extension of $\eta$ in the normal direction of $\Gamma$.
Then, since $\nabla\pi(y)=\mathbf{P}(y)$ for $y\in\Gamma$ by \eqref{E:Fermi} and $d(y)=0$, it follows that
\begin{align} \label{E:CEGr_Sur}
  \nabla\bar{\eta}(y) = \nabla_\Gamma\eta(y), \quad \partial_i\bar{\eta}(y) = \underline{D}_i\eta(y) \quad\text{for all}\quad y\in\Gamma.
\end{align}
In what follows, we always write $\bar{\eta}$ for the constant extension of a function $\eta$ on $\Gamma$.

Let $\mathbf{v}=(\mathrm{v}_1,\mathrm{v}_2,\mathrm{v}_3)^{\mathrm{T}}$ be a (not necessarily tangential) vector field on $\Gamma$.
We define the tangential gradient matrix and the surface divergence of $\mathbf{v}$ by
\begin{align*}
  \nabla_\Gamma\mathbf{v} := (\underline{D}_i\mathrm{v}_j)_{i,j}, \quad \mathrm{div}_\Gamma\mathbf{v} := \mathrm{tr}[\nabla_\Gamma\mathbf{v}] = \textstyle\sum_{i=1}^3\underline{D}_i\mathrm{v}_i \quad\text{on}\quad \Gamma.
\end{align*}
Note that the indices of $\nabla_\Gamma\mathbf{v}$ are reversed in some literature.
Under our notation, we have $\nabla_\Gamma\mathbf{v}=\mathbf{P}\nabla\tilde{\mathbf{v}}$ on $\Gamma$ for any extension $\tilde{\mathbf{v}}$ of $\mathbf{v}$ to $\mathcal{N}_\delta$.
When $\mathbf{v}=\nabla_\Gamma\eta$, we set
\begin{align*}
  \nabla_\Gamma^2\eta := \nabla_\Gamma(\nabla_\Gamma\eta) = (\underline{D}_i\underline{D}_j\eta)_{i,j}, \quad \Delta_\Gamma\eta := \mathrm{div}_\Gamma(\nabla_\Gamma\eta) \quad\text{on}\quad \Gamma
\end{align*}
and call them the tangential Hessian and the Laplace--Beltrami operator on $\Gamma$, respectively.
Moreover, we use the notations
\begin{align*}
  |\nabla_\Gamma^0\eta| := |\eta|, \quad |\nabla_\Gamma^k\eta| := (\textstyle\sum_{i_1,\dots,i_k=1}^3|\underline{D}_{i_1}\cdots\underline{D}_{i_k}\eta|^2)^{1/2}, \quad k\geq1
\end{align*}
including the case where $\eta$ is vector-valued and matrix-valued.

We define the Weingarten map $\mathbf{W}$ and the mean curvature $H$ of $\Gamma$ by
\begin{align*}
  \mathbf{W} = (W_{ij})_{i,j} := -\nabla_\Gamma\bm{\nu}, \quad H := \mathrm{tr}[\mathbf{W}] = -\mathrm{div}_\Gamma\bm{\nu} \quad\text{on}\quad \Gamma.
\end{align*}
Also, let $K:=\kappa_1\kappa_2$ be the Gaussian curvature of $\Gamma$.
Then, $\mathbf{W}$, $H$, and $K$ are of class $C^3$ on $\Gamma$.
Moreover, $\mathbf{W}$ is symmetric and has the eigenvalue zero with eigenvector $\bm{\nu}$.
Indeed, by \eqref{E:Fermi}, \eqref{E:CEGr_Sur}, and $|\bm{\nu}|=1$ on $\Gamma$, we have
\begin{align*}
  \mathbf{W} = -\nabla^2d = \mathbf{W}^{\mathrm{T}}, \quad \mathbf{W}\bm{\nu} = -\nabla_\Gamma(|\bm{\nu}|^2/1) = 0 \quad\text{on}\quad \Gamma.
\end{align*}
From these relations and $\mathbf{P}=\mathbf{I}_3-\bm{\nu}\otimes\bm{\nu}$, it also follows that
\begin{align} \label{E:WP_PW}
  \mathbf{W}\mathbf{P} = \mathbf{P}\mathbf{W} = \mathbf{W} \quad\text{on}\quad \Gamma.
\end{align}
Moreover, $\kappa_1$ and $\kappa_2$ are the eigenvalues of $\mathbf{W}$ (see e.g. \cite{Lee18}).
Thus, taking an orthonormal basis of $\mathbb{R}^3$ consisting of eigenvectors of $\mathbf{W}$, we easily find that $H=\kappa_1+\kappa_2$ and
\begin{align*}
  \det[\mathbf{I}_3-r\mathbf{W}(y)] = \{1-r\kappa_1(y)\}\{1-r\kappa_2(y)\} \quad\text{for all}\quad y\in\Gamma, \, r\in[-\delta,\delta].
\end{align*}
Since the determinant does not vanish by \eqref{E:Cur_Bd}, the matrix
\begin{align*}
  \mathbf{I}_3-d(x)\overline{\mathbf{W}}(x) = \mathbf{I}_3-r\mathbf{W}(y)
\end{align*}
is invertible for all $x=y+r\bm{\nu}(y)\in\overline{\mathcal{N}_\delta}$ with $y=\pi(x)\in\Gamma$ and $r=d(x)\in[-\delta,\delta]$.

\begin{lemma} \label{L:IdW_Inv}
  We define
  \begin{align*}
    \mathbf{R}(x) := \{\mathbf{I}_3-d(x)\overline{\mathbf{W}}(x)\}^{-1}, \quad x\in\overline{\mathcal{N}_\delta}.
  \end{align*}
  Then, $\mathbf{R}$ is symmetric and of class $C^3$ on $\overline{\mathcal{N}_\delta}$.
  Moreover, $\mathbf{R}\overline{\mathbf{P}}=\overline{\mathbf{P}}\mathbf{R}$ in $\overline{\mathcal{N}_\delta}$.
\end{lemma}

\begin{proof}
  We see that $\mathbf{R}$ is symmetric, since $\mathbf{W}$ is so.
  Also, the regularity of $\mathbf{R}$ follows from that of $d$ and $\overline{\mathbf{W}}=\mathbf{W}\circ\pi$.
  We have $\mathbf{R}\overline{\mathbf{P}}=\overline{\mathbf{P}}\mathbf{R}$ by \eqref{E:WP_PW}.
\end{proof}

\begin{lemma} \label{L:CEGr_NB}
  Let $\eta$ be a function on $\Gamma$ and $\bar{\eta}$ be its constant extension.
  Then,
  \begin{align} \label{E:CEGr_NB}
    \nabla\bar{\eta}(x) = \mathbf{R}(x)\overline{\nabla_\Gamma\eta}(x), \quad \bar{\bm{\nu}}(x)\cdot\nabla\bar{\eta}(x) = 0 \quad\text{for all}\quad x\in\overline{\mathcal{N}_\delta}.
  \end{align}
  Moreover, for all $x\in\overline{\mathcal{N}_\delta}$ and $k=1,2,3$, we have
  \begin{align}
    |\nabla^k\bar{\eta}(x)| &\leq c\sum_{\ell=1}^k\left|\overline{\nabla_\Gamma^\ell\eta}(x)\right|, \label{E:CEGr_Bd} \\
    \left|\nabla\bar{\eta}(x)-\overline{\nabla_\Gamma\eta}(x)\right| &\leq c|d(x)|\left|\overline{\nabla_\Gamma\eta}(x)\right|. \label{E:CEGr_Df}
  \end{align}
\end{lemma}

\begin{proof}
  We refer to \cite[Lemma 2.2]{Miu23} for the proofs of \eqref{E:CEGr_NB} and \eqref{E:CEGr_Df}.
  Also, we can get \eqref{E:CEGr_Bd} by the repeated use of \eqref{E:CEGr_NB} and the regularity of $\mathbf{R}$ on $\overline{\mathcal{N}_\delta}$.
\end{proof}

For $\eta,\zeta\in C^1(\Gamma)$ and $i=1,2,3$, the integration by parts formula
\begin{align} \label{E:IbP_Sur}
  \int_\Gamma\zeta\underline{D}_i\eta\,d\mathcal{H}^2 = -\int_\Gamma\eta(\underline{D}_i\zeta+\zeta H\nu_i)\,d\mathcal{H}^2
\end{align}
holds (see e.g. \cite[Lemma 16.1]{GilTru01} and \cite[Theorem 2.10]{DziEll13_AN}).
Let $p\in[1,\infty]$ and $\eta\in L^p(\Gamma)$.
We say that $\eta$ has the $i$-th weak tangential derivative if there exists a function $\underline{D}_i\eta\in L^p(\Gamma)$ such that \eqref{E:IbP_Sur} holds for all $\zeta\in C^1(\Gamma)$.
Also, we define the Sobolev space
\begin{align*}
  W^{1,p}(\Gamma) := \{\eta\in L^p(\Gamma) \mid \underline{D}_1\eta,\underline{D}_2\eta,\underline{D}_3\eta\in L^p(\Gamma)\}
\end{align*}
and similarly the higher order spaces $W^{m,p}(\Gamma)$.
We write
\begin{align*}
  H^m(\Gamma) := W^{m,2}(\Gamma), \quad W^{0,p}(\Gamma) := L^p(\Gamma), \quad H^0(\Gamma) = W^{0,2}(\Gamma) = L^2(\Gamma)
\end{align*}
and define the norm on $W^{m,p}(\Gamma)$ and the inner product on $H^m(\Gamma)$ as usual.

Let $[H^1(\Gamma)]'$ be the dual space of $H^1(\Gamma)$.
Setting
\begin{align*}
  \langle f,\eta\rangle_{H^1(\Gamma)} := (f,\eta)_{L^2(\Gamma)}, \quad f\in L^2(\Gamma), \, \eta\in H^1(\Gamma),
\end{align*}
we have $L^2(\Gamma)\subset[H^1(\Gamma)]'$ and $\|f\|_{[H^1(\Gamma)]'}\leq\|f\|_{L^2(\Gamma)}$.

For $\eta\in H^2(\Gamma)$, we define
\begin{align*}
  A_g\eta := \frac{1}{g}\mathrm{div}_\Gamma(g\nabla_\Gamma\eta) \quad\text{on}\quad \Gamma.
\end{align*}
Then, for all $\eta\in H^2(\Gamma)$ and $\zeta\in H^1(\Gamma)$, we have
\begin{align} \label{E:IbP_Ag}
  (gA_g\eta,\zeta)_{L^2(\Gamma)} = -\int_\Gamma g\nabla_\Gamma\eta\cdot(\nabla_\Gamma\zeta+\zeta H\bm{\nu})\,d\mathcal{H}^2 = -(g\nabla_\Gamma\eta,\nabla_\Gamma\zeta)_{L^2(\Gamma)}
\end{align}
by integration by parts \eqref{E:IbP_Sur} and $\nabla_\Gamma\eta\cdot\bm{\nu}=0$ on $\Gamma$.

\subsection{Curved thin domain} \label{SS:Pr_CTD}
Let $g_0$ and $g_1$ be $C^3$ functions on $\Gamma$ such that
\begin{align} \label{E:G_Bdd}
  c^{-1} \leq g(y) := g_1(y)-g_0(y) \leq c \quad\text{for all}\quad y\in\Gamma
\end{align}
with some constant $c>0$.
Note that we do not make any assumptions on the signs of $g_0$ and $g_1$.
For a sufficiently small $\varepsilon>0$, we define the curved thin domain $\Omega_\varepsilon$ in $\mathbb{R}^3$ by \eqref{E:Def_CTD}.
Note that the boundary $\partial\Omega_\varepsilon$ is of the form
\begin{align*}
  \partial\Omega_\varepsilon = \Gamma_\varepsilon^0\cup\Gamma_\varepsilon^1, \quad \Gamma_\varepsilon^i := \{y+\varepsilon g_i(y)\bm{\nu}(y) \mid y\in\Gamma\}, \quad i=0,1,
\end{align*}
and it is of class $C^3$ by the regularity of $\Gamma$, $g_0$, and $g_1$.
Let $\bm{\nu}_\varepsilon$ be the unit outward normal vector field of $\partial\Omega_\varepsilon$.
We write $\partial_{\nu_\varepsilon}:=\bm{\nu}_\varepsilon\cdot\nabla$ for the outer normal derivative on $\partial\Omega_\varepsilon$.

Let $\delta$ be the radius of the tubular neighborhood $\mathcal{N}_\delta$ of $\Gamma$ given in Section \ref{SS:Pr_Sur}.
Since $g_0$ and $g_1$ are bounded on $\Gamma$, we can take an $\varepsilon_\delta\in(0,1)$ such that $\varepsilon_\delta|g_i|\leq\delta$ on $\Gamma$ for $i=0,1$.
Hence, $\overline{\Omega_\varepsilon}\subset\overline{\mathcal{N}_\delta}$ for $\varepsilon\in(0,\varepsilon_\delta]$.
In what follows, we always assume $0<\varepsilon\leq\varepsilon_\delta<1$.

For $y\in\Gamma$ and $r\in[-\delta,\delta]$, we define
\begin{align} \label{E:Def_J}
  J(y,r) := \det[\mathbf{I}_3-r\mathbf{W}(y)] =\{1-r\kappa_1(y)\}\{1-r\kappa_2(y)\}.
\end{align}
Using $H=\kappa_1+\kappa_2$ and $K=\kappa_1\kappa_2$, we can also write
\begin{align*}
  J(y,r) = 1-rH(y)+r^2K(y).
\end{align*}
Since $\mathbf{W}$ is of class $C^3$ on $\Gamma$ and \eqref{E:Cur_Bd} holds, we have
\begin{align} \label{E:J_Bdd}
  c^{-1} \leq J \leq c, \quad |\partial_r^\ell J| \leq c, \quad |\nabla_\Gamma^\ell J| \leq c \quad\text{on}\quad \Gamma\times[-\delta,\delta], \, \ell=1,2,3.
\end{align}
Here, $\nabla_\Gamma$ applies to the $y$-variable of $J$.
Also, since $g_0$ and $g_1$ are bounded on $\Gamma$,
\begin{align} \label{E:J_Diff}
  |J(y,r)-1| \leq c\varepsilon \quad \text{for all}\quad y\in\Gamma, \, r\in[\varepsilon g_0(y),\varepsilon g_1(y)].
\end{align}
The function $J$ is the Jacobian in the change of variables formula
\begin{align} \label{E:CoV_CTD}
  \int_{\Omega_\varepsilon}\varphi(x)\,dx = \int_{\Gamma}\int_{\varepsilon g_0(y)}^{\varepsilon g_1(y)}\varphi^\sharp(y,r)J(y,r)\,dr\,d\mathcal{H}^2(y)
\end{align}
for a function $\varphi$ on $\Omega_\varepsilon$ (see e.g. \cite[Section 14.6]{GilTru01}), where we set
\begin{align} \label{E:Pull_CTD}
  \varphi^\sharp(y,r) := \varphi(y+r\bm{\nu}(y,t)), \quad y\in\Gamma, \, r\in[\varepsilon g_0(y),\varepsilon g_1(y)].
\end{align}
For $p\in[1,\infty)$, it follows from \eqref{E:J_Bdd} and \eqref{E:CoV_CTD} that
\begin{align} \label{E:Lp_CTD}
  c^{-1}\|\varphi\|_{L^p(\Omega_\varepsilon)}^p \leq \int_{\Gamma}\int_{\varepsilon g_0(y)}^{\varepsilon g_1(y)}|\varphi^\sharp(y,r)|^p\,dr\,d\mathcal{H}^{n-1}(y) \leq c\|\varphi\|_{L^p(\Omega_\varepsilon)}^p.
\end{align}
Also, for a function $\eta$ on $\Gamma$ and its constant extension $\bar{\eta}$, we have
\begin{align} \label{E:Lp_CE}
  c^{-1}\varepsilon^{1/p}\|\eta\|_{L^p(\Gamma)} \leq \|\bar{\eta}\|_{L^p(\Omega_\varepsilon)} \leq c\varepsilon^{1/p}\|\eta\|_{L^p(\Gamma)}
\end{align}
by $\bar{\eta}^\sharp(y,r)=\eta(y)$, \eqref{E:G_Bdd}, and \eqref{E:J_Bdd}.

In what follows, we write $|\Omega_\varepsilon|$ for the volume of $\Omega_\varepsilon$.
Then,
\begin{align*}
  |\Omega_\varepsilon| = \int_{\Omega_\varepsilon}dx = \int_\Gamma\int_{\varepsilon g_0(y)}^{\varepsilon g_1(y)}J(y,r)\,dr\,d\mathcal{H}^2(y) \leq c\varepsilon
\end{align*}
by \eqref{E:G_Bdd}, \eqref{E:J_Bdd}, and \eqref{E:CoV_CTD}.

Let $[H^1(\Omega_\varepsilon)]'$ be the dual space of $H^1(\Omega_\varepsilon)$.
As in the case of $\Gamma$, we set and consider
\begin{align*}
  \langle f,\cdot\rangle_{H^1(\Omega_\varepsilon)} := (f,\cdot)_{L^2(\Omega_\varepsilon)}, \quad f\in L^2(\Omega_\varepsilon), \quad  L^2(\Omega_\varepsilon) \subset [H^1(\Omega_\varepsilon)]'.
\end{align*}

\subsection{Basic results} \label{SS:Pr_Bas}
Let us give some basic results used in the analysis of the thin-domain problem \eqref{E:CH_CTD} and the limit problem \eqref{E:CH_Lim}.
First, we consider the closed surface $\Gamma$.

\begin{lemma} \label{L:Poi_Sur}
  Let $v\in H^1(\Gamma)$ satisfy $\int_\Gamma gv\,d\mathcal{H}^2=0$.
  Then,
  \begin{align} \label{E:Poi_Sur}
    \|v\|_{L^2(\Gamma)} \leq c\|\nabla_\Gamma v\|_{L^2(\Gamma)}.
  \end{align}
\end{lemma}

As in the case $g\equiv 1$ (see e.g. \cite[Lemma 3.8]{Miu20_03}), Lemma \ref{L:Poi_Sur} is shown by a contradiction argument, the compact embedding $H^1(\Gamma)\hookrightarrow L^2(\Gamma)$, and $\int_\Gamma g\,d\mathcal{H}^2>0$ due to \eqref{E:G_Bdd}.

\begin{lemma} \label{L:Sob_Sur}
  We have
  \begin{alignat}{3}
    \|v\|_{L^2(\Gamma)} &\leq c\|v\|_{W^{1,1}(\Gamma)} &\quad &\text{for all} &\quad &v\in W^{1,1}(\Gamma), \label{E:SoSu_W11} \\
    \|v\|_{L^p(\Gamma)} &\leq c\|v\|_{H^1(\Gamma)} &\quad &\text{for all} &\quad &v\in H^1(\Gamma), \, p\in[2,\infty). \label{E:SoSu_H1}
  \end{alignat}
\end{lemma}

\begin{lemma} \label{L:Ipl_Sur}
  Let $v\in H^2(\Gamma)$ and $p\in[2,\infty]$.
  Then,
  \begin{align} \label{E:Ipl_Sur}
    \|v\|_{L^p(\Gamma)} \leq c\|v\|_{L^2(\Gamma)}^{1-\sigma}\|v\|_{H^2(\Gamma)}^\sigma, \quad \sigma = \frac{1}{2}-\frac{1}{p}.
  \end{align}
\end{lemma}

Lemmas \ref{L:Sob_Sur} and \ref{L:Ipl_Sur} are shown by a localization argument with a partition of unity on $\Gamma$ and application of the same inequalities on flat domains in $\mathbb{R}^2$ given in e.g. \cite[Theorems 4.12 and 5.8]{AdaFou03}.
We omit details here.

\begin{lemma} \label{L:InA_Sur}
  Let $f\in[H^1(\Gamma)]'$ satisfy $\langle f,g\rangle_{H^1(\Gamma)}=0$.
  Then, there exists a unique function $L_gf\in H^1(\Gamma)$ such that
  \begin{align} \label{E:InA_Sur}
    (g\nabla_\Gamma L_gf,\nabla_\Gamma\eta)_{L^2(\Gamma)} = \langle f,g\eta\rangle_{H^1(\Gamma)} \quad\text{for all}\quad \eta\in H^1(\Gamma)
  \end{align}
  and $\int_\Gamma gL_gf\,d\mathcal{H}^2=0$.
  Moreover, $L_g$ is linear and
  \begin{align} \label{E:IAS_Bdd}
    \|L_gf\|_{H^1(\Gamma)} \leq c\|f\|_{[H^1(\Gamma)]'}.
  \end{align}
\end{lemma}

The operator $L_g$ is the inverse of $-A_g$ by \eqref{E:IbP_Ag}.
Note that the condition $\langle f,g\rangle_{H^1(\Gamma)}=0$ is required for the validity of \eqref{E:InA_Sur} with $\eta\equiv1$.

\begin{proof}
  If $v\in H^1(\Gamma)$ and $\int_\Gamma gv\,d\mathcal{H}^2=0$, then we see by \eqref{E:G_Bdd} and \eqref{E:Poi_Sur} that
  \begin{align*}
    \|v\|_{H^1(\Gamma)} \leq c\|\nabla_\Gamma v\|_{L^2(\Gamma)} \leq c\|\sqrt{g}\,\nabla_\Gamma v\|_{L^2(\Gamma)}.
  \end{align*}
  Hence, by the Lax--Milgram theorem, there exists a unique solution $L_gf\in H^1(\Gamma)$ to \eqref{E:InA_Sur} satisfying $\int_\Gamma gL_gf\,d\mathcal{H}^2=0$.
  Since $\langle f,g\eta\rangle_{H^1(\Gamma)}$ is linear in $f$, we see that $L_g$ is linear.
  Also, setting $\eta=L_gf$ in \eqref{E:InA_Sur}, we have
  \begin{align*}
    \|L_gf\|_{H^1(\Gamma)}^2 &\leq c\|\sqrt{g}\nabla_\Gamma L_gf\|_{L^2(\Gamma)}^2 = c\langle f,gL_g f\rangle_{H^1(\Gamma)} \leq c\|f\|_{[H^1(\Gamma)]'}\|gL_gf\|_{H^1(\Gamma)}.
  \end{align*}
  By this inequality and $\|gL_gf\|_{H^1(\Gamma)}\leq c\|L_gf\|_{H^1(\Gamma)}$, we get \eqref{E:IAS_Bdd}.
\end{proof}

Next, we consider the curved thin domain $\Omega_\varepsilon$.
Recall that we write $c$ for a general positive constant independent of $\varepsilon$.

\begin{lemma} \label{L:Sob_CTD}
  Let $u\in H^1(\Omega_\varepsilon)$ and $p\in[2,6]$.
  Then,
  \begin{align} \label{E:Sob_CTD}
    \varepsilon^{-1/p}\|u\|_{L^p(\Omega_\varepsilon)} \leq c\varepsilon^{-1/2}\|u\|_{H^1(\Omega_\varepsilon)}.
  \end{align}
\end{lemma}

\begin{proof}
  We refer to \cite[Lemma 3.6]{Miu24_SNS} (note that $(p-2)/2p=1/2-1/p$).
\end{proof}

\begin{lemma} \label{L:Ipl_CTD}
  Let $u\in H^2(\Omega_\varepsilon)$ and $p\in[2,\infty]$.
  Then,
  \begin{align} \label{E:Ipl_CTD}
    \varepsilon^{-1/p}\|u\|_{L^p(\Omega_\varepsilon)} \leq c\varepsilon^{-1/2}\|u\|_{L^2(\Omega_\varepsilon)}^{1-\sigma}\|u\|_{H^2(\Omega_\varepsilon)}^\sigma, \quad \sigma = \frac{3}{4}-\frac{3}{2p}.
  \end{align}
\end{lemma}

We give the proof of Lemma \ref{L:Ipl_CTD} in Section \ref{S:Pf_Aux}.

\begin{lemma} \label{L:ND_CTD}
  Let $u\in H^2(\Omega_\varepsilon)$ satisfy $\partial_{\nu_\varepsilon}u=0$ on $\partial\Omega_\varepsilon$.
  Then,
  \begin{align} \label{E:ND_CTD}
    \|\partial_\nu u\|_{L^2(\Omega_\varepsilon)} \leq c\varepsilon\|u\|_{H^2(\Omega_\varepsilon)}.
  \end{align}
  Here, $\partial_\nu u:=\bar{\bm{\nu}}\cdot\nabla u$ is the derivative of $u$ in the normal direction of $\Gamma$.
\end{lemma}

\begin{proof}
  It is shown in \cite[Lemma 4.5]{Miu21_02} that
  \begin{align*}
    \|\mathbf{v}\cdot\bar{\bm{\nu}}\|_{L^2(\Omega_\varepsilon)} \leq c\varepsilon\|\mathbf{v}\|_{H^1(\Omega_\varepsilon)} \quad\text{for all}\quad \mathbf{v}\in H^1(\Omega_\varepsilon)^3, \quad  \mathbf{v}\cdot\bm{\nu}_\varepsilon = 0 \quad\text{on}\quad \partial\Omega_\varepsilon.
  \end{align*}
  Applying this inequality to $\mathbf{v}=\nabla u$, we obtain \eqref{E:ND_CTD}.
\end{proof}

\begin{lemma} \label{L:UER_CTD}
  Let $k=0,1$ and $u\in H^{2+k}(\Omega_\varepsilon)$ satisfy $\partial_{\nu_\varepsilon}u=0$ on $\partial\Omega_\varepsilon$.
  Then,
  \begin{align} \label{E:UER_CTD}
    \|u\|_{H^{2+k}(\Omega_\varepsilon)} \leq c\Bigl(\|\Delta u\|_{H^k(\Omega_\varepsilon)}+\|u\|_{L^2(\Omega_\varepsilon)}\Bigr).
  \end{align}
\end{lemma}

The proof of Lemma \ref{L:UER_CTD} requires long and careful discussions, since we need to confirm that the constant $c$ is independent of $\varepsilon$.
We give the proof in Section \ref{S:Uni_ER}.

\subsection{Bochner spaces} \label{SS:Pr_Boch}
Let $\mathcal{X}$ be a Banach space, $a<b$, and $u,v\in L^2(a,b;\mathcal{X})$.
If
\begin{align*}
  \int_a^b\zeta(t)v(t)\,dt = -\int_a^b\zeta'(t)u(t)\,dt \quad\text{in}\quad \mathcal{X}
\end{align*}
for all $\zeta\in C_c^1(a,b)$, then we write $v=\partial_tu$ and define
\begin{align*}
  H^1(a,b;\mathcal{X}) := \{u\in L^2(a,b;\mathcal{X}) \mid \partial_tu\in L^2(a,b;\mathcal{X})\} \subset C([a,b];\mathcal{X}).
\end{align*}
Let $S=\Gamma$ or $S=\Omega_\varepsilon$.
For $T>0$, we define
\begin{align} \label{E:Def_ET}
  \begin{aligned}
    \mathcal{E}_T(S) &:= \{u\in L^2(0,T;H^1(S)) \mid \partial_tu\in L^2(0,T;[H^1(S)]')\}, \\
    \|u\|_{\mathcal{E}_T(S)} &:= \|u\|_{L^2(0,T;H^1(S))}+\|\partial_tu\|_{L^2(0,T;[H^1(S)]')},
  \end{aligned}
\end{align}
which is a Hilbert space.
Let us give some basic results on Bochner spaces.

\begin{lemma} \label{L:ET_Equi}
  For $u\in L^2(0,T;H^1(S))$, the following conditions are equivalent:
  \begin{enumerate}
    \item $\partial_tu\in L^2(0,T;[H^1(S)]')$,
    \item there exists a function $v\in L^2(0;T;[H^1(S)]')$ such that
    \begin{align*}
      \int_0^T\langle v(t),\varphi(t)\rangle_{H^1(S)}\,dt = -\int_0^T(u(t),\partial_t\varphi(t))_{L^2(S)}\,dt
    \end{align*}
    for all $\varphi\in C_c^1(0,T;H^1(S))$,
    \item there exists a constant $c>0$ such that
    \begin{align*}
      \left|\int_0^T(u(t),\partial_t\varphi(t))_{L^2(S)}\,dt\right| \leq c\|\varphi\|_{L^2(0,T;H^1(S))}
    \end{align*}
    for all $\varphi\in C_c^1(0,T;H^1(S))$.
  \end{enumerate}
\end{lemma}

\begin{proof}
  If (ii) is valid, then we have (iii) with $c=\|v\|_{L^2(0,T;[H^1(S)]')}$ by H\"{o}lder's inequality.
  Conversely, (iii) yields (ii) since $C_c^1(0,T;H^1(S))$ is dense in $L^2(0,T;H^1(S))$ and
  \begin{align*}
    [L^2(0,T;H^1(S))]' = L^2(0,T;[H^1(S)]').
  \end{align*}
  Suppose that (ii) holds.
  Then, we set $\varphi=\zeta w$ for any $\zeta\in C_c^1(0,T)$ and $w\in H^1(S)$ to find that $\partial_tu=v$.
  Hence, (i) follows.
  Conversely, suppose that (i) is valid and let us show (ii).
  Let $\{\phi_k\}_{k=1}^\infty$ be an orthonormal basis of $H^1(S)$ (e.g. the normalized eigenfunctions of the Neumann Laplacian on $\Omega_\varepsilon$ or the Laplace--Beltrami operator on $\Gamma$).
  We set
  \begin{align*}
    \zeta_k(t) := (\varphi(t),\phi_k)_{H^1(S)}, \quad \varphi_K(t) := \sum_{k=1}^K\zeta_k(t)\phi_k, \quad k,K\in\mathbb{N}
  \end{align*}
  for $\varphi\in C_c^1(0,T;H^1(S))$ and $t\in[0,T]$.
  Then, since (i) holds and $\zeta_k\in C_c^1(0,T)$,
  \begin{align*}
    \int_0^T\zeta_k(t)v(t)\,dt = -\int_0^T\zeta_k'(t)u(t)\,dt \quad\text{in}\quad [H^1(S)]'
  \end{align*}
  with $v=\partial_tu$.
  We take $\langle\cdot,\phi_k\rangle_{H^1(S)}$ of both sides and sum up for $k=1,\dots,K$ to get
  \begin{align} \label{Pf_ETE:PhiK}
    \int_0^T\langle v(t),\varphi_K(t)\rangle_{H^1(S)}\,dt = -\int_0^T(u(t),\partial_t\varphi_K(t))_{L^2(S)}\,dt.
  \end{align}
  Moreover, since $\zeta_k'(t)=(\partial_t\varphi(t),\phi_k)_{H^1(S)}$ and $\varphi\in C_c^1(0,T;H^1(S))$,
  \begin{align*}
    \|\partial_t^\ell\varphi_K(t)-\partial_t^\ell\varphi(t)\|_{H^1(S)} \leq 2\|\partial_t^\ell\varphi(t)\|_{H^1(S)}, \quad \lim_{K\to\infty}\|\partial_t^\ell\varphi_K(t)-\partial_t^\ell\varphi(t)\|_{H^1(S)} = 0
  \end{align*}
  for $\ell=0,1$ and all $t\in[0,T]$.
  Thus, by the dominated convergence theorem,
  \begin{align*}
    \lim_{K\to\infty}\|\partial_t^\ell\varphi_K-\partial_t^\ell\varphi\|_{L^2(0,T;H^1(S))} = 0, \quad \ell=0,1,
  \end{align*}
  and we find that (ii) holds by letting $K\to\infty$ in \eqref{Pf_ETE:PhiK}.
\end{proof}

\begin{lemma} \label{L:ET_Emb}
  We have the continuous embeddings
  \begin{align*}
    \mathcal{E}_T(S) \hookrightarrow C([0,T];L^2(S)), \quad \mathcal{E}_T(S)\cap L^2(0,T;H^3(S)) \hookrightarrow C([0,T];H^1(S)).
  \end{align*}
\end{lemma}

\begin{proof}
  The result follows from \cite[Chapter 1, Theorem 3.1]{LioMag72}, since
  \begin{align*}
    [H^1(S),[H^1(S)]']_{1/2} = L^2(S), \quad [H^3(S),[H^1(S)]']_{1/2} = H^1(S)
  \end{align*}
  for the complex interpolation (see \cite[Chapter 1, Theorems 7.7 and 12.5]{LioMag72}).
\end{proof}

\begin{lemma}[Weak dominated convergence theorem] \label{L:WeDo}
  Suppose that
  \begin{itemize}
    \item $f,f_k\in L^2(0,T;L^2(S))$ for all $k\in\mathbb{N}$,
    \item $f_k\to f$ a.e. in $S\times(0,T)$ as $k\to\infty$, and
    \item there exists a constant $c>0$ such that $\|f_k\|_{L^2(0,T;L^2(S))}\leq c$ for all $k\in\mathbb{N}$.
  \end{itemize}
  Then, $f_k\to f$ weakly in $L^2(0,T;L^2(S))$ as $k\to\infty$.
\end{lemma}

\begin{proof}
  We refer to \cite[Lemma 8.3]{Rob01} for the proof.
\end{proof}

Now, let $S=\Gamma$.
We give some results on the time derivatives of integrals over $\Gamma$.

\begin{lemma} \label{L:DtL2_gSu}
  Let $v_1,v_2\in\mathcal{E}_T(\Gamma)$.
  Then, for a.a. $t\in(0,T)$, we have
  \begin{align} \label{E:DtL2_gSu}
    \frac{d}{dt}(gv_1(t),v_2(t))_{L^2(\Gamma)} = \langle\partial_tv_1(t),gv_2(t)\rangle_{H^1(\Gamma)}+\langle\partial_tv_2(t),gv_1(t)\rangle_{H^1(\Gamma)}.
  \end{align}
\end{lemma}

\begin{proof}
  We refer to \cite[Theorem II.5.12]{BoyFab13}.
  Note that $g$ is independent of time.
\end{proof}

\begin{lemma} \label{L:DtEn_gSu}
  For $v\in\mathcal{E}_T(\Gamma)\cap L^2(0,T;H^3(\Gamma))$, the following statements hold:
  \begin{enumerate}
    \item for $A_gv=g^{-1}\mathrm{div}_\Gamma(g\nabla_\Gamma v)$ on $\Gamma$, we have
    \begin{align} \label{E:DtGr_gSu}
      \frac{1}{2}\frac{d}{dt}\int_\Gamma g|\nabla_\Gamma v(t)|^2\,d\mathcal{H}^2 = \langle\partial_tv(t),-gA_gv(t)\rangle_{H^1(\Gamma)} \quad\text{for a.a.}\quad t\in(0,T),
    \end{align}
    \item if $F\in C^3(\mathbb{R})$ and $|F'''(z)|\leq c(|z|+1)$ for all $z\in\mathbb{R}$, then
    \begin{gather}
      F(v) \in C([0,T];L^1(\Gamma)), \quad F'(v) \in L^2(0,T;H^1(\Gamma)), \notag \\
      \frac{d}{dt}\int_\Gamma gF(v(t))\,d\mathcal{H}^2 = \langle\partial_tv(t),gF'(v(t))\rangle_{H^1(\Gamma)} \quad\text{for a.a.}\quad t\in(0,T). \label{E:DtFu_gSu}
    \end{gather}
  \end{enumerate}
\end{lemma}

\begin{lemma} \label{L:DtIA_Sur}
  Let $L_g$ be the operator given in Lemma \ref{L:InA_Sur}, and let
  \begin{align*}
    f\in H^1(0,T;[H^1(\Gamma)]'), \quad \langle f(t),g\rangle_{H^1(\Gamma)} = 0 \quad\text{for all}\quad t\in[0,T].
  \end{align*}
  Then, $L_gf\in H^1(0,T;H^1(\Gamma))$ and $\partial_tL_gf=L_g(\partial_tf)$ a.e. on $\Gamma\times(0,T)$.
  Also, we have
  \begin{align} \label{E:DtIA_Sur}
    \frac{1}{2}\frac{d}{dt}\int_\Gamma g|\nabla_\Gamma L_gv(t)|^2\,d\mathcal{H}^2 = \langle\partial_tv(t),gL_gv(t)\rangle_{H^1(\Gamma)} \quad\text{for a.a.}\quad t\in(0,T),
  \end{align}
  if $v\in\mathcal{E}_T(\Gamma)$ and it satisfies $\int_\Gamma gv(t)\,d\mathcal{H}^2=0$ for all $t\in[0,T]$.
\end{lemma}

We give the proofs of Lemmas \ref{L:DtEn_gSu} and \ref{L:DtIA_Sur} in Section \ref{S:Pf_Aux}.

Next, we consider the case $S=\Omega_\varepsilon$.

\begin{lemma} \label{L:DtL2_CTD}
  Let $u_1,u_2\in\mathcal{E}_T(\Omega_\varepsilon)$.
  Then, for a.a. $t\in(0,T)$, we have
  \begin{align} \label{E:DtL2_CTD}
    \frac{d}{dt}(u_1(t),u_2(t))_{L^2(\Omega_\varepsilon)} = \langle\partial_tu_1(t),u_2(t)\rangle_{H^1(\Omega_\varepsilon)}+\langle\partial_tu_2(t),u_1(t)\rangle_{H^1(\Omega_\varepsilon)}.
  \end{align}
\end{lemma}

\begin{proof}
  We refer to \cite[Theorem II.5.12]{BoyFab13} for the proof.
\end{proof}

\begin{lemma} \label{L:DtEn_CTD}
  For $u\in\mathcal{E}_T(\Omega_\varepsilon)\cap L^2(0,T;H^3(\Omega_\varepsilon))$, the following statements hold:
  \begin{enumerate}
    \item if $\partial_{\nu_\varepsilon}u(t)=0$ on $\partial\Omega_\varepsilon$ for a.a. $t\in(0,T)$, then
    \begin{align*}
      \frac{1}{2}\frac{d}{dt}\int_{\Omega_\varepsilon}|\nabla u(t)|^2\,dx = \langle\partial_tu(t),-\Delta u(t)\rangle_{H^1(\Omega_\varepsilon)} \quad\text{for a.a.}\quad t\in(0,T),
    \end{align*}
    \item if $F\in C^3(\mathbb{R})$ and $|F'''(z)|\leq c(|z|+1)$ for all $z\in\mathbb{R}$, then
    \begin{gather*}
      F(u) \in C([0,T];L^1(\Omega_\varepsilon)), \quad F'(u) \in L^2(0,T;H^1(\Omega_\varepsilon)), \\
      \frac{d}{dt}\int_{\Omega_\varepsilon} F(u(t))\,dx = \langle\partial_tu(t),F'(u(t))\rangle_{H^1(\Omega_\varepsilon)} \quad\text{for a.a.}\quad t\in(0,T).
    \end{gather*}
  \end{enumerate}
\end{lemma}

The proof of Lemma \ref{L:DtEn_CTD} is the same as that of Lemma \ref{L:DtEn_gSu}.
Thus, we omit the proof, but we should give a remark on the boundary condition in (i).
To prove (i), we regularize $u$ in time by cut-off, translation, and mollification as in \cite[Lemma II.5.10]{BoyFab13}.
This procedure is carried out only with respect to the time variable, and we do nothing on the spatial variable.
Hence, if $\partial_{\nu_\varepsilon}u(t)=0$ on $\partial\Omega_\varepsilon$ for a.a. $t\in(0,T)$, then the regularization of $u$ in time also satisfies the same boundary condition for all $t\in(0,T)$.

\section{Weak solution to the thin-domain problem} \label{S:CHT_Weak}
Let us give a weak formulation of the thin-domain problem \eqref{E:CH_CTD} and show some properties of a weak solution to \eqref{E:CH_CTD}.
Recall that we make Assumption \ref{A:Poten} on $F$.

\subsection{Weak formulation} \label{SS:CHT_WF}
We multiply \eqref{E:CH_CTD} by a test function, carry out integration by parts, and use the boundary conditions to get the following weak formulation.

\begin{definition} \label{D:CHT_WeSo}
  For given $u_0^\varepsilon\in H^1(\Omega_\varepsilon)$ and $T>0$, we say that a pair $(u^\varepsilon,w^\varepsilon)$ is a weak solution to \eqref{E:CH_CTD} on $[0,T)$ if it satisfies
  \begin{align} \label{E:CHT_WeRe}
    u^\varepsilon\in\mathcal{E}_T(\Omega_\varepsilon)\cap L^\infty(0,T;H^1(\Omega_\varepsilon)), \quad w^\varepsilon\in L^2(0,T;H^1(\Omega_\varepsilon)),
  \end{align}
  and for all $\varphi\in L^2(0,T;H^1(\Omega_\varepsilon))$,
  \begin{align}
    &\int_0^T\langle\partial_tu^\varepsilon,\varphi\rangle_{H^1(\Omega_\varepsilon)}\,dt+\int_0^T(\nabla w^\varepsilon,\nabla\varphi)_{L^2(\Omega_\varepsilon)}\,dx = 0, \label{E:CHT_WF_u} \\
    &\int_0^T(w^\varepsilon,\varphi)_{L^2(\Omega_\varepsilon)}\,dt = \int_0^T(\nabla u^\varepsilon,\nabla\varphi)_{L^2(\Omega_\varepsilon)}\,dt+\int_0^T(F'(u^\varepsilon),\varphi)_{L^2(\Omega_\varepsilon)}\,dt, \label{E:CHT_WF_w}
  \end{align}
  and $u^\varepsilon(0)=u_0^\varepsilon$ a.e. in $\Omega_\varepsilon$.
\end{definition}

\begin{definition} \label{D:CHT_GlW}
  For a given $u_0^\varepsilon\in H^1(\Omega_\varepsilon)$, we say that a pair $(u^\varepsilon,w^\varepsilon)$ is a global weak solution to \eqref{E:CH_CTD} if it is a weak solution to \eqref{E:CH_CTD} on $[0,T)$ for all $T>0$.
\end{definition}

\begin{remark} \label{R:CHT_WeSo}
  When $u^\varepsilon\in L^\infty(0,T;H^1(\Omega_\varepsilon))$, we see by \eqref{E:Sob_CTD} with $p=6$ that
  \begin{align*}
    \|[(u^\varepsilon)^3](t)\|_{L^2(\Omega_\varepsilon)} = \|u^\varepsilon(t)\|_{L^6(\Omega_\varepsilon)}^3 \leq c\varepsilon^{-1}\|u^\varepsilon(t)\|_{H^1(\Omega_\varepsilon)}^3 \quad\text{for a.a.}\quad t\in(0,T).
  \end{align*}
  By this estimate and $|F'(z)|\leq c(|z|^3+1)$ for $z\in\mathbb{R}$ (see Remark \ref{R:Po_Grow}),
  \begin{align} \label{E:FL2_CTD}
    \begin{aligned}
      \|F'(u^\varepsilon)\|_{L^2(0,T;L^2(\Omega_\varepsilon))} \leq c\Bigl(\varepsilon^{-1}\|u^\varepsilon\|_{L^\infty(0,T;H^1(\Omega_\varepsilon))}^2\|u^\varepsilon\|_{L^2(0,T;H^1(\Omega_\varepsilon))}+T^{1/2}|\Omega_\varepsilon|^{1/2}\Bigr).
    \end{aligned}
  \end{align}
  Also, when $u^\varepsilon\in\mathcal{E}_T(\Omega_\varepsilon)$, the initial condition makes sense by Lemma \ref{L:ET_Emb}.
\end{remark}

For a weak solution to \eqref{E:CH_CTD}, the following uniqueness and existence results hold.

\begin{proposition} \label{P:CHT_Uni}
  Suppose that Assumption \ref{A:Poten} is satisfied.
  Then, for each $T>0$ and $u_0^\varepsilon\in H^1(\Omega_\varepsilon)$, there exists at most one weak solution to \eqref{E:CH_CTD} on $[0,T)$.
\end{proposition}

We omit the proof of Proposition \ref{P:CHT_Uni}, since it is the same as that of Proposition \ref{P:CLi_Uni} below on the uniqueness of a weak solution to the limit problem \eqref{E:CH_Lim}.
Here, we just note that the proof of Proposition \ref{P:CHT_Uni} uses the solution operator $L_\varepsilon$ of the problem
\begin{align*}
  L_\varepsilon f\in H^1(\Omega_\varepsilon), \quad \int_{\Omega_\varepsilon}L_\varepsilon f\,dx = 0, \quad (\nabla L_\varepsilon f,\nabla\varphi)_{L^2(\Omega_\varepsilon)} = \langle f,\varphi\rangle_{H^1(\Omega_\varepsilon)}
\end{align*}
for all $\varphi\in H^1(\Omega_\varepsilon)$, where $f\in[H^1(\Omega_\varepsilon)]'$ is given and satisfies $\langle f,1\rangle_{H^1(\Omega_\varepsilon)}=0$.

\begin{proposition} \label{P:CHT_GlEx}
  Let $u_0^\varepsilon\in H^1(\Omega_\varepsilon)$ and Assumption \ref{A:Poten} be satisfied.
  Then, there exists a unique global weak solution $(u^\varepsilon,w^\varepsilon)$ to \eqref{E:CH_CTD}.
\end{proposition}

\begin{proof}
  We have the uniqueness by Proposition \ref{P:CHT_Uni}.
  The existence is shown by the Galerkin method.
  For the reader's convenience, we explain the outline in Section \ref{S:Galer}.
\end{proof}

Let us show that $u^\varepsilon$ has a higher regularity.

\begin{proposition} \label{P:CHT_Reg}
  Let $u_0^\varepsilon\in H^1(\Omega_\varepsilon)$ and Assumption \ref{A:Poten} be satisfied.
  Also, let $(u^\varepsilon,w^\varepsilon)$ be the unique global weak solution to \eqref{E:CH_CTD}.
  Then, for all $T>0$,
  \begin{enumerate}
    \item $u^\varepsilon\in L^2(0,T;H^2(\Omega_\varepsilon))$ and, for a.a. $t\in(0,T)$,
    \begin{align} \label{E:CHT_Poi}
      \left\{
      \begin{aligned}
        -\Delta u^\varepsilon(t) &= w^\varepsilon(t)-F'(u^\varepsilon(t)) \quad\text{a.e. in}\quad \Omega_\varepsilon, \\
        \partial_{\nu_\varepsilon}u^\varepsilon(t) &= 0 \quad\text{a.e. on}\quad \partial\Omega_\varepsilon,
      \end{aligned}
      \right.
    \end{align}
    \item $F'(u^\varepsilon)\in L^2(0,T;H^1(\Omega_\varepsilon))$ and
    \begin{align} \label{E:FH1_CTD}
      \|\nabla F'(u^\varepsilon)\|_{L^2(0,T;L^2(\Omega_\varepsilon))} \leq c\Bigl(\varepsilon^{-1}\|u^\varepsilon\|_{L^\infty(0,T;H^1(\Omega_\varepsilon))}^2+1\Bigr)\|u^\varepsilon\|_{L^2(0,T;H^2(\Omega_\varepsilon))},
    \end{align}
    \item $u^\varepsilon\in L^2(0,T;H^3(\Omega_\varepsilon))\cap C([0,T];H^1(\Omega_\varepsilon))$.
  \end{enumerate}
\end{proposition}

\begin{proof}
  Fix any $T>0$.

  Let us prove (i).
  In \eqref{E:CHT_WF_w}, we take a test function of the form
  \begin{align*}
    \varphi(x,t) = f(t)\zeta(x) \quad\text{with any}\quad f\in C_c^\infty(0,T), \, \zeta\in H^1(\Omega_\varepsilon).
  \end{align*}
  Then, we find that, for a.a. $t\in(0,T)$ and all $\zeta\in H^1(\Omega_\varepsilon)$,
  \begin{align*}
    (\nabla u^\varepsilon(t),\nabla\zeta)_{L^2(\Omega_\varepsilon)} = (w^\varepsilon(t)-F'(u^\varepsilon(t)),\zeta)_{L^2(\Omega_\varepsilon)}.
  \end{align*}
  Hence, $u^\varepsilon(t)\in H^2(\Omega_\varepsilon)$ and \eqref{E:CHT_Poi} holds by the elliptic regularity theorem (see e.g. \cite{GilTru01}).
  Moreover, since we can apply \eqref{E:UER_CTD} with $k=0$ to $u^\varepsilon(t)$, and since
  \begin{align*}
    -\Delta u^\varepsilon = w^\varepsilon-F'(u^\varepsilon) \in L^2(0,T;L^2(\Omega_\varepsilon))
  \end{align*}
  by \eqref{E:CHT_WeRe} and \eqref{E:FL2_CTD}, it follows that $u^\varepsilon\in L^2(0,T;H^2(\Omega_\varepsilon))$.

  Next, we prove (ii).
  Since $|F''(z)|\leq c(|z|^2+1)$ for $z\in\mathbb{R}$ by Remark \ref{R:Po_Grow},
  \begin{align*}
    |\nabla F'(u^\varepsilon)| = |F''(u^\varepsilon)\nabla u^\varepsilon| \leq c(|u^\varepsilon|^2+1)|\nabla u^\varepsilon| \quad\text{a.e. in}\quad \Omega_\varepsilon\times(0,T).
  \end{align*}
  Moreover, for a.a. $t\in(0,T)$, we see by H\"{o}lder's inequality and \eqref{E:Sob_CTD} that
  \begin{align*}
    \|[(u^\varepsilon)^2\nabla u^\varepsilon](t)\|_{L^2(\Omega_\varepsilon)} &\leq \|u^\varepsilon(t)\|_{L^6(\Omega_\varepsilon)}^2\|\nabla u^\varepsilon(t)\|_{L^6(\Omega_\varepsilon)} \\
    &\leq c\varepsilon^{-1}\|u^\varepsilon(t)\|_{H^1(\Omega_\varepsilon)}^2\|\nabla u^\varepsilon(t)\|_{H^1(\Omega_\varepsilon)} \\
    &\leq c\varepsilon^{-1}\|u^\varepsilon\|_{L^\infty(0,T;H^1(\Omega_\varepsilon))}^2\|u^\varepsilon(t)\|_{H^2(\Omega_\varepsilon)}.
  \end{align*}
  By these inequalities and
  \begin{align*}
    \|\nabla u^\varepsilon\|_{L^2(0,T;L^2(\Omega_\varepsilon))} \leq \|u^\varepsilon\|_{L^2(0,T;H^1(\Omega_\varepsilon))} \leq \|u^\varepsilon\|_{L^2(0,T;H^2(\Omega_\varepsilon))},
  \end{align*}
  we get \eqref{E:FH1_CTD}.
  Also, $F'(u^\varepsilon)\in L^2(0,T;H^1(\Omega_\varepsilon))$ by \eqref{E:FL2_CTD} and \eqref{E:FH1_CTD}.

  Lastly, let us show (iii).
  By (i), (ii), and \eqref{E:CHT_WeRe}, we see that
  \begin{align*}
    -\Delta u^\varepsilon(t) = w^\varepsilon(t)-F'(u^\varepsilon(t)) \in H^1(\Omega_\varepsilon) \quad\text{for a.a.}\quad t\in(0,T).
  \end{align*}
  Hence, $u^\varepsilon(t)\in H^3(\Omega_\varepsilon)$ by the elliptic regularity theorem.
  Moreover, since $\partial_{\nu_\varepsilon}u^\varepsilon(t)=0$ on $\partial\Omega_\varepsilon$, we can apply \eqref{E:UER_CTD} with $k=1$ to $u^\varepsilon(t)$.
  By this fact and
  \begin{align*}
    -\Delta u^\varepsilon = w^\varepsilon-F'(u^\varepsilon) \in L^2(0,T;H^1(\Omega_\varepsilon)),
  \end{align*}
  we get $u^\varepsilon\in L^2(0,T;H^3(\Omega_\varepsilon))$.
  We also observe that
  \begin{align*}
    u^\varepsilon \in \mathcal{E}_T(\Omega_\varepsilon)\cap L^2(0,T;H^3(\Omega_\varepsilon)) \subset C([0,T];H^1(\Omega_\varepsilon))
  \end{align*}
  by \eqref{E:CHT_WeRe} and Lemma \ref{L:ET_Emb}.
\end{proof}

\subsection{Uniform energy estimates} \label{SS:CHT_Ena}
Let $u_0^\varepsilon\in H^1(\Omega_\varepsilon)$ and Assumption \ref{A:Poten} be satisfied.
Also, let $(u^\varepsilon,w^\varepsilon)$ be the unique weak solution to \eqref{E:CH_CTD}.

The purpose of this subsection is to derive energy estimates for $u^\varepsilon$ and $w^\varepsilon$ with constants independent of $\varepsilon$.
Such uniform energy estimates are crucial in the study of the thin-film limit of \eqref{E:CH_CTD} carried out in Section \ref{S:TFL}.

First, we see that an energy equality holds for the Ginzburg--Landau free energy
\begin{align*}
  E_\varepsilon(u^\varepsilon) = \int_{\Omega_\varepsilon}\left(\frac{|\nabla u^\varepsilon|^2}{2}+F(u^\varepsilon)\right)\,dx.
\end{align*}
Note that $E_\varepsilon(u^\varepsilon)$ may take a negative value under Assumption \ref{A:Poten}.
Also, recall that the symbol $c$ denotes a general positive constant independent of $\varepsilon$.

\begin{proposition} \label{P:CHT_GLE}
  For all $t\geq0$, we have
  \begin{align}
    E_\varepsilon(u^\varepsilon(t))+\int_0^t\|\nabla w^\varepsilon(s)\|_{L^2(\Omega_\varepsilon)}^2\,ds &= E_\varepsilon(u_0^\varepsilon), \label{E:CHT_GLE} \\
    \|\nabla u^\varepsilon(t)\|_{L^2(\Omega_\varepsilon)}^2+\int_0^t\|\nabla w^\varepsilon(s)\|_{L^2(\Omega_\varepsilon)}^2\,ds &\leq c\{|E(u_0^\varepsilon)|+\varepsilon\}. \label{E:CHT_Gr2}
  \end{align}
\end{proposition}

\begin{proof}
  By Assumption \ref{A:Poten} and Proposition \ref{P:CHT_Reg}, we can use Lemma \ref{L:DtEn_CTD} to get
  \begin{align*}
    \frac{d}{dt}E_\varepsilon(u^\varepsilon) = \langle\partial_tu^\varepsilon,-\Delta u^\varepsilon+F'(u^\varepsilon)\rangle_{H^1(\Omega_\varepsilon)} = \langle\partial_tu^\varepsilon,w^\varepsilon\rangle_{H^1(\Omega_\varepsilon)} \quad\text{a.e. on}\quad (0,\infty),
  \end{align*}
  where the last equality follows from \eqref{E:CHT_Poi}.
  Integrating both sides over $(0,t)$ and applying \eqref{E:CHT_WF_u} with $\varphi=w^\varepsilon$, we obtain \eqref{E:CHT_GLE}.
  For each $t\geq0$, we see by \eqref{E:Poten} that
  \begin{align*}
    E_\varepsilon(u^\varepsilon(t))+C_0|\Omega_\varepsilon| = \int_{\Omega_\varepsilon}\left(\frac{|\nabla u^\varepsilon(t)|^2}{2}+\{F(u^\varepsilon(t))+C_0\}\right)\,dx \geq \frac{1}{2}\|\nabla u^\varepsilon(t)\|_{L^2(\Omega_\varepsilon)}^2.
  \end{align*}
  From this inequality and \eqref{E:CHT_GLE}, we deduce that
  \begin{align*}
    \frac{1}{2}\|\nabla u^\varepsilon(t)\|_{L^2(\Omega_\varepsilon)}^2+\int_0^t\|\nabla w^\varepsilon(s)\|_{L^2(\Omega_\varepsilon)}^2\,ds \leq E_\varepsilon(u_0^\varepsilon)+C_0|\Omega_\varepsilon|.
  \end{align*}
  Hence, \eqref{E:CHT_Gr2} follows from $E_\varepsilon(u_0^\varepsilon)\leq|E_\varepsilon(u_0^\varepsilon)|$ and $|\Omega_\varepsilon|\leq c\varepsilon$.
\end{proof}

Next, we derive a uniform $L^2$-energy estimate.

\begin{proposition} \label{P:CHT_EL2}
  For all $t\geq0$, we have
  \begin{align}
    &\|u^\varepsilon(t)\|_{L^2(\Omega_\varepsilon)}^2+\int_0^t\|\Delta u^\varepsilon(s)\|_{L^2(\Omega_\varepsilon)}^2\,ds \leq c\Bigl(\|u_0^\varepsilon\|_{L^2(\Omega_\varepsilon)}^2+t\{|E_\varepsilon(u_0^\varepsilon)|+\varepsilon\}\Bigr), \label{E:CHT_EL2} \\
    &\int_0^t\|u^\varepsilon(s)\|_{H^2(\Omega_\varepsilon)}^2\,ds \leq c(1+t)^2\Bigl(\|u_0^\varepsilon\|_{L^2(\Omega_\varepsilon)}^2+|E_\varepsilon(u_0^\varepsilon)|+\varepsilon\Bigr). \label{E:CHT_H2E}
  \end{align}
\end{proposition}

\begin{proof}
  For the sake of simplicity, we sometimes omit the time variable.

  Let $\varphi=u^\varepsilon$ in \eqref{E:CHT_WF_u} with $T$ replaced by $t\geq0$.
  Then, by \eqref{E:DtL2_CTD},
  \begin{align*}
    \frac{1}{2}\|u^\varepsilon(t)\|_{L^2(\Omega_\varepsilon)}^2+\int_0^t(\nabla w^\varepsilon,\nabla u^\varepsilon)_{L^2(\Omega_\varepsilon)}\,ds = \frac{1}{2}\|u_0^\varepsilon\|_{L^2(\Omega_\varepsilon)}^2.
  \end{align*}
  Moreover, we see by \eqref{E:CHT_Poi}, integration by parts, and \eqref{E:Poten} that
  \begin{align*}
    (\nabla w^\varepsilon,\nabla u^\varepsilon)_{L^2(\Omega_\varepsilon)} &= (\nabla[-\Delta u^\varepsilon+F'(u^\varepsilon)],\nabla u^\varepsilon)_{L^2(\Omega_\varepsilon)} \\
    &= \|\Delta u^\varepsilon\|_{L^2(\Omega_\varepsilon)}^2+\int_{\Omega_\varepsilon}F''(u^\varepsilon)|\nabla u^\varepsilon|^2\,dx \\
    &\geq \|\Delta u^\varepsilon\|_{L^2(\Omega_\varepsilon)}^2-C_2\|\nabla u^\varepsilon\|_{L^2(\Omega_\varepsilon)}^2.
  \end{align*}
  Hence, we find that
  \begin{align*}
    \frac{1}{2}\|u^\varepsilon(t)\|_{L^2(\Omega_\varepsilon)}^2+\int_0^t\|\Delta u^\varepsilon\|_{L^2(\Omega_\varepsilon)}^2\,ds \leq \frac{1}{2}\|u_0^\varepsilon\|_{L^2(\Omega_\varepsilon)}^2+C_2\int_0^t\|\nabla u^\varepsilon\|_{L^2(\Omega_\varepsilon)}^2\,ds,
  \end{align*}
  and we get \eqref{E:CHT_EL2} by using \eqref{E:CHT_Gr2} to the last term.
  Also, for a.a. $s>0$, we can apply \eqref{E:UER_CTD} with $k=0$ to $u^\varepsilon(s)$ by Proposition \ref{P:CHT_Reg}.
  Hence, \eqref{E:CHT_H2E} follows from \eqref{E:CHT_EL2}.
\end{proof}

\section{Weighted average operator} \label{S:Ave}
In this section, we study the weighted average of functions on $\Omega_\varepsilon$ in the thin direction.
We write $c$ for a general positive constant independent of $\varepsilon$.
Also, for a function $\eta$ on $\Gamma$, let $\bar{\eta}$ be the constant extension of $\eta$ in the normal direction of $\Gamma$.

\subsection{Definition and basic properties} \label{Ss:Ave_Def}
For a function $u$ in $\Omega_\varepsilon$, we define the weighted average of $u$ in the thin direction by
\begin{align*}
  \mathcal{M}_\varepsilon u(y) := \frac{1}{\varepsilon g(y)}\int_{\varepsilon g_0(y)}^{\varepsilon g_1(y)}u(y+r\bm{\nu}(y,t))J(y,r)\,dr, \quad y\in\Gamma.
\end{align*}
Here, $J$ is the Jacobian given by \eqref{E:Def_J}.
The change of variables formula \eqref{E:CoV_CTD} gives
\begin{align} \label{E:Ave_Pair}
  \int_{\Omega_\varepsilon}u(x)\bar{\eta}(x)\,dx = \varepsilon\int_{\Gamma}g(y)\mathcal{M}_\varepsilon u(y)\eta(y)\,d\mathcal{H}^{n-1}(y)
\end{align}
for functions $u$ on $\Omega_\varepsilon$ and $\eta$ on $\Gamma$.
Note that the formula \eqref{E:Ave_Pair} does not involve any error.
This is the main advantage of taking the weighted average in our analysis.

Let us give basic properties of $\mathcal{M}_\varepsilon$.
We set $\partial_\nu u=\bar{\bm{\nu}}\cdot\nabla u$ for a function $u$ on $\Omega_\varepsilon$, which is the derivative of $u$ in the normal direction of $\Gamma$.
Also, we use the notation \eqref{E:Pull_CTD} and sometimes suppress the argument $y\in\Gamma$.
For example, we write
\begin{align*}
  \mathcal{M}_\varepsilon u = \frac{1}{\varepsilon g}\int_{\varepsilon g_0}^{\varepsilon g_1}u^\sharp(r)J(r)\,dr \quad\text{on}\quad \Gamma.
\end{align*}
Note that $(\partial_\nu u)^\sharp(r)=\partial_ru^\sharp(r)$.
Also, if $\eta$ is a function on $\Gamma$, then $\partial_\nu\bar{\eta}=0$ in $\Omega_\varepsilon$.

\begin{lemma} \label{L:Ave_Lp}
  Let $p\in[1,\infty]$ and $u\in L^p(\Omega_\varepsilon)$.
  Then, $\mathcal{M}_\varepsilon u\in L^p(\Gamma)$ and
  \begin{align}
    \|\mathcal{M}_\varepsilon u\|_{L^p(\Gamma)} &\leq c\varepsilon^{-1/p}\|u\|_{L^p(\Omega_\varepsilon)}, \label{E:ALp_Sur} \\
    \left\|\overline{\mathcal{M}_\varepsilon u}\right\|_{L^p(\Omega_\varepsilon)} &\leq c\|u\|_{L^p(\Omega_\varepsilon)}. \label{E:ALp_TD}
  \end{align}
\end{lemma}

\begin{proof}
  The case $p=\infty$ just follows from \eqref{E:J_Bdd}.
  Let $p\neq\infty$.
  Then,
  \begin{align*}
    |\mathcal{M}_\varepsilon u| \leq \frac{1}{\varepsilon g}\int_{\varepsilon g_0}^{\varepsilon g_1}|u^\sharp(r)|J(r)\,dr \leq c\varepsilon^{-1/p}\left(\int_{\varepsilon g_0}^{\varepsilon g_1}|u^\sharp(r)|^pJ(r)\,dr\right)^{1/p}
  \end{align*}
  by H\"{o}lder's inequality, \eqref{E:G_Bdd}, and \eqref{E:J_Bdd}.
  We integrate the $p$-th power of both sides over $\Gamma$ and use \eqref{E:CoV_CTD} to get \eqref{E:ALp_Sur}.
  Also, \eqref{E:ALp_TD} follows from \eqref{E:Lp_CE} and \eqref{E:ALp_Sur}.
\end{proof}

\begin{lemma} \label{L:ADif_Lp}
  Let $p\in[1,\infty]$ and $u\in W^{1,p}(\Omega_\varepsilon)$.
  Then,
  \begin{align}
    \left|u-\overline{\mathcal{M}_\varepsilon u}\right| &\leq c\varepsilon\,\overline{\mathcal{M}_\varepsilon(|u|+|\partial_\nu u|)} \quad\text{in}\quad \Omega_\varepsilon, \label{E:ADif_PW} \\
    \left\|u-\overline{\mathcal{M}_\varepsilon u}\right\|_{L^p(\Omega_\varepsilon)} &\leq c\varepsilon\|u\|_{W^{1,p}(\Omega_\varepsilon)}. \label{E:ADif_Lp}
  \end{align}
\end{lemma}

\begin{proof}
  Let $r\in(\varepsilon g_0,\varepsilon g_1)$ (recall that we suppress $y\in\Gamma$).
  We write
  \begin{align*}
    &u^\sharp(r)-\mathcal{M}_\varepsilon u = u^\sharp(r)-\frac{1}{\varepsilon g}\int_{\varepsilon g_0}^{\varepsilon g_1}u^\sharp(r_1)J(r_1)\,dr_1 = K_1+K_2, \\
    &K_1 := \frac{1}{\varepsilon g}\int_{\varepsilon g_0}^{\varepsilon g_1}\{u^\sharp(r)-u^\sharp(r_1)\}\,dr_1, \quad K_2 := \frac{1}{\varepsilon g}\int_{\varepsilon g_0}^{\varepsilon g_1}u^\sharp(r_1)\{1-J(r_1)\}\,dr_1.
  \end{align*}
  For $K_1$, we see by $\partial_ru^\sharp(r)=(\partial_\nu u)^\sharp(r)$ that
  \begin{align*}
    |u^\sharp(r)-u^\sharp(r_1)| = \left|\int_{r_1}^r(\partial_\nu u)^\sharp(r_2)\,dr_2\right| \leq \int_{\varepsilon g_0}^{\varepsilon g_1}|(\partial_\nu u)^\sharp(r_2)|\,dr_2.
  \end{align*}
  Here, the last term is independent of $r$ and $r_1$.
  Hence,
  \begin{align*}
    |K_1| \leq \frac{1}{\varepsilon g}\cdot(\varepsilon g)\int_{\varepsilon g_0}^{\varepsilon g_1}|(\partial_\nu u)^\sharp(r_2)|\,dr_2 = \varepsilon g\mathcal{M}_\varepsilon(|\partial_\nu u|) \leq c\varepsilon\mathcal{M}_\varepsilon(|\partial_\nu u|)
  \end{align*}
  by \eqref{E:G_Bdd}.
  Also, it follows from \eqref{E:J_Diff} that
  \begin{align*}
    |K_2| \leq c\varepsilon\cdot\frac{1}{\varepsilon g}\int_{\varepsilon g_0}^{\varepsilon g_1}|u^\sharp(r_1)|\,dr_1 = c\varepsilon\mathcal{M}_\varepsilon(|u|).
  \end{align*}
  By these estimates and $\mathcal{M}_\varepsilon(|u|)+\mathcal{M}_\varepsilon(|\partial_\nu u|)=\mathcal{M}_\varepsilon(|u|+|\partial_\nu u|)$, we get
  \begin{align*}
    |u^\sharp(r)-\mathcal{M}_\varepsilon u| \leq |K_1|+|K_2| \leq c\varepsilon\mathcal{M}_\varepsilon(|u|+|\partial_\nu u|),
  \end{align*}
  which is \eqref{E:ADif_PW} under the notation \eqref{E:Pull_CTD}.
  Also, we have \eqref{E:ADif_Lp} by using \eqref{E:ADif_PW} and applying \eqref{E:ALp_TD} and $|\partial_\nu u|=|\bm{\nu}\cdot\nabla u|\leq|\nabla u|$ in $\Omega_\varepsilon$.
\end{proof}

\subsection{Tangential derivatives of the average} \label{SS:Ave_TGr}
We can compute the tangential derivatives of $\mathcal{M}_\varepsilon u$ explicitly as follows.

\begin{lemma} \label{L:Ave_TGr}
  Let $u\in C(\overline{\Omega_\varepsilon})\cap C^1(\Omega_\varepsilon)$.
  Then,
  \begin{align} \label{E:Ave_TGr}
    \nabla_\Gamma\mathcal{M}_\varepsilon u = \mathcal{M}_\varepsilon(\mathbf{B}\nabla u)+\mathcal{M}_\varepsilon((\partial_\nu u+uf_J)\mathbf{b}_\varepsilon)+\mathcal{M}_\varepsilon(u\mathbf{b}_J) \quad\text{on}\quad \Gamma.
  \end{align}
  Here, $\mathbf{B}\colon\overline{\mathcal{N}_\delta}\to\mathbb{R}^{3\times 3}$, $f_J\colon\overline{\mathcal{N}_\delta}\to\mathbb{R}$, and $\mathbf{b}_\varepsilon,\mathbf{b}_J\colon\overline{\mathcal{N}_\delta}\to\mathbb{R}^3$ are given by
  \begin{align} \label{E:Def_Bfb}
    \begin{aligned}
      \mathbf{B}(x) &:= \overline{\mathbf{P}}(x)-d(x)\overline{\mathbf{W}}(x), \quad f_J(x) := \frac{-\overline{H}(x)+2d(x)\overline{K}(x)}{J_{\pi,d}(x)}, \\
      \mathbf{b}_\varepsilon(x) &:= \frac{1}{\bar{g}(x)}\Bigl[\{d(x)-\varepsilon\bar{g}_0(x)\}\overline{\nabla_\Gamma g_1}(x)+\{\varepsilon\bar{g}_1(x)-d(x)\}\overline{\nabla_\Gamma g_0}(x)\Bigr], \\
      \mathbf{b}_J(x) &:= \frac{d(x)}{J_{\pi,d}(x)}\Bigl\{-\overline{\nabla_\Gamma H}(x)+d(x)\overline{\nabla_\Gamma K}(x)\Bigr\}
    \end{aligned}
  \end{align}
  for $x\in\overline{\mathcal{N}_\delta}$, where $J_{\pi,d}(x):=J(\pi(x),d(x))$ with $\pi(x)\in\Gamma$ given in \eqref{E:Fermi}.
  Moreover,
  \begin{align} \label{E:ATG_Diff}
    \left|\nabla_\Gamma\mathcal{M}_\varepsilon u-\mathcal{M}_\varepsilon\Bigl(\overline{\mathbf{P}}\nabla u\Bigr)\right| \leq c\varepsilon\mathcal{M}_\varepsilon(|u|+|\nabla u|) \quad\text{on}\quad \Gamma.
  \end{align}
\end{lemma}

\begin{proof}
  We refer to \cite[Lemma 5.6]{Miu24_GL} for the proof of \eqref{E:Ave_TGr}.
  Note that $\mathbf{b}_\varepsilon$ and $\mathbf{b}_J$ are denoted by $\Psi_\varepsilon$ and $\Psi_J$ in \cite{Miu24_GL}, respectively.
  Also, in \cite{Miu24_GL}, $f_J$ and $\mathbf{b}_J$ are written as
  \begin{align*}
    f_J(y+r\bm{\nu}(y)) = \frac{\partial_rJ(y,r)}{J(y,r)}, \quad \mathbf{b}_J(y+r\bm{\nu}(y)) = \frac{\nabla_\Gamma J(y,r)}{J(y,r)}
  \end{align*}
  for $x=y+r\bm{\nu}(y)\in\overline{\mathcal{N}_\delta}$ (see also \eqref{E:Def_J} for the definition of $J$).

  In \eqref{E:Ave_TGr}, we see that $|\partial_\nu u|=|\bar{\bm{\nu}}\cdot\nabla u|\leq|\nabla u|$ in $\Omega_\varepsilon$ and
  \begin{align} \label{Pf_ATG:coef}
    |d| \leq c\varepsilon, \quad |\mathbf{b}_\varepsilon| \leq c\varepsilon, \quad |\mathbf{b}_J| \leq c\varepsilon, \quad |f_J| \leq c \quad\text{in}\quad \Omega_\varepsilon, \quad |\mathbf{W}| \leq c \quad\text{on}\quad \Gamma
  \end{align}
  by \eqref{E:G_Bdd}, \eqref{E:J_Bdd}, and the regularity of $\Gamma$, $g_0$, and $g_1$.
  Thus, \eqref{E:ATG_Diff} follows.
\end{proof}

\begin{lemma} \label{L:Ave_Hk}
  Let $k=1,2,3$ and $u\in H^k(\Omega_\varepsilon)$.
  Then, $\mathcal{M}_\varepsilon u\in H^k(\Gamma)$ and
  \begin{align}
    \|\mathcal{M}_\varepsilon u\|_{H^k(\Gamma)} &\leq c\varepsilon^{-1/2}\|u\|_{H^k(\Omega_\varepsilon)}, \label{E:AHk_Sur} \\
    \left\|\overline{\mathcal{M}_\varepsilon u}\right\|_{H^k(\Omega_\varepsilon)} &\leq c\|u\|_{H^k(\Omega_\varepsilon)}. \label{E:AHk_TD}
  \end{align}
\end{lemma}

\begin{proof}
  By \eqref{E:CEGr_Bd}, \eqref{E:J_Bdd}, and the regularity of $\Gamma$, $g_0$, and $g_1$, we have
  \begin{align} \label{Pf_AHk:Coef}
    |\varphi|+|\nabla\varphi|+|\nabla^2\varphi| \leq c \quad\text{in}\quad \overline{\mathcal{N}_\delta}, \quad \varphi = \bar{\bm{\nu}}, \mathbf{B}, f_J, \mathbf{b}_\varepsilon, \mathbf{b}_J.
  \end{align}
  We use \eqref{E:Ave_TGr} repeatedly (and componentwisely) to compute the tangential derivatives of $\mathcal{M}_\varepsilon u$ up to the order three, and then apply \eqref{Pf_AHk:Coef}.
  Then, we find that
  \begin{align} \label{Pf_AHk:PW}
    |\nabla_\Gamma^k\mathcal{M}_\varepsilon u| \leq c\sum_{\ell=0}^k\mathcal{M}_\varepsilon(|\nabla^\ell u|) \quad\text{on}\quad \Gamma, \quad k=1,2,3.
  \end{align}
  For $k=1,2,3$, we observe by \eqref{E:ALp_Sur} and \eqref{Pf_AHk:PW} that
  \begin{align} \label{Pf_AHk:L2}
    \|\nabla_\Gamma^k\mathcal{M}_\varepsilon u\|_{L^2(\Gamma)} \leq c\sum_{\ell=0}^k\|\mathcal{M}_\varepsilon(|\nabla^\ell u|)\|_{L^2(\Gamma)} \leq c\varepsilon^{-1/2}\sum_{\ell=0}^k\|\nabla^\ell u\|_{L^2(\Omega_\varepsilon)}.
  \end{align}
  Hence, \eqref{E:AHk_Sur} follows.
  Also, \eqref{E:AHk_TD} holds by \eqref{E:CEGr_Bd}, \eqref{E:Lp_CE}, and \eqref{Pf_AHk:L2}.
\end{proof}

\begin{lemma} \label{L:ATG_L2D}
  Let $u\in H^2(\Omega_\varepsilon)$.
  Then,
  \begin{align} \label{E:ATG_L2D}
    \left\|\overline{\mathbf{P}}\nabla u-\overline{\nabla_\Gamma\mathcal{M}_\varepsilon u}\right\|_{L^2(\Omega_\varepsilon)} \leq c\varepsilon\|u\|_{H^2(\Omega_\varepsilon)}.
  \end{align}
\end{lemma}

\begin{proof}
  Let $\mathbf{v}:=\overline{\mathbf{P}}\nabla u$ on $\Omega_\varepsilon$.
  Then, we see by \eqref{E:ADif_Lp} and \eqref{E:ATG_Diff} that
  \begin{align*}
    \left\|\mathbf{v}-\overline{\nabla_\Gamma\mathcal{M}_\varepsilon u}\right\|_{L^2(\Omega_\varepsilon)} &\leq \left\|\mathbf{v}-\overline{\mathcal{M}_\varepsilon\mathbf{v}}\right\|_{L^2(\Omega_\varepsilon)}+\left\|\overline{\mathcal{M}_\varepsilon\mathbf{v}}-\overline{\nabla_\Gamma\mathcal{M}_\varepsilon u}\right\|_{L^2(\Omega_\varepsilon)} \\
    &\leq c\varepsilon\left(\|\mathbf{v}\|_{H^1(\Omega_\varepsilon)}+\left\|\overline{\mathcal{M}_\varepsilon(|u|+|\nabla u|)}\right\|_{L^2(\Omega_\varepsilon)}\right).
  \end{align*}
  By this inequality, \eqref{E:ALp_TD}, and the regularity of $\mathbf{P}$, we obtain \eqref{E:ATG_L2D}.
\end{proof}

Using \eqref{E:Ave_TGr}, we approximate the Dirichlet form on $\Omega_\varepsilon$ by the weighted one on $\Gamma$.

\begin{lemma} \label{L:ADF_Diff}
  Let $u\in H^1(\Omega_\varepsilon)$ and $\eta\in H^1(\Gamma)$.
  Then,
  \begin{align} \label{E:ADF_Diff}
    \left|\int_{\Omega_\varepsilon}\nabla u\cdot\nabla\bar{\eta}\,dx-\varepsilon\int_\Gamma g\nabla_\Gamma\mathcal{M}_\varepsilon u\cdot\nabla_\Gamma\eta\,d\mathcal{H}^2\right| \leq c\varepsilon^{3/2}\|u\|_{H^1(\Omega_\varepsilon)}\|\nabla_\Gamma\eta\|_{L^2(\Gamma)}.
  \end{align}
\end{lemma}

\begin{proof}
  We split the difference on the left-hand side of \eqref{E:ADF_Diff} into the sum of
  \begin{align*}
    K_1 &:= \int_{\Omega_\varepsilon}\nabla u\cdot\nabla\bar{\eta}\,dx-\int_{\Omega_\varepsilon}\nabla u\cdot\overline{\nabla_\Gamma\eta}\,dx, \\
    K_2 &:= \int_{\Omega_\varepsilon}\nabla u\cdot\overline{\nabla_\Gamma\eta}\,dx-\varepsilon\int_\Gamma g\nabla_\Gamma\mathcal{M}_\varepsilon u\cdot\nabla_\Gamma\eta\,d\mathcal{H}^2.
  \end{align*}
  By \eqref{E:CEGr_Df}, $|d|\leq c\varepsilon$ in $\Omega_\varepsilon$, H\"{o}lder's inequality, and \eqref{E:Lp_CE}, we have
  \begin{align*}
    |K_1| \leq c\varepsilon\|\nabla u\|_{L^2(\Omega_\varepsilon)}\left\|\overline{\nabla_\Gamma\eta}\right\|_{L^2(\Omega_\varepsilon)} \leq c\varepsilon^{3/2}\|\nabla u\|_{L^2(\Omega_\varepsilon)}\|\nabla_\Gamma\eta\|_{L^2(\Gamma)}.
  \end{align*}
  Also, noting that $\nabla_\Gamma\eta=\mathbf{P}\nabla_\Gamma\eta$ and $\mathbf{P}^T=\mathbf{P}$ on $\Gamma$, we see by \eqref{E:Ave_Pair} that
  \begin{align*}
    K_2 &= \int_{\Omega_\varepsilon}\Bigl(\overline{\mathbf{P}}\nabla u\Bigr)\cdot\overline{\nabla_\Gamma\eta}\,dx-\varepsilon\int_\Gamma g\nabla_\Gamma\mathcal{M}_\varepsilon u\cdot\nabla_\Gamma\eta\,d\mathcal{H}^2 \\
    &= \varepsilon\int_\Gamma g\left\{\mathcal{M}_\varepsilon\Bigl(\overline{\mathbf{P}}\nabla u\Bigr)-\nabla_\Gamma\mathcal{M}_\varepsilon u\right\}\cdot\nabla_\Gamma\eta\,d\mathcal{H}^2.
  \end{align*}
  Hence, it follows from \eqref{E:G_Bdd}, \eqref{E:ATG_Diff}, H\"{o}lder's inequality, and \eqref{E:ALp_Sur} that
  \begin{align*}
    |K_2| \leq c\varepsilon^2\|\mathcal{M}_\varepsilon(|u|+|\nabla u|)\|_{L^2(\Gamma)}\|\nabla_\Gamma\eta\|_{L^2(\Gamma)} \leq c\varepsilon^{3/2}\|u\|_{H^1(\Omega_\varepsilon)}\|\nabla_\Gamma\eta\|_{L^2(\Gamma)}.
  \end{align*}
  The inequality \eqref{E:ADF_Diff} follows from the above estimates for $K_1$ and $K_2$.
\end{proof}

Let us express $\mathcal{M}_\varepsilon(\Delta u)$ in terms of $A_g\mathcal{M}_\varepsilon u=g^{-1}\mathrm{div}_\Gamma(g\nabla_\Gamma\mathcal{M}_\varepsilon u)$.

\begin{lemma} \label{L:Ave_Lap}
  Let $u\in H^2(\Omega_\varepsilon)$ satisfy $\partial_{\nu_\varepsilon}u=0$ on $\partial\Omega_\varepsilon$.
  Then,
  \begin{align} \label{E:Ave_Lap}
    \mathcal{M}_\varepsilon(\Delta u) = A_g\mathcal{M}_\varepsilon u+\zeta_\Delta^\varepsilon \quad\text{on}\quad \Gamma,
  \end{align}
  where $\zeta_\Delta^\varepsilon$ is a residual term.
  Moreover, if $u\in H^{k+2}(\Omega_\varepsilon)$ with $k=0,1$, then
  \begin{align} \label{E:ALa_Hk}
    \|\zeta_\Delta^\varepsilon\|_{H^k(\Gamma)} \leq c\varepsilon^{1/2}\|u\|_{H^{2+k}(\Omega_\varepsilon)}.
  \end{align}
\end{lemma}

\begin{proof}
  For $\eta\in H^1(\Gamma)$, we see by \eqref{E:IbP_Ag} and \eqref{E:Ave_TGr} that
  \begin{align*}
    &\int_\Gamma g(A_g\mathcal{M}_\varepsilon u)\eta\,d\mathcal{H}^2 = -\int_\Gamma g\nabla_\Gamma\mathcal{M}_\varepsilon u\cdot\nabla_\Gamma\eta\,d\mathcal{H}^2 = K_1+K_2, \\
    &K_1 := -\int_\Gamma g\mathcal{M}_\varepsilon\Bigl(\overline{\mathbf{P}}\nabla u\Bigr)\cdot\nabla_\Gamma\eta\,d\mathcal{H}^2, \\
    &K_2 := -\int_\Gamma g\mathcal{M}_\varepsilon\Bigl(-d\overline{\mathbf{W}}\nabla u+(\partial_\nu u+uf_J)\mathbf{b}_\varepsilon+u\mathbf{b}_J\Bigr)\cdot\nabla_\Gamma\eta\,d\mathcal{H}^2.
  \end{align*}
  Moreover, by \eqref{E:Ave_Pair}, and by $\mathbf{P}^{\mathrm{T}}=\mathbf{P}$ and $\mathbf{P}\nabla_\Gamma\eta=\nabla_\Gamma\eta$ on $\Gamma$,
  \begin{align*}
    K_1 = -\frac{1}{\varepsilon}\int_{\Omega_\varepsilon}\Bigl(\overline{\mathbf{P}}\nabla u\Bigr)\cdot\overline{\nabla_\Gamma\eta}\,dx = -\frac{1}{\varepsilon}\int_{\Omega_\varepsilon}\nabla u\cdot\overline{\nabla_\Gamma\eta}\,dx.
  \end{align*}
  To carry out integration by parts, we see by \eqref{E:CEGr_NB} that
  \begin{align*}
    \nabla\bar{\eta} = \mathbf{R}\overline{\nabla_\Gamma\eta} = \overline{\nabla_\Gamma\eta}+d\mathbf{R}\overline{\mathbf{W}\nabla_\Gamma\eta} \quad\text{in}\quad \overline{\mathcal{N}_\delta},
  \end{align*}
  since $\mathbf{R}=\mathbf{I}_3+d\mathbf{R}\overline{\mathbf{W}}$ by $\mathbf{R}(\mathbf{I}_3-d\overline{\mathbf{W}})=\mathbf{I}_3$.
  Using this relation, we have
  \begin{align*}
    K_1 &= -\frac{1}{\varepsilon}\int_{\Omega_\varepsilon}\nabla u\cdot\nabla\bar{\eta}\,dx+\int_{\Omega_\varepsilon}\nabla u\cdot\Bigl(d\mathbf{R}\overline{\mathbf{W}\nabla_\Gamma\eta}\Bigr)\,dx \\
    &= \frac{1}{\varepsilon}\int_{\Omega_\varepsilon}(\Delta u)\bar{\eta}\,dx+\frac{1}{\varepsilon}\int_{\Omega_\varepsilon}\Bigl(d\overline{\mathbf{W}}\mathbf{R}\nabla u\Bigr)\cdot\overline{\nabla_\Gamma\eta}\,dx \\
    &= \int_\Gamma g\{\mathcal{M}_\varepsilon(\Delta u)\}\eta\,d\mathcal{H}^2+\int_\Gamma g\mathcal{M}_\varepsilon\Bigl(d\overline{\mathbf{W}}\mathbf{R}\nabla u\Bigr)\cdot\nabla_\Gamma\eta\,d\mathcal{H}^2
  \end{align*}
  by integration by parts, $\partial_{\nu_\varepsilon}u=0$ on $\partial\Omega_\varepsilon$, $\mathbf{R}^{\mathrm{T}}=\mathbf{R}$, $\mathbf{W}^{\mathrm{T}}=\mathbf{W}$, and \eqref{E:Ave_Pair}.
  Thus,
  \begin{align} \label{Pf_ALa:AgM}
      \int_\Gamma g(A_g\mathcal{M}_\varepsilon u)\eta\,d\mathcal{H}^2 = \int_\Gamma g\{\mathcal{M}_\varepsilon(\Delta u)\}\eta\,d\mathcal{H}^2+\int_\Gamma g\mathcal{M}_\varepsilon(d\mathbf{v}_1+\varepsilon\mathbf{v}_2)\cdot\nabla_\Gamma\eta\,d\mathcal{H}^2
  \end{align}
  for all $\eta\in H^1(\Gamma)$.
  Here, we take vector fields $\mathbf{v}_1,\mathbf{v}_2\colon\Omega_\varepsilon\to\mathbb{R}^3$ such that
  \begin{align*}
    d\overline{\mathbf{W}}\mathbf{R}\nabla u+d\overline{\mathbf{W}}\nabla u-(\partial_\nu u+uf_J)\mathbf{b}_\varepsilon-u\mathbf{b}_J = d\mathbf{v}_1+\varepsilon\mathbf{v}_2 \quad\text{in}\quad \Omega_\varepsilon.
  \end{align*}
  More precisely, using \eqref{E:Def_Bfb}, we define
  \begin{align} \label{Pf_ALa:v12}
    \begin{aligned}
      \mathbf{v}_1 &:= \overline{\mathbf{W}}(\mathbf{R}+\mathbf{I}_3)\nabla u-\frac{\partial_\nu u+uf_J}{\bar{g}}\,\overline{\nabla_\Gamma g}+\frac{u}{J_{\pi,d}}\Bigl(\overline{\nabla_\Gamma H}-d\overline{\nabla_\Gamma K}\Bigr), \\
      \mathbf{v}_2 &:= \frac{\partial_\nu u+uf_J}{\bar{g}}\Bigl(\bar{g}_0\overline{\nabla_\Gamma g_1}-\bar{g}_1\overline{\nabla_\Gamma g_0}\Bigr).
    \end{aligned}
  \end{align}
  We see that $\mathbf{v}_i\cdot\bar{\bm{\nu}}=0$ in $\Omega_\varepsilon$ for $i=1,2$, since $\mathbf{W}^{\mathrm{T}}=\mathbf{W}$, $\mathbf{W}\bm{\nu}=0$, and $\nabla_\Gamma\zeta\cdot\bm{\nu}=0$ on $\Gamma$ for a function $\zeta$ on $\Gamma$.
  Hence, we carry out integration by parts \eqref{E:IbP_Sur} to get
  \begin{align} \label{Pf_ALa:IbP}
    \int_\Gamma g\mathcal{M}_\varepsilon(d\mathbf{v}_1+\varepsilon\mathbf{v}_2)\cdot\nabla_\Gamma\eta\,d\mathcal{H}^2 = -\int_\Gamma\{\mathrm{div}_\Gamma[g\mathcal{M}_\varepsilon(d\mathbf{v}_1+\varepsilon\mathbf{v}_2)]\}\eta\,d\mathcal{H}^2.
  \end{align}
  Since $H^1(\Gamma)$ is dense in $L^2(\Gamma)$, we deduce from \eqref{Pf_ALa:AgM} and \eqref{Pf_ALa:IbP} that
  \begin{align*}
    gA_g\mathcal{M}_\varepsilon u = g\mathcal{M}_\varepsilon(\Delta u)-\mathrm{div}_\Gamma[g\mathcal{M}_\varepsilon(d\mathbf{v}_1+\varepsilon\mathbf{v}_2)] \quad\text{on}\quad \Gamma.
  \end{align*}
  Hence, we find that \eqref{E:Ave_Lap} is valid if we set
  \begin{align*}
    \zeta_\Delta^\varepsilon := \frac{1}{g}\,\mathrm{div}_\Gamma[g\mathcal{M}_\varepsilon(d\mathbf{v}_1)]+\frac{\varepsilon}{g}\,\mathrm{div}_\Gamma(g\mathcal{M}_\varepsilon\mathbf{v}_2) \quad\text{on}\quad \Gamma.
  \end{align*}
  To estimate $\zeta_\Delta^\varepsilon$, we apply \eqref{E:Ave_dphi} given below to $\mathbf{v}_1$ and \eqref{Pf_AHk:PW} to $\mathbf{v}_2$, use \eqref{E:G_Bdd} and \eqref{E:J_Bdd}, and note that the functions in \eqref{Pf_ALa:v12} except for $u$ are independent of $\varepsilon$ and of class $C^2$ on $\Gamma$ or on $\overline{\mathcal{N}_\delta}$.
  Then, for $k=0,1,2$, we find that
  \begin{align*}
    |\nabla_\Gamma^k\mathcal{M}_\varepsilon(d\mathbf{v}_1)| &\leq c\varepsilon\sum_{\ell=0}^k\mathcal{M}_\varepsilon(|\nabla^\ell\mathbf{v}_1|) \leq c\varepsilon\sum_{\ell=0}^{k+1}\mathcal{M}_\varepsilon(|\nabla^\ell u|), \\
    |\nabla_\Gamma^k\mathcal{M}_\varepsilon\mathbf{v}_2| &\leq c\sum_{\ell=0}^k\mathcal{M}_\varepsilon(|\nabla^\ell\mathbf{v}_2|) \leq c\sum_{\ell=0}^{k+1}\mathcal{M}_\varepsilon(|\nabla^\ell u|)
  \end{align*}
  on $\Gamma$.
  Applying these estimates to $\zeta_\Delta^\varepsilon$ and $\nabla_\Gamma\zeta_\Delta^\varepsilon$, we get
  \begin{align*}
    |\nabla_\Gamma^k\zeta_\Delta^\varepsilon| \leq c\varepsilon\sum_{\ell=0}^{k+2}\mathcal{M}_\varepsilon(|\nabla^\ell u|) \quad\text{on}\quad \Gamma, \quad k=0,1.
  \end{align*}
  By this estimate and \eqref{E:ALp_Sur}, we obtain \eqref{E:ALa_Hk}.
\end{proof}

\begin{lemma} \label{L:Ave_dphi}
  Let $u\in H^k(\Omega_\varepsilon)$ with $k=0,1,2$.
  Then,
  \begin{align} \label{E:Ave_dphi}
    |\nabla_\Gamma^k\mathcal{M}_\varepsilon(du)| \leq c\varepsilon\sum_{\ell=0}^k\mathcal{M}_\varepsilon(|\nabla^\ell u|) \quad\text{on}\quad \Gamma.
  \end{align}
\end{lemma}

\begin{proof}
  The inequality \eqref{E:Ave_dphi} with $k=0$ just follows from $|d|\leq c\varepsilon$ in $\Omega_\varepsilon$.

  Let $k=1$.
  Since $\nabla d=\bar{\bm{\nu}}$ in $\overline{\mathcal{N}_\delta}$ by \eqref{E:Fermi}, we have $\overline{\mathbf{P}}\nabla d=0$ in $\overline{\mathcal{N}_\delta}$.
  Hence, setting
  \begin{align*}
    \mathbf{u}_1 &:= \overline{\mathbf{P}}\nabla u-\overline{\mathbf{W}}\nabla(du)+\frac{\partial_\nu(du)+duf_J}{\bar{g}}\,\overline{\nabla_\Gamma g}+\frac{du}{J^\flat}\Bigl(-\overline{\nabla_\Gamma H}+d\overline{\nabla_\Gamma K}\Bigr), \\
    \mathbf{u}_2 &:= \frac{\partial_\nu(du)+duf_J}{\bar{g}}\Bigl(-\bar{g}_0\overline{\nabla_\Gamma g_1}+\bar{g}_1\overline{\nabla_\Gamma g_0}\Bigr),
  \end{align*}
  and using \eqref{E:Ave_TGr} with $u$ replaced by $du$, we find that
  \begin{align} \label{Pf_Adp:TGr}
    \nabla_\Gamma\mathcal{M}_\varepsilon(du) = \mathcal{M}_\varepsilon(d\mathbf{u}_1)+\varepsilon\mathcal{M}_\varepsilon\mathbf{u}_2 \quad\text{on}\quad \Gamma,
  \end{align}
  and \eqref{E:Ave_dphi} holds when $k=1$ by $|d|\leq c\varepsilon$ and $|\mathbf{u}_i|\leq c(|u|+|\nabla u|)$ in $\Omega_\varepsilon$ for $i=1,2$.

  Let $k=2$.
  We take $\nabla_\Gamma$ of \eqref{Pf_Adp:TGr}, and apply \eqref{Pf_Adp:TGr} to $\mathbf{u}_1$ and \eqref{E:Ave_TGr} to $\mathbf{u}_2$.
  Then,
  \begin{align*}
    \nabla_\Gamma^2\mathcal{M}_\varepsilon(du) = \mathcal{M}_\varepsilon(d\mathbf{A}_1)+\varepsilon\mathcal{M}_\varepsilon\mathbf{A}_2 \quad\text{on}\quad \Gamma,
  \end{align*}
  where $\mathbf{A}_1$ and $\mathbf{A}_2$ are $3\times 3$ matrix-valued functions on $\Omega_\varepsilon$ that satisfy
  \begin{align*}
    |\mathbf{A}_i| \leq c(|u|+|\nabla u|+|\nabla^2u|) \quad\text{in}\quad \Omega_\varepsilon, \quad i=1,2.
  \end{align*}
  By these relations and $|d|\leq c\varepsilon$ in $\Omega_\varepsilon$, we obtain \eqref{E:Ave_dphi} with $k=2$.
\end{proof}

\subsection{Average of the nonlinear term} \label{SS:Ave_NL}
For $u\colon\Omega_\varepsilon\to\mathbb{R}$ and $G\colon\mathbb{R}\to\mathbb{R}$, let
\begin{align} \label{E:Def_ZG}
  \zeta_G^\varepsilon := \mathcal{M}_\varepsilon(G(u))-G(\mathcal{M}_\varepsilon u) \quad\text{on}\quad \Gamma.
\end{align}
This kind of commutator appears when we take the average of the thin-domain problem \eqref{E:CH_CTD}, where $G=F'$.
Let us estimate $\zeta_G^\varepsilon$.
For the sake of simplicity, we write
\begin{align*}
  v^\varepsilon := \mathcal{M}_\varepsilon u \quad\text{on}\quad \Gamma, \quad U_1 := |u|+|\bar{v}^\varepsilon|, \quad U_2 := |u|+|\nabla u| \quad\text{in}\quad \Omega_\varepsilon.
\end{align*}

\begin{lemma} \label{L:ANL_Diff}
  Suppose that
  \begin{align} \label{E:ANL_G1}
    G\in C^1(\mathbb{R}), \quad |G'(z)| \leq c(|z|^2+1) \quad\text{for all}\quad z\in\mathbb{R}.
  \end{align}
  Let $u\in W^{1,p}(\Omega_\varepsilon)$ with $p\in[1,\infty]$.
  Then,
  \begin{align} \label{E:ANL_Diff}
    |G(u)-G(\bar{v}^\varepsilon)| \leq c\varepsilon(U_1^2+1)\overline{\mathcal{M}_\varepsilon U_2} \quad\text{in}\quad \Omega_\varepsilon.
  \end{align}
\end{lemma}

\begin{proof}
  By the mean value theorem for $G(z)$, we can write
  \begin{align*}
    |G(u)-G(\bar{v}^\varepsilon)| = |G'(\lambda u+(1-\lambda)\bar{v}^\varepsilon)|\cdot|u-\bar{v}^\varepsilon|
  \end{align*}
  with some $\lambda\in(0,1)$.
  We apply \eqref{E:ADif_PW} and \eqref{E:ANL_G1} to the right-hand side and use
  \begin{align*}
    |\lambda u+(1-\lambda)\bar{v}^\varepsilon| \leq |u|+|\bar{v}^\varepsilon| = U_1, \quad |\partial_\nu u| = |\bar{\bm{\nu}}\cdot\nabla u| \leq |\nabla u| \quad\text{in}\quad \Omega_\varepsilon
  \end{align*}
  to obtain \eqref{E:ANL_Diff}.
\end{proof}

\begin{lemma} \label{L:Gu_L2}
  Suppose that \eqref{E:ANL_G1} holds.
  Let $u\in H^1(\Omega_\varepsilon)$.
  Then,
  \begin{align} \label{E:Gu_L2}
    \|G(u)\|_{L^2(\Omega_\varepsilon)} &\leq c\Bigl(\varepsilon^{-1}\|u\|_{H^1(\Omega_\varepsilon)}^3+\varepsilon^{1/2}\Bigr).
  \end{align}
  Moreover, if $u\in H^2(\Omega_\varepsilon)$, then
  \begin{align} \label{E:NaGu_L2}
    \|\nabla G(u)\|_{L^2(\Omega_\varepsilon)} &\leq c\Bigl(\varepsilon^{-1}\|u\|_{H^1(\Omega_\varepsilon)}^2\|u\|_{H^2(\Omega_\varepsilon)}+\|u\|_{H^1(\Omega_\varepsilon)}\Bigr).
  \end{align}
\end{lemma}

\begin{proof}
  As in Remark \ref{R:Po_Grow}, we have $|G(z)|\leq c(|z|^3+1)$ for all $z\in\mathbb{R}$ by \eqref{E:ANL_G1}.
  Hence,
  \begin{align*}
    \|G(u)\|_{L^2(\Omega_\varepsilon)} &\leq c\Bigl(\|u^3\|_{L^2(\Omega_\varepsilon)}+|\Omega_\varepsilon|^{1/2}\Bigr) = c\Bigl(\|u\|_{L^6(\Omega_\varepsilon)}^3+|\Omega_\varepsilon|^{1/2}\Bigr),
  \end{align*}
  and we apply \eqref{E:Sob_CTD} and $|\Omega_\varepsilon|\leq c\varepsilon$ to the right-hand side to get \eqref{E:Gu_L2}.
  Also, since
  \begin{align*}
    |\nabla G(u)| = |G'(u)\nabla u| \leq c(|u|^2+1)|\nabla u| \quad\text{in}\quad \Omega_\varepsilon
  \end{align*}
  by \eqref{E:ANL_G1}, we observe by H\"{o}lder's inequality that
  \begin{align*}
    \|\nabla G(u)\|_{L^2(\Omega_\varepsilon)} &\leq c\Bigl(\|u^2\nabla u\|_{L^2(\Omega_\varepsilon)}+\|\nabla u\|_{L^2(\Omega_\varepsilon)}\Bigr) \\
    &\leq c\Bigl(\|u\|_{L^6(\Omega_\varepsilon)}^2\|\nabla u\|_{L^6(\Omega_\varepsilon)}+\|\nabla u\|_{L^2(\Omega_\varepsilon)}\Bigr).
  \end{align*}
  Hence, we get \eqref{E:NaGu_L2} by applying \eqref{E:Sob_CTD} to the second line.
\end{proof}

\begin{lemma} \label{L:ANL_L2}
  Suppose that \eqref{E:ANL_G1} holds.
  Let $u\in H^2(\Omega_\varepsilon)$.
  Then,
  \begin{align} \label{E:ANL_L2}
    \|\zeta_G^\varepsilon\|_{L^2(\Gamma)} \leq c\varepsilon\Bigl(\varepsilon^{-3/2}\|u\|_{H^1(\Omega_\varepsilon)}^2\|u\|_{H^2(\Omega_\varepsilon)}+1\Bigr)
  \end{align}
  for $\zeta_G^\varepsilon=\mathcal{M}_\varepsilon(G(u))-G(v^\varepsilon)$ with $v^\varepsilon=\mathcal{M}_\varepsilon u$.
\end{lemma}

\begin{proof}
  We use the notation \eqref{E:Pull_CTD} and suppress the variable $y\in\Gamma$.
  Let
  \begin{align*}
    K_1 &:= \frac{1}{\varepsilon g}\int_{\varepsilon g_0}^{\varepsilon g_1}[G(u)]^\sharp(r)\{J(r)-1\}\,dr, \\
    K_2 &:= \frac{1}{\varepsilon g}\int_{\varepsilon g_0}^{\varepsilon g_1}\{[G(u)]^\sharp(r)-G(v^\varepsilon)\}\,dr.
  \end{align*}
  Then, $\zeta_G^\varepsilon=K_1+K_2$ on $\Gamma$ since $G(v^\varepsilon)$ is independent of the variable $r$.
  Moreover,
  \begin{align*}
    |K_1| &\leq c\varepsilon\cdot\frac{1}{\varepsilon g}\int_{\varepsilon g_0}^{\varepsilon g_1}|[G(u)]^\sharp(r)|\,dr, \\
    |K_2| &\leq \frac{1}{\varepsilon g}\int_{\varepsilon g_0}^{\varepsilon g_1}c\varepsilon([U_1^2]^\sharp(r)+1)\mathcal{M}_\varepsilon U_2\,dr = c\varepsilon\left(\frac{1}{\varepsilon g}\int_{\varepsilon g_0}^{\varepsilon g_1}[U_1^2]^\sharp(r)\,dr+1\right)\mathcal{M}_\varepsilon U_2
  \end{align*}
  by \eqref{E:J_Diff} and \eqref{E:ANL_Diff}.
  Noting that $|G(u)|,U_1^2,U_2\geq0$, we further use \eqref{E:J_Bdd} to get
  \begin{align*}
    |K_1| \leq c\varepsilon\mathcal{M}_\varepsilon(|G(u)|), \quad |K_2| \leq c\varepsilon(\mathcal{M}_\varepsilon(U_1^2)+1)\mathcal{M}_\varepsilon U_2 \quad\text{on}\quad \Gamma.
  \end{align*}
  Hence, it follows from $\zeta_G^\varepsilon=K_1+K_2$ on $\Gamma$ that
  \begin{align} \label{Pf_AN2:L2}
    \|\zeta_G^\varepsilon\|_{L^2(\Gamma)} \leq c\varepsilon\Bigl(\|\mathcal{M}_\varepsilon(|G(u)|)\|_{L^2(\Gamma)}+\|\mathcal{M}_\varepsilon(U_1^2)\mathcal{M}_\varepsilon U_2\|_{L^2(\Gamma)}+\|\mathcal{M}_\varepsilon U_2\|_{L^2(\Gamma)}\Bigr).
  \end{align}
  Let us estimate the right-hand side.
  By \eqref{E:ALp_Sur} and \eqref{E:Gu_L2}, we have
  \begin{align} \label{Pf_AN2:AGu}
    \|\mathcal{M}_\varepsilon(|G(u)|)\|_{L^2(\Gamma)} \leq c\varepsilon^{-1/2}\|G(u)\|_{L^2(\Omega_\varepsilon)} \leq c\Bigl(\varepsilon^{-3/2}\|u\|_{H^1(\Omega_\varepsilon)}^3+1\Bigr).
  \end{align}
  Since $U_2=|u|+|\nabla u|$ in $\Omega_\varepsilon$, we see by \eqref{E:ALp_Sur} that
  \begin{align} \label{Pf_AN2:AU2}
    \|\mathcal{M}_\varepsilon U_2\|_{L^2(\Gamma)} \leq c\varepsilon^{-1/2}\|U_2\|_{L^2(\Omega_\varepsilon)} \leq c\varepsilon^{-1/2}\|u\|_{H^1(\Omega_\varepsilon)}.
  \end{align}
  We use H\"{o}lder's inequality, \eqref{E:ALp_Sur}, and $\|U_1^2\|_{L^3(\Omega_\varepsilon)}=\|U_1\|_{L^6(\Omega_\varepsilon)}^2$ to get
  \begin{align} \label{Pf_AN2:Tri}
    \begin{aligned}
      \|\mathcal{M}_\varepsilon(U_1^2)\mathcal{M}_\varepsilon U_2\|_{L^2(\Gamma)} &\leq \|\mathcal{M}_\varepsilon(U_1^2)\|_{L^3(\Gamma)}\|\mathcal{M}_\varepsilon U_2\|_{L^6(\Gamma)} \\
      &\leq c\varepsilon^{-1/2}\|U_1\|_{L^6(\Omega_\varepsilon)}^2\|U_2\|_{L^6(\Omega_\varepsilon)}.
    \end{aligned}
  \end{align}
  Moreover, since $U_1=|u|+|\bar{v}^\varepsilon|$ in $\Omega_\varepsilon$ with $v^\varepsilon=\mathcal{M}_\varepsilon u$, we see that
  \begin{align} \label{Pf_AN2:U16}
    \|U_1\|_{L^6(\Omega_\varepsilon)} \leq \|u\|_{L^6(\Omega_\varepsilon)}+\|\bar{v}^\varepsilon\|_{L^6(\Omega_\varepsilon)} \leq c\|u\|_{L^6(\Omega_\varepsilon)} \leq c\varepsilon^{-1/3}\|u\|_{H^1(\Omega_\varepsilon)}
  \end{align}
  by \eqref{E:ALp_TD} and then by \eqref{E:Sob_CTD}.
  Also, by $U_2=|u|+|\nabla u|$ in $\Omega_\varepsilon$ and \eqref{E:Sob_CTD},
  \begin{align*}
    \|U_2\|_{L^6(\Omega_\varepsilon)} \leq \|u\|_{L^6(\Omega_\varepsilon)}+\|\nabla u\|_{L^6(\Omega_\varepsilon)} \leq c\varepsilon^{-1/3}\|u\|_{H^2(\Omega_\varepsilon)}.
  \end{align*}
  Combining the above estimates, we find that
  \begin{align} \label{Pf_AN2:cu}
    \|\mathcal{M}_\varepsilon(U_1^2)\mathcal{M}_\varepsilon U_2\|_{L^2(\Gamma)} \leq c\varepsilon^{-3/2}\|u\|_{H^1(\Omega_\varepsilon)}^2\|u\|_{H^2(\Omega_\varepsilon)}.
  \end{align}
  Thus, we get \eqref{E:ANL_L2} by applying \eqref{Pf_AN2:AGu}, \eqref{Pf_AN2:AU2}, and \eqref{Pf_AN2:cu} to \eqref{Pf_AN2:L2} and using
  \begin{align} \label{Pf_AN2:Yo}
    \|u\|_{H^1(\Omega_\varepsilon)} = \varepsilon^{-1/3}\|u\|_{H^1(\Omega_\varepsilon)}\cdot\varepsilon^{1/3} \leq c\Bigl(\varepsilon^{-1}\|u\|_{H^1(\Omega_\varepsilon)}^3+\varepsilon^{1/2}\Bigr)
  \end{align}
  and  $\|u\|_{H^1(\Omega_\varepsilon)}\leq\|u\|_{H^2(\Omega_\varepsilon)}$, where we used Young's inequality in \eqref{Pf_AN2:Yo}.
\end{proof}

Next, we estimate $\nabla_\Gamma\zeta_G^\varepsilon$.
To this end, we prepare an auxiliary result.

\begin{lemma} \label{L:Prod_L3}
  Let $\eta\in H^2(\Gamma)$ and $u\in H^2(\Omega_\varepsilon)$.
  Then,
  \begin{align} \label{E:Prod_L3}
    \|\bar{\eta}u\|_{L^3(\Omega_\varepsilon)} \leq c\varepsilon^{-1/6}\|\eta\|_{L^2(\Gamma)}^{2/3}\|\eta\|_{H^2(\Gamma)}^{1/3}\|u\|_{L^2(\Omega_\varepsilon)}^{5/12}\|u\|_{H^1(\Omega_\varepsilon)}^{1/3}\|u\|_{H^2(\Omega_\varepsilon)}^{1/4}.
  \end{align}
\end{lemma}

\begin{proof}
  We see by \eqref{E:Ave_Pair} that
  \begin{align*}
    \|\bar{\eta}u\|_{L^3(\Omega_\varepsilon)} = \left(\int_{\Omega_\varepsilon}|\bar{\eta}|^3|u|^3\,dx\right)^{1/3} = \left(\varepsilon\int_\Gamma g|\eta|^3\mathcal{M}_\varepsilon(|u|^3)\,d\mathcal{H}^2\right)^{1/3}.
  \end{align*}
  Then, we use \eqref{E:G_Bdd}, H\"{o}lder's inequality, \eqref{E:SoSu_W11}, and \eqref{E:Ipl_Sur} with $p=6$ to get
  \begin{align} \label{Pf_P3:L3}
    \begin{aligned}
      \|\bar{\eta}u\|_{L^3(\Omega_\varepsilon)} &\leq c\varepsilon^{1/3}\|\eta\|_{L^6(\Gamma)}\|\mathcal{M}_\varepsilon(|u|^3)\|_{L^2(\Gamma)}^{1/3} \\
      &\leq c\varepsilon^{1/3}\|\eta\|_{L^2(\Gamma)}^{2/3}\|\eta\|_{H^2(\Gamma)}^{1/3}\|\mathcal{M}_\varepsilon(|u|^3)\|_{W^{1,1}(\Gamma)}^{1/3}.
    \end{aligned}
  \end{align}
  Since $|\nabla_\Gamma\mathcal{M}_\varepsilon(|u|^3)|\leq c\{\mathcal{M}_\varepsilon(|u|^2|\nabla u|)+\mathcal{M}_\varepsilon(|u|^3)\}$ on $\Gamma$ by \eqref{E:Ave_TGr} and \eqref{Pf_ATG:coef}, we have
  \begin{align*}
    \|\mathcal{M}_\varepsilon(|u|^3)\|_{W^{1,1}(\Gamma)} &\leq c\Bigl(\|\mathcal{M}_\varepsilon(|u|^2|\nabla u|)\|_{L^1(\Gamma)}+\|\mathcal{M}_\varepsilon(|u|^3)\|_{L^1(\Gamma)}\Bigr) \\
    &\leq c\varepsilon^{-1}\Bigl(\bigl\|\,|u|^2|\nabla u|\,\bigr\|_{L^1(\Omega_\varepsilon)}+\bigl\|\,|u|^3\,\bigr\|_{L^1(\Omega_\varepsilon)}\Bigr) \\
    &\leq c\varepsilon^{-1}\bigl\|\,|u|^2\,\bigr\|_{L^2(\Omega_\varepsilon)}\Bigl(\|\nabla u\|_{L^2(\Omega_\varepsilon)}+\|u\|_{L^2(\Omega_\varepsilon)}\Bigr)
  \end{align*}
  by \eqref{E:ALp_Sur} and H\"{o}lder's inequality.
  Moreover, by \eqref{E:Ipl_CTD} with $p=4$,
  \begin{align*}
    \bigl\|\,|u|^2\,\bigr\|_{L^2(\Omega_\varepsilon)} = \|u\|_{L^4(\Omega_\varepsilon)}^2 \leq c\varepsilon^{-1/2}\|u\|_{L^2(\Omega_\varepsilon)}^{5/4}\|u\|_{H^2(\Omega_\varepsilon)}^{3/4}.
  \end{align*}
  By this inequality and $\|\nabla u\|_{L^2(\Omega_\varepsilon)}+\|u\|_{L^2(\Omega_\varepsilon)}\leq 2\|u\|_{H^1(\Omega_\varepsilon)}$, we find that
  \begin{align*}
    \|\mathcal{M}_\varepsilon(|u|^3)\|_{W^{1,1}(\Gamma)} \leq c\varepsilon^{-3/2}\|u\|_{L^2(\Omega_\varepsilon)}^{5/4}\|u\|_{H^1(\Omega_\varepsilon)}\|u\|_{H^2(\Omega_\varepsilon)}^{3/4}.
  \end{align*}
  Applying this inequality to \eqref{Pf_P3:L3}, we obtain \eqref{E:Prod_L3}.
\end{proof}

\begin{lemma} \label{L:ANL_TGr}
  Suppose that
  \begin{align} \label{E:ANL_G2}
    G\in C^2(\mathbb{R}), \quad |G''(z)| \leq c(|z|+1) \quad\text{for all}\quad z\in\mathbb{R}.
  \end{align}
  Let $u\in H^3(\Omega_\varepsilon)$ and $\zeta_G^\varepsilon=\mathcal{M}_\varepsilon(G(u))-G(v^\varepsilon)$ on $\Gamma$ with $v^\varepsilon=\mathcal{M}_\varepsilon u$.
  Then,
  \begin{align} \label{E:ANL_TGr}
    \begin{aligned}
      &\|\nabla_\Gamma\zeta_G^\varepsilon\|_{L^2(\Gamma)} \leq c\varepsilon\{\varepsilon^{-3/2}\sigma_\varepsilon(u)+1\}, \\
      &\sigma_\varepsilon(u) := \Bigl(\|u\|_{H^1(\Omega_\varepsilon)}+\varepsilon^{1/2}\Bigr)\|u\|_{H^1(\Omega_\varepsilon)}\|u\|_{H^2(\Omega_\varepsilon)}^{5/12}\|u\|_{H^3(\Omega_\varepsilon)}^{7/12}.
    \end{aligned}
  \end{align}
\end{lemma}

\begin{proof}
  As in Remark \ref{R:Po_Grow}, we can derive \eqref{E:ANL_G1} from \eqref{E:ANL_G2}.
  Thus, we can apply Lemmas \ref{L:ANL_Diff} and \ref{L:Gu_L2}.
  We use the functions given in \eqref{E:Def_Bfb} and define
  \begin{align*}
    \mathbf{v}_G := -d\overline{\mathbf{W}}\nabla G(u)+\{\partial_\nu G(u)+G(u)f_J\}\mathbf{b}_\varepsilon+G(u)\mathbf{b}_J \quad \text{on}\quad \Omega_\varepsilon.
  \end{align*}
  Then, it follows from \eqref{E:Ave_TGr} that
  \begin{align*}
    \nabla_\Gamma\mathcal{M}_\varepsilon(G(u)) = \mathcal{M}_\varepsilon\Bigl(\overline{\mathbf{P}}\nabla G(u)\Bigr)+\mathcal{M}_\varepsilon\mathbf{v}_G = \mathcal{M}_\varepsilon\Bigl(G'(u)\overline{\mathbf{P}}\nabla u\Bigr)+\mathcal{M}_\varepsilon\mathbf{v}_G \quad\text{on}\quad \Gamma.
  \end{align*}
  Also, $\nabla_\Gamma G(v^\varepsilon)=G'(v^\varepsilon)\nabla_\Gamma v^\varepsilon$ on $\Gamma$.
  Hence,
  \begin{align*}
    \nabla_\Gamma\zeta_G^\varepsilon = \mathcal{M}_\varepsilon\Bigl(G'(u)\overline{\mathbf{P}}\nabla u\Bigr)+\mathcal{M}_\varepsilon\mathbf{v}_G-G'(v^\varepsilon)\nabla_\Gamma v^\varepsilon \quad\text{on}\quad \Gamma.
  \end{align*}
  Moreover, $G'(v^\varepsilon)\mathcal{M}_\varepsilon\varphi=\mathcal{M}_\varepsilon(G'(\bar{v}^\varepsilon)\varphi)$ for a function $\varphi$ on $\Omega_\varepsilon$, since $G'(v^\varepsilon)$ is independent of the variable $r$.
  Thus, we can write
  \begin{align} \label{Pf_ANG:TGr}
    \nabla_\Gamma\zeta_G^\varepsilon = L_1+L_2+\mathcal{M}_\varepsilon\mathbf{v}_G \quad\text{on}\quad \Gamma,
  \end{align}
  where
  \begin{align*}
    L_1 := \mathcal{M}_\varepsilon\Bigl(\{G'(u)-G'(\bar{v}^\varepsilon)\}\overline{\mathbf{P}}\nabla u\Bigr), \quad L_2 := G'(v^\varepsilon)\left\{\mathcal{M}_\varepsilon\Bigl(\overline{\mathbf{P}}\nabla u\Bigr)-\nabla_\Gamma v^\varepsilon\right\}.
  \end{align*}
  Let us estimate $\mathcal{M}_\varepsilon\mathbf{v}_G$.
  Since $|\mathbf{v}_G|\leq c\varepsilon(|G(u)|+|\nabla G(u)|)$ in $\Omega_\varepsilon$ by \eqref{Pf_ATG:coef}, we have
  \begin{align} \label{Pf_ANG:MV}
    \begin{aligned}
      \|\mathcal{M}_\varepsilon\mathbf{v}_G\|_{L^2(\Gamma)} &\leq c\varepsilon^{-1/2}\|\mathbf{v}_G\|_{L^2(\Omega_\varepsilon)} \leq c\varepsilon^{1/2}\|G(u)\|_{H^1(\Omega_\varepsilon)} \\
      &\leq c\varepsilon\Bigl(\varepsilon^{-3/2}\|u\|_{H^1(\Omega_\varepsilon)}^2\|u\|_{H^2(\Omega_\varepsilon)}+1\Bigr)
    \end{aligned}
  \end{align}
  by \eqref{E:ALp_Sur}, \eqref{E:Gu_L2}, \eqref{E:NaGu_L2}, \eqref{Pf_AN2:Yo}, and $\|u\|_{H^1(\Omega_\varepsilon)}\leq\|u\|_{H^2(\Omega_\varepsilon)}$.

  Next, we estimate $L_2$.
  We see by $v^\varepsilon=\mathcal{M}_\varepsilon u$, \eqref{E:ATG_Diff}, and \eqref{E:ANL_G1} that
  \begin{align*}
    |L_2|\leq c\varepsilon(|v^\varepsilon|^2+1)\mathcal{M}_\varepsilon(|u|+|\nabla u|) = c\varepsilon(|v^\varepsilon|^2+1)\mathcal{M}_\varepsilon U_2 \quad\text{on}\quad \Gamma.
  \end{align*}
  From this inequality and H\"{o}lder's inequality, it follows that
  \begin{align*}
    \|L_2\|_{L^2(\Gamma)} &\leq c\varepsilon\Bigl(\|(v^\varepsilon)^2\mathcal{M}_\varepsilon U_2\|_{L^2(\Gamma)}+\|\mathcal{M}_\varepsilon U_2\|_{L^2(\Gamma)}\Bigr) \\
    &\leq c\varepsilon\Bigl(\|v^\varepsilon\|_{L^6(\Gamma)}^2\|\mathcal{M}_\varepsilon U_2\|_{L^6(\Gamma)}+\|\mathcal{M}_\varepsilon U_2\|_{L^2(\Gamma)}\Bigr).
  \end{align*}
  Moreover, by $v^\varepsilon=\mathcal{M}_\varepsilon u$ on $\Gamma$, $U_2=|u|+|\nabla u|$ in $\Omega_\varepsilon$, \eqref{E:Sob_CTD}, and \eqref{E:ALp_Sur}, we have
  \begin{align*}
    \|v^\varepsilon\|_{L^6(\Gamma)} &\leq c\varepsilon^{-1/6}\|u\|_{L^6(\Omega_\varepsilon)} \leq c\varepsilon^{-1/2}\|u\|_{H^1(\Omega_\varepsilon)}, \\
    \|\mathcal{M}_\varepsilon U_2\|_{L^6(\Gamma)} &\leq c\varepsilon^{-1/6}\|U_2\|_{L^6(\Omega_\varepsilon)} \leq c\varepsilon^{-1/2}\|u\|_{H^2(\Omega_\varepsilon)}.
  \end{align*}
  We see by these estimates, \eqref{Pf_AN2:AU2}, \eqref{Pf_AN2:Yo}, and $\|u\|_{H^1(\Omega_\varepsilon)}\leq\|u\|_{H^2(\Omega_\varepsilon)}$ that
  \begin{align} \label{Pf_ANG:L2}
    \|L_2\|_{L^2(\Gamma)} \leq c\varepsilon\Bigl(\varepsilon^{-3/2}\|u\|_{H^1(\Omega_\varepsilon)}^2\|u\|_{H^2(\Omega_\varepsilon)}+1\Bigr).
  \end{align}
  Lastly, let us estimate $L_1$.
  As in the proof of \eqref{E:ANL_Diff}, we have
  \begin{align*}
    |G'(u)-G'(\bar{v}^\varepsilon)| \leq c\varepsilon(U_1+1)\overline{\mathcal{M}_\varepsilon U_2} \quad\text{in}\quad \Omega_\varepsilon
  \end{align*}
  by the mean value theorem for $G'(z)$, \eqref{E:ADif_PW}, and \eqref{E:ANL_G2}.
  Also,
  \begin{align*}
    \left|\overline{\mathbf{P}}\nabla u\right| \leq |\nabla u| \leq |u|+|\nabla u| = U_2 \quad\text{in}\quad \Omega_\varepsilon.
  \end{align*}
  By the above inequalities, we find that
  \begin{align} \label{Pf_ANG:L1P}
    |L_1| \leq c\varepsilon\mathcal{M}_\varepsilon\Bigl[(U_1+1)\Bigl(\overline{\mathcal{M}_\varepsilon U_2}\Bigr)U_2\Bigr] \quad\text{on}\quad \Gamma.
  \end{align}
  on $\Gamma$.
  To estimate the $L^2(\Gamma)$-norm of the right-hand side, we see that
  \begin{align} \label{Pf_ANG:122}
    \begin{aligned}
      \left\|\mathcal{M}_\varepsilon\Bigl[(U_1+1)\Bigl(\overline{\mathcal{M}_\varepsilon U_2}\Bigr)U_2\Bigr]\right\|_{L^2(\Gamma)} &\leq c\varepsilon^{-1/2}\left\|(U_1+1)\Bigl(\overline{\mathcal{M}_\varepsilon U_2}\Bigr)U_2\right\|_{L^2(\Omega_\varepsilon)} \\
      &\leq c\varepsilon^{-1/2}\|U_1+1\|_{L^6(\Omega_\varepsilon)}\left\|\Bigl(\overline{\mathcal{M}_\varepsilon U_2}\Bigr)U_2\right\|_{L^3(\Omega_\varepsilon)}
    \end{aligned}
  \end{align}
  by \eqref{E:ALp_Sur} and H\"{o}lder's inequality.
  Moreover, we see by \eqref{Pf_AN2:U16} and $|\Omega_\varepsilon|\leq c\varepsilon$ that
  \begin{align} \label{Pf_ANG:1pl}
    \|U_1+1\|_{L^6(\Omega_\varepsilon)} &\leq \|U_1\|_{L^6(\Omega_\varepsilon)}+|\Omega_\varepsilon|^{1/6} \leq c\Bigl(\varepsilon^{-1/3}\|u\|_{H^1(\Omega_\varepsilon)}+\varepsilon^{1/6}\Bigr).
  \end{align}
  Also, we use \eqref{E:Prod_L3} and then \eqref{E:ALp_Sur}, \eqref{E:AHk_Sur}, and $U_2=|u|+|\nabla u|$ in $\Omega_\varepsilon$ to get
  \begin{align*}
    \left\|\Bigl(\overline{\mathcal{M}_\varepsilon U_2}\Bigr)U_2\right\|_{L^3(\Omega_\varepsilon)} &\leq c\varepsilon^{-2/3}\|U_2\|_{L^2(\Omega_\varepsilon)}^{13/12}\|U_2\|_{H^1(\Omega_\varepsilon)}^{1/3}\|U_2\|_{H^2(\Omega_\varepsilon)}^{7/12} \\
    &\leq c\varepsilon^{-2/3}\|u\|_{H^1(\Omega_\varepsilon)}^{13/12}\|u\|_{H^2(\Omega_\varepsilon)}^{1/3}\|u\|_{H^3(\Omega_\varepsilon)}^{7/12}.
  \end{align*}
  Applying the above estimates to \eqref{Pf_ANG:122}, and recalling \eqref{Pf_ANG:L1P}, we find that
  \begin{align} \label{Pf_ANG:L1}
    \begin{aligned}
      \|L_1\|_{L^2(\Gamma)} &\leq c\varepsilon\left\|\mathcal{M}_\varepsilon\Bigl[(U_1+1)\Bigl(\overline{\mathcal{M}_\varepsilon U_2}\Bigr)U_2\Bigr]\right\|_{L^2(\Gamma)} \\
      &\leq c\varepsilon^{-1/2}\Bigl(\|u\|_{H^1(\Omega_\varepsilon)}+\varepsilon^{1/2}\Bigr)\|u\|_{H^1(\Omega_\varepsilon)}^{13/12}\|u\|_{H^2(\Omega_\varepsilon)}^{1/3}\|u\|_{H^3(\Omega_\varepsilon)}^{7/12}.
    \end{aligned}
  \end{align}
  Noting that $\varepsilon^{-1/2}=\varepsilon\cdot\varepsilon^{-3/2}$ and
  \begin{align*}
    \|u\|_{H^1(\Omega_\varepsilon)}^{13/12}\|u\|_{H^2(\Omega_\varepsilon)}^{1/3} \leq \|u\|_{H^1(\Omega_\varepsilon)}\|u\|_{H^2(\Omega_\varepsilon)}^{5/12}, \quad \|u\|_{H^2(\Omega_\varepsilon)} \leq \|u\|_{H^2(\Omega_\varepsilon)}^{5/12}\|u\|_{H^3(\Omega_\varepsilon)}^{7/12},
  \end{align*}
  we obtain \eqref{E:ANL_TGr} by \eqref{Pf_ANG:TGr}, \eqref{Pf_ANG:MV}, \eqref{Pf_ANG:L2}, and \eqref{Pf_ANG:L1}.
\end{proof}

\begin{remark} \label{R:ANL_TGr}
  For the estimate of $L_1$, we may also consider in \eqref{Pf_ANG:L1P} that
  \begin{align*}
    |L_1| \leq c\varepsilon\mathcal{M}_\varepsilon\Bigl[(U_1+1)\Bigl(\overline{\mathcal{M}_\varepsilon U_2}\Bigr)U_2\Bigr] = c\varepsilon\mathcal{M}_\varepsilon[(U_1+1)U_2]\mathcal{M}_\varepsilon U_2 \quad\text{on}\quad \Gamma.
  \end{align*}
  Then, by H\"{o}lder's inequality on $\Gamma$, \eqref{E:ALp_Sur}, and H\"{o}lder's inequality in $\Omega_\varepsilon$,
  \begin{align*}
    \|L_1\|_{L^2(\Gamma)} &\leq c\varepsilon\|\mathcal{M}_\varepsilon[(U_1+1)U_2]\|_{L^3(\Gamma)}\|\mathcal{M}_\varepsilon U_2\|_{L^6(\Gamma)} \\
    &\leq c\varepsilon^{1/2}\|U_1+1\|_{L^6(\Omega_\varepsilon)}\|U_2\|_{L^6(\Omega_\varepsilon)}^2.
  \end{align*}
  Moreover, we see by \eqref{E:Ipl_CTD} with $p=6$ and $U_2=|u|+|\nabla u|$ in $\Omega_\varepsilon$ that
  \begin{align*}
    \|U_2\|_{L^6(\Omega_\varepsilon)} \leq c\varepsilon^{-1/3}\|U_2\|_{L^2(\Omega_\varepsilon)}^{1/2}\|U_2\|_{H^2(\Omega_\varepsilon)}^{1/2} \leq c\varepsilon^{-1/3}\|u\|_{H^1(\Omega_\varepsilon)}^{1/2}\|u\|_{H^3(\Omega_\varepsilon)}^{1/2}.
  \end{align*}
  By these estimates and \eqref{Pf_ANG:1pl}, we find that
  \begin{align*}
    \|L_1\|_{L^2(\Gamma)} \leq c\varepsilon^{-1/2}\Bigl(\|u\|_{H^1(\Omega_\varepsilon)}+\varepsilon^{1/2}\Bigr)\|u\|_{H^1(\Omega_\varepsilon)}\|u\|_{H^3(\Omega_\varepsilon)}.
  \end{align*}
  This idea is rather simple and natural, but the resulting estimate becomes
  \begin{align} \label{E:ANL_GrW}
    \|\nabla_\Gamma\zeta_G^\varepsilon\|_{L^2(\Gamma)} \leq c\varepsilon\left\{\varepsilon^{-3/2}\Bigl(\|u\|_{H^1(\Omega_\varepsilon)}+\varepsilon^{1/2}\Bigr)\|u\|_{H^1(\Omega_\varepsilon)}\|u\|_{H^3(\Omega_\varepsilon)}+1\right\},
  \end{align}
  which is a slightly worse than \eqref{E:ANL_TGr}.
\end{remark}

\subsection{Weak time derivative of the average} \label{SS:Ave_WDt}
We have the following relation on the weak time derivative of $\mathcal{M}_\varepsilon u$.
Recall the function space $\mathcal{E}_T(S)$ given in \eqref{E:Def_ET}.

\begin{lemma} \label{L:Ave_WDt}
  Let $T>0$ and $u\in\mathcal{E}_T(\Omega_\varepsilon)$.
  Then, $\mathcal{M}_\varepsilon u\in\mathcal{E}_T(\Gamma)$ and
  \begin{align} \label{E:Ave_WDt}
    \int_0^T\langle\partial_tu,\bar{\eta}\rangle_{H^1(\Omega_\varepsilon)}\,dt = \varepsilon\int_0^T\langle\partial_t\mathcal{M}_\varepsilon u,g\eta\rangle_{H^1(\Gamma)}\,dt
  \end{align}
  for all $\eta\in L^2(0,T;H^1(\Gamma))$.
  Here, $\bar{\eta}(\cdot,t)$ is the constant extension of $\eta(\cdot,t)$ in the normal direction of $\Gamma$ for each $t\in(0,T)$.
\end{lemma}

\begin{proof}
  The proof is the same as that of \cite[Lemma 5.9]{Miu24_GL}, where $H^1$ is replaced by $H^1\cap L^4$.
  Here, we give the proof for the reader's convenience.

  First, $\mathcal{M}_\varepsilon u\in L^2(0,T;H^1(\Omega_\varepsilon))$ by \eqref{E:ALp_Sur} and \eqref{E:AHk_Sur}.

  Next, let $\eta\in C_c^1(0,T;H^1(\Gamma))$.
  Then, by \eqref{E:CEGr_Bd} and \eqref{E:Lp_CE}, we have
  \begin{align*}
    \bar{\eta} \in C_c^1(0,T;H^1(\Omega_\varepsilon)), \quad \|\bar{\eta}\|_{L^2(0,T;H^1(\Omega_\varepsilon))} \leq c\varepsilon^{1/2}\|\eta\|_{L^2(0,T;H^1(\Gamma))}.
  \end{align*}
  We apply Lemma \ref{L:ET_Equi}, (ii) to $u$ and use $\partial_t\bar{\eta}=\overline{\partial_t\eta}$ and \eqref{E:Ave_Pair} to get
  \begin{align} \label{Pf_ADt:IbP}
    \int_0^T\langle\partial_tu,\bar{\eta}\rangle_{H^1(\Omega_\varepsilon)}\,dt = -\int_0^T(u,\partial_t\bar{\eta})_{L^2(\Omega_\varepsilon)}\,dt = -\varepsilon\int_0^T(g\mathcal{M}_\varepsilon u,\partial_t\eta)_{L^2(\Gamma)}\,dt.
  \end{align}
  In \eqref{Pf_ADt:IbP}, we replace $\eta$ by $g^{-1}\eta$.
  Then, since $g$ is independent of time,
  \begin{align*}
    \int_0^T(\mathcal{M}_\varepsilon u,\partial_t\eta)_{L^2(\Gamma)}\,dt = -\varepsilon^{-1}\int_0^T\langle\partial_tu,\bar{g}^{-1}\bar{\eta}\rangle_{H^1(\Omega_\varepsilon)}\,dt.
  \end{align*}
  Hence, by H\"{o}lder's inequality, the regularity of $g$, and \eqref{E:G_Bdd}, we find that
  \begin{align*}
    \left|\int_0^T(\mathcal{M}_\varepsilon u,\partial_t\eta)_{L^2(\Gamma)}\,dt\right| &\leq c\varepsilon^{-1}\|\partial_tu\|_{L^2(0,T;[H^1(\Omega_\varepsilon)]')}\|\bar{\eta}\|_{L^2(0,T;H^1(\Omega_\varepsilon))} \\
    &\leq c\varepsilon^{-1/2}\|\partial_tu\|_{L^2(0,T;[H^1(\Omega_\varepsilon)]')}\|\eta\|_{L^2(0,T;H^1(\Gamma))}
  \end{align*}
  for all $\eta\in C_c^1(0,T;H^1(\Gamma))$.
  Thus, by Lemma \ref{L:ET_Equi}, we obtain
  \begin{align*}
    \partial_t\mathcal{M}_\varepsilon u\in L^2(0,T;[H^1(\Gamma)]'), \quad \mathcal{M}_\varepsilon u\in\mathcal{E}_T(\Gamma).
  \end{align*}
  When $\eta\in C_c^1(0,T;H^1(\Gamma))$, we have \eqref{E:Ave_WDt} by \eqref{Pf_ADt:IbP}, since $g$ is independent of time.
  By a density argument, \eqref{E:Ave_WDt} is also valid for $\eta\in L^2(0,T;H^1(\Gamma))$.
\end{proof}

\section{Thin-film limit problem} \label{S:TFL}
The goal of this section is to prove Theorems \ref{T:TFL_Weak}, \ref{T:DiE_Sur}, and \ref{T:DiE_CTD}.
Let $\mathcal{M}_\varepsilon$ be the weighted average operator given in Section \ref{S:Ave}.
We write $c$ for a general positive constant independent of $\varepsilon$, $t$, and $T$.
For function $\eta$ on $\Gamma$, let $\bar{\eta}$ be the constant extension of $\eta$ in the normal direction of $\Gamma$.
In what follows, we frequently use the inequalities
\begin{align*}
  1 \leq \varepsilon^{-\beta} \leq \varepsilon^{-\gamma}, \quad 1 \leq (1+t)^\beta \leq (1+t)^\gamma
\end{align*}
for $0<\varepsilon<1$, $t\geq0$, and $0\leq\beta\leq\gamma$ without mention.

\subsection{Estimates for a solution to the thin-domain problem} \label{SS:TFL_Est}
Let $u_0^\varepsilon\in H^1(\Omega_\varepsilon)$ and Assumption \ref{A:Poten} be satisfied.
Also, let $(u^\varepsilon,w^\varepsilon)$ be the unique weak solution to \eqref{E:CH_CTD}.
We give some estimates for $u^\varepsilon$ and $w^\varepsilon$ explicitly in terms of $\varepsilon$.
Recall that
\begin{align*}
  E_\varepsilon(u^\varepsilon) = \int_{\Omega_\varepsilon}\left(\frac{|\nabla u^\varepsilon|^2}{2}+F(u^\varepsilon)\right)\,dx
\end{align*}
is the Ginzburg--Landau free energy.

\begin{proposition} \label{P:TFL_Est}
  Let $c>0$, $\alpha\geq0$, and $\varepsilon_0\in(0,1)$ be constants.
  Also, let $\varepsilon\in(0,\varepsilon_0)$.
  Suppose that the initial data $u_0^\varepsilon$ satisfies
  \begin{align} \label{E:TFLE_u0}
    \|u_0^\varepsilon\|_{L^2(\Omega_\varepsilon)}^2+|E_\varepsilon(u_0^\varepsilon)| \leq c\varepsilon^{1-2\alpha}.
  \end{align}
  Then, for all $t\geq0$, we have
  \begin{align}
    &\|u^\varepsilon(t)\|_{H^1(\Omega_\varepsilon)}^2+\int_0^t\Bigl(\|u^\varepsilon(s)\|_{H^2(\Omega_\varepsilon)}^2+\|\nabla w^\varepsilon(s)\|_{L^2(\Omega_\varepsilon)}^2\Bigr)\,ds \leq c\varepsilon^{1-2\alpha}(1+t)^2, \label{E:TFLE_H1} \\
    &\int_0^t\Bigl(\|u^\varepsilon(s)\|_{H^3(\Omega_\varepsilon)}^2+\|w^\varepsilon(s)\|_{H^1(\Omega_\varepsilon)}^2\Bigr)\,ds \leq c\varepsilon^{1-6\alpha}(1+t)^6. \label{E:TFLE_H3}
  \end{align}
\end{proposition}

\begin{proof}
  The estimate \eqref{E:TFLE_H1} follows from \eqref{E:CHT_Gr2}--\eqref{E:CHT_H2E} and \eqref{E:TFLE_u0}.

  Let us show \eqref{E:TFLE_H3}.
  To this end, we first observe by \eqref{E:FL2_CTD} and \eqref{E:FH1_CTD} that
  \begin{align*}
    \|F'(u^\varepsilon)\|_{L^2(0,t;L^2(\Omega_\varepsilon))}^2 &\leq c\Bigl(\varepsilon^{-2}\|u^\varepsilon\|_{L^\infty(0,t;H^1(\Omega_\varepsilon))}^4\|u^\varepsilon\|_{L^2(0,t;H^1(\Omega_\varepsilon))}^2+t|\Omega_\varepsilon|\Bigr), \\
    \|\nabla F'(u^\varepsilon)\|_{L^2(0,t;L^2(\Omega_\varepsilon))}^2 &\leq c\Bigl(\varepsilon^{-2}\|u^\varepsilon\|_{L^\infty(0,t;H^1(\Omega_\varepsilon))}^4+1\Bigr)\|u^\varepsilon\|_{L^2(0,t;H^2(\Omega_\varepsilon))}^2.
  \end{align*}
  Hence, by \eqref{E:TFLE_H1} and $|\Omega_\varepsilon|\leq c\varepsilon$, we have
  \begin{align} \label{Pf_TEs:FH1}
    \int_0^t\|F'(u^\varepsilon(s))\|_{H^1(\Omega_\varepsilon)}^2\,ds \leq c\varepsilon^{1-6\alpha}(1+t)^6.
  \end{align}
  For a.a. $s>0$, since $u^\varepsilon(s)$ satisfies \eqref{E:CHT_Poi}, we can use \eqref{E:UER_CTD} with $k=1$ to get
  \begin{align*}
    \|u^\varepsilon(s)\|_{H^3(\Omega_\varepsilon)}^2 &\leq c\Bigl(\|\Delta u^\varepsilon(s)\|_{H^1(\Omega_\varepsilon)}^2+\|u^\varepsilon(s)\|_{L^2(\Omega_\varepsilon)}^2\Bigr) \\
    &\leq c\Bigl(\|\nabla\Delta u^\varepsilon(s)\|_{L^2(\Omega_\varepsilon)}^2+\|u^\varepsilon(s)\|_{H^2(\Omega_\varepsilon)}^2\Bigr) \\
    &\leq c\Bigl(\|\nabla w^\varepsilon(s)\|_{L^2(\Omega_\varepsilon)}^2+\|\nabla F'(u^\varepsilon(s))\|_{L^2(\Omega_\varepsilon)}^2+\|u^\varepsilon(s)\|_{H^2(\Omega_\varepsilon)}^2\Bigr).
  \end{align*}
  Also, since $w^\varepsilon(s)=-\Delta u^\varepsilon(s)+F'(u^\varepsilon(s))$ a.e. in $\Omega_\varepsilon$,
  \begin{align*}
    \|w^\varepsilon(s)\|_{H^1(\Omega_\varepsilon)}^2 &= \|w^\varepsilon(s)\|_{L^2(\Omega_\varepsilon)}^2+\|\nabla w^\varepsilon(s)\|_{L^2(\Omega_\varepsilon)}^2 \\
    &\leq c\Bigl(\|u^\varepsilon(s)\|_{H^2(\Omega_\varepsilon)}^2+\|F'(u^\varepsilon(s))\|_{L^2(\Omega_\varepsilon)}^2+\|\nabla w^\varepsilon(s)\|_{L^2(\Omega_\varepsilon)}^2\Bigr).
  \end{align*}
  By these inequalities, \eqref{E:TFLE_H1}, and \eqref{Pf_TEs:FH1}, we get \eqref{E:TFLE_H3}.
\end{proof}

\subsection{Weighed average of a solution to the thin-domain problem} \label{SS:TFL_Ave}
From now on, we suppose that the assumptions in Proposition \ref{P:TFL_Est} are satisfied.

\begin{proposition} \label{P:ACW_Reg}
  Let $v^\varepsilon:=\mathcal{M}_\varepsilon u^\varepsilon$ and $\mu^\varepsilon:=\mathcal{M}_\varepsilon w^\varepsilon$ on $\Gamma\times(0,\infty)$.
  Then,
  \begin{align*}
    v^\varepsilon\in \mathcal{E}_T(\Gamma)\cap L^2(0,T;H^3(\Gamma))\cap C([0,T];H^1(\Gamma)), \quad \mu^\varepsilon\in L^2(0,T;H^1(\Gamma)),
  \end{align*}
  and $v^\varepsilon(0)=v_0^\varepsilon:=\mathcal{M}_\varepsilon u_0^\varepsilon$ a.e. on $\Gamma$.
\end{proposition}

\begin{proof}
  The regularity follows from \eqref{E:CHT_WeRe}, Proposition \ref{P:CHT_Reg}, and Lemmas \ref{L:Ave_Hk} and \ref{L:Ave_WDt}.
  Also, since $u^\varepsilon(0)=u_0^\varepsilon$ a.e. in $\Omega_\varepsilon$, we take the weighted average to get $v^\varepsilon(0)=v_0^\varepsilon$ a.e. on $\Gamma$.
\end{proof}

Let us derive a weak form of $v^\varepsilon$ from the weak form \eqref{E:CHT_WF_u} of $u^\varepsilon$.

\begin{proposition} \label{P:ACW_WFv}
  For all $T>0$ and $\eta\in L^2(0,T;H^1(\Gamma))$, we have
  \begin{align} \label{E:ACW_WFv}
    \int_0^T\langle\partial_tv^\varepsilon,g\eta\rangle_{H^1(\Gamma)}\,dt+\int_0^T(g\nabla_\Gamma\mu^\varepsilon,\nabla_\Gamma\eta)_{L^2(\Gamma)}\,dt = R_\varepsilon(\eta;T).
  \end{align}
  Here, $R_\varepsilon(\eta;T)$ is a residual term which is linear in $\eta$ and satisfies
  \begin{align} \label{E:AWv_Res}
    |R_\varepsilon(\eta;T)| \leq c\varepsilon^{1-3\alpha}(1+T)^3\|\nabla_\Gamma\eta\|_{L^2(0,T;L^2(\Gamma))}.
  \end{align}
\end{proposition}

\begin{proof}
  We set $\varphi=\bar{\eta}$ in \eqref{E:CHT_WF_u}, divide both sides by $\varepsilon$, use \eqref{E:Ave_WDt}, and define
  \begin{align*}
    R_\varepsilon(\eta;T) := \int_0^T\Bigl\{-\varepsilon^{-1}(\nabla w^\varepsilon,\nabla\bar{\eta})_{L^2(\Omega_\varepsilon)}+(g\nabla_\Gamma\mu^\varepsilon,\nabla_\Gamma\eta)_{L^2(\Gamma)}\Bigr\}\,dt.
  \end{align*}
  Then, we get \eqref{E:ACW_WFv}.
  By the definition, $R_\varepsilon(\eta;T)$ is linear in $\eta$.
  Moreover,
  \begin{align*}
    |R_\varepsilon(\eta;T)| \leq c\varepsilon^{1/2}\|w^\varepsilon\|_{L^2(0,T;H^1(\Omega_\varepsilon))}\|\nabla_\Gamma\eta\|_{L^2(0,T;L^2(\Gamma))}
  \end{align*}
  by \eqref{E:ADF_Diff} and H\"{o}lder's inequality,
  Hence, \eqref{E:AWv_Res} follows from \eqref{E:TFLE_H3}.
\end{proof}

Also, we have a relation between $v^\varepsilon$ and $\mu^\varepsilon$ in a strong form.

\begin{proposition} \label{P:ACW_wSt}
  Let $A_gv^\varepsilon=g^{-1}\mathrm{div}_\Gamma(g\nabla_\Gamma v^\varepsilon)$ on $\Gamma\times(0,\infty)$.
  Then,
  \begin{align} \label{E:ACW_wSt}
    \mu^\varepsilon = -A_gv^\varepsilon+F'(v^\varepsilon)+\zeta^\varepsilon \quad\text{a.e. on}\quad \Gamma\times(0,\infty),
  \end{align}
  where $\zeta^\varepsilon$ is a residual term.
  Moreover, for all $T>0$, we have
  \begin{align} \label{E:AwS_Res}
    \begin{aligned}
      \|\zeta^\varepsilon\|_{L^2(0,T;L^2(\Gamma))} &\leq c\varepsilon^{1-3\alpha}(1+T)^3, \\
      \|\nabla_\Gamma\zeta^\varepsilon\|_{L^2(0,T;L^2(\Gamma))} &\leq c\varepsilon^{1-25\alpha/6}(1+T)^{25/6}.
    \end{aligned}
  \end{align}
\end{proposition}

\begin{proof}
  Let $\zeta^\varepsilon:=-\zeta_\Delta^\varepsilon+\zeta_{F'}^\varepsilon$ on $\Gamma\times(0,\infty)$, where
  \begin{align*}
    \zeta_\Delta^\varepsilon := \mathcal{M}_\varepsilon(\Delta u^\varepsilon)-A_gv^\varepsilon, \quad \zeta_{F'}^\varepsilon := \mathcal{M}_\varepsilon(F'(u^\varepsilon))-F'(v^\varepsilon).
  \end{align*}
  Then, we have \eqref{E:ACW_wSt} by taking the weighted average of
  \begin{align*}
    w^\varepsilon = -\Delta u^\varepsilon+F'(u^\varepsilon) \quad\text{a.e. in}\quad \Omega_\varepsilon\times(0,\infty).
  \end{align*}
  Let us show \eqref{E:AwS_Res}.
  We see by \eqref{E:ALa_Hk}, \eqref{E:TFLE_H1}, and \eqref{E:TFLE_H3} that
  \begin{align} \label{Pf_AwS:ZD}
    \begin{aligned}
      \|\zeta_\Delta^\varepsilon\|_{L^2(0,T;L^2(\Gamma))} &\leq c\varepsilon^{1/2}\|u^\varepsilon\|_{L^2(0,T;H^2(\Omega_\varepsilon))} \leq c\varepsilon^{1-\alpha}(1+T), \\
      \|\nabla_\Gamma\zeta_\Delta^\varepsilon\|_{L^2(0,T;L^2(\Gamma))} &\leq c\varepsilon^{1/2}\|u^\varepsilon\|_{L^2(0,T;H^3(\Omega_\varepsilon))} \leq c\varepsilon^{1-3\alpha}(1+T)^3.
    \end{aligned}
  \end{align}
  Next, since $G=F'$ satisfies \eqref{E:ANL_G1} and \eqref{E:ANL_G2} by Assumption \ref{A:Poten} and Remark \ref{R:Po_Grow}, we can use \eqref{E:ANL_L2} and \eqref{E:ANL_TGr}.
  By these estimates and H\"{o}lder's inequality, we get
  \begin{align*}
    \|\zeta_{F'}^\varepsilon\|_{L^2(0,T;L^2(\Gamma))} &\leq c\varepsilon\Bigl(\varepsilon^{-3/2}\|u^\varepsilon\|_{L^\infty(0,T;H^1(\Omega_\varepsilon))}^2\|u^\varepsilon\|_{L^2(0,T;H^2(\Omega_\varepsilon))}+T^{1/2}\Bigr), \\
    \|\nabla_\Gamma\zeta_{F'}^\varepsilon\|_{L^2(0,T;L^2(\Gamma))} &\leq c\varepsilon\{\varepsilon^{-3/2}\sigma_\varepsilon(u^\varepsilon;T)+T^{1/2}\}.
  \end{align*}
  Here, we set $\mathcal{X}_{\varepsilon,T}:=L^\infty(0,T;H^1(\Omega_\varepsilon))$, $\mathcal{Y}_{\varepsilon,T}^k:=L^2(0,T;H^k(\Omega_\varepsilon))$, and
  \begin{align*}
    \sigma_\varepsilon(u^\varepsilon;T) := \Bigl(\|u^\varepsilon\|_{\mathcal{X}_{\varepsilon,T}}+\varepsilon^{1/2}\Bigr)\|u^\varepsilon\|_{\mathcal{X}_{\varepsilon,T}}\|u^\varepsilon\|_{\mathcal{Y}_{\varepsilon,T}^2}^{5/12}\|u^\varepsilon\|_{\mathcal{Y}_{\varepsilon,T}^3}^{7/12}.
  \end{align*}
  We further apply \eqref{E:TFLE_H1} and \eqref{E:TFLE_H3} to the above inequality to get
  \begin{align} \label{Pf_AwS:ZF}
    \begin{aligned}
      \|\zeta_{F'}^\varepsilon\|_{L^2(0,T;L^2(\Gamma))} &\leq c\varepsilon^{1-3\alpha}(1+T)^3, \\
      \|\nabla_\Gamma\zeta_{F'}^\varepsilon\|_{L^2(0,T;L^2(\Gamma))} &\leq c\varepsilon^{1-25\alpha/6}(1+T)^{25/6}.
    \end{aligned}
  \end{align}
  Thus, the estimate \eqref{E:AwS_Res} follows from \eqref{Pf_AwS:ZD} and \eqref{Pf_AwS:ZF}.
\end{proof}

\begin{remark} \label{R:ACW_wSt}
  If we use \eqref{E:ANL_GrW} instead of \eqref{E:ANL_TGr} to estimate $\nabla_\Gamma\zeta_{F'}^\varepsilon$, then the second line of \eqref{Pf_AwS:ZF} is slightly changed, and the estimate for $\nabla_\Gamma\zeta^\varepsilon$ becomes
  \begin{align} \label{E:AwR_Wor}
    \|\nabla_\Gamma\zeta^\varepsilon\|_{L^2(0,T;L^2(\Gamma))} \leq c\varepsilon^{1-5\alpha}(1+T)^5.
  \end{align}
  This is worse than \eqref{E:AwS_Res} by $\varepsilon^{-5\alpha/6}$.
\end{remark}

For a later use, we rewrite \eqref{E:ACW_wSt} in a weak form.

\begin{proposition} \label{P:ACW_wWe}
  For all $T>0$ and $\eta\in L^2(0,T;H^1(\Gamma))$, we have
  \begin{align} \label{E:ACW_wWe}
    \int_0^T(g\mu^\varepsilon,\eta)_{L^2(\Gamma)}\,dt = \int_0^T(g\nabla_\Gamma v^\varepsilon,\nabla_\Gamma\eta)_{L^2(\Gamma)}\,dt+\int_0^T(g\{F'(v^\varepsilon)+\zeta^\varepsilon\},\eta)_{L^2(\Gamma)}\,dt.
  \end{align}
\end{proposition}

\begin{proof}
  We multiply \eqref{E:ACW_wSt} by $g\eta$, integrate over $\Gamma\times(0,T)$, and use \eqref{E:IbP_Ag} to get \eqref{E:ACW_wWe}.
\end{proof}

\subsection{Estimates for the averaged solution} \label{SS:TFL_AEn}
Let us estimate $v^\varepsilon$ and $\mu^\varepsilon$ in order to show that they are uniformly bounded in suitable function sapces.

In what follows, we write $c_t=c(1+t)^\gamma\geq1$ for $t\geq0$ with some general constants $c\geq1$ and $\gamma>0$ independent of $t$ in order to keep expressions simple.

First, we consider an estimate for the weighted Ginzburg--Landau free energy
\begin{align*}
  E_g(v^\varepsilon) := \int_\Gamma g\left(\frac{|\nabla_\Gamma v^\varepsilon|^2}{2}+F(v^\varepsilon)\right)\,d\mathcal{H}^2.
\end{align*}

\begin{proposition} \label{P:AEna_GL}
  Let $\tilde{\mu}^\varepsilon:=-A_gv^\varepsilon+F'(v^\varepsilon)$ on $\Gamma\times(0,\infty)$.
  Then, for all $t\geq0$,
  \begin{align} \label{E:AEna_GL}
    E_g(v^\varepsilon(t))+\frac{1}{2}\int_0^t\|\sqrt{g}\,\nabla_\Gamma\tilde{\mu}^\varepsilon(s)\|_{L^2(\Gamma)}^2\,ds \leq E_g(v_0^\varepsilon)+c_t\varepsilon^{2-25\alpha/3}.
  \end{align}
\end{proposition}

\begin{proof}
  As in the proof of Proposition \ref{P:CHT_EL2}, we sometimes omit the time variable.

  By Assumption \ref{A:Poten} and Proposition \ref{P:ACW_Reg}, we can use Lemma \ref{L:DtEn_gSu} to get
  \begin{align*}
    \frac{d}{dt}E_g(v^\varepsilon) = \langle\partial_tv^\varepsilon,g\{-A_gv^\varepsilon+F'(v^\varepsilon)\}\rangle_{H^1(\Gamma)} = \langle\partial_tv^\varepsilon,g\tilde{\mu}^\varepsilon\rangle_{H^1(\Gamma)} \quad\text{a.e. on}\quad (0,\infty).
  \end{align*}
  We integrate both sides over $(0,t)$ and use \eqref{E:ACW_WFv} with $\eta=\tilde{\mu}^\varepsilon$.
  Then,
  \begin{align*}
    E_g(v^\varepsilon(t))-E_g(v_0^\varepsilon) = -\int_0^t(g\nabla_\Gamma\mu^\varepsilon,\nabla_\Gamma\tilde{\mu}^\varepsilon)_{L^2(\Gamma)}\,ds+R_\varepsilon(\tilde{\mu}^\varepsilon;t).
  \end{align*}
  Moreover, since $\mu^\varepsilon=\tilde{\mu}^\varepsilon+\zeta^\varepsilon$ a.e. on $\Gamma\times(0,\infty)$ by \eqref{E:ACW_wSt},
  \begin{multline} \label{Pf_AGL:Ine}
    E_g(v^\varepsilon(t))+\int_0^t\|\sqrt{g}\,\nabla_\Gamma\tilde{\mu}^\varepsilon\|_{L^2(\Gamma)}^2\,ds \\
    = E_g(v_0^\varepsilon)+\int_0^t(g\nabla_\Gamma\zeta^\varepsilon,\nabla_\Gamma\tilde{\mu}^\varepsilon)_{L^2(\Gamma)}\,ds+R(\tilde{\mu}^\varepsilon;t).
  \end{multline}
  We further observe by H\"{o}lder's inequality, \eqref{E:G_Bdd}, and \eqref{E:AwS_Res} that
  \begin{align*}
    \left|\int_0^t(g\nabla_\Gamma\zeta^\varepsilon,\nabla_\Gamma\tilde{\mu}^\varepsilon)_{L^2(\Gamma)}\,ds\right| &\leq c\|\sqrt{g}\,\nabla_\Gamma\zeta^\varepsilon\|_{L^2(0,t;L^2(\Gamma))}\|\sqrt{g}\,\nabla_\Gamma\tilde{\mu}^\varepsilon\|_{L^2(0,t:L^2(\Gamma))} \\
    &\leq c_t\varepsilon^{1-25\alpha/6}\|\sqrt{g}\,\nabla_\Gamma\tilde{\mu}^\varepsilon\|_{L^2(0,t;L^2(\Gamma))}.
  \end{align*}
  Also, we use \eqref{E:AWv_Res} and then \eqref{E:G_Bdd} (and $3\alpha\leq 25\alpha/6$) to get
  \begin{align*}
    |R_\varepsilon(\tilde{\mu}^\varepsilon;t)| \leq c_t\varepsilon^{1-3\alpha}\|\nabla_\Gamma\tilde{\mu}^\varepsilon\|_{L^2(0,t;L^2(\Gamma))} \leq c_t\varepsilon^{1-25\alpha/6}\|\sqrt{g}\,\nabla_\Gamma\tilde{\mu}^\varepsilon\|_{L^2(0,t;L^2(\Gamma))}.
  \end{align*}
  We apply the above estimates and \eqref{E:G_Bdd} to \eqref{Pf_AGL:Ine}.
  Then, we find that
  \begin{align*}
    E_g(v^\varepsilon(t))+\int_0^t\|\sqrt{g}\,\nabla_\Gamma\tilde{\mu}^\varepsilon\|_{L^2(\Gamma)}^2\,ds &\leq E_g(v_0^\varepsilon)+c_t\varepsilon^{1-25\alpha/6}\|\sqrt{g}\,\nabla_\Gamma\tilde{\mu}^\varepsilon\|_{L^2(0,t;L^2(\Gamma))} \\
    &\leq E_g(v_0^\varepsilon)+c_t\varepsilon^{2-25\alpha/3}+\frac{1}{2}\int_0^t\|\sqrt{g}\,\nabla_\Gamma\tilde{\mu}^\varepsilon\|_{L^2(\Gamma)}^2\,ds
  \end{align*}
  by Young's inequality.
  Therefore, \eqref{E:AEna_GL} follows.
\end{proof}

Using Proposition \ref{P:AEna_GL}, we can derive several estimates for $v^\varepsilon$ and $\mu^\varepsilon$.

\begin{proposition} \label{P:AEna_Gr2}
  For all $t\geq0$, we have
  \begin{align} \label{E:AEna_Gr2}
    \|\nabla_\Gamma v^\varepsilon(t)\|_{L^2(\Gamma)}^2+\int_0^t\|\nabla_\Gamma\mu^\varepsilon(s)\|_{L^2(\Gamma)}^2\,ds \leq c_t(|E_g(v_0^\varepsilon)|+\varepsilon^{2-25\alpha/3}+1).
  \end{align}
\end{proposition}

\begin{proof}
  Since $g$ is positive on $\Gamma$, for all $t\geq0$ we see by \eqref{E:Poten} that
  \begin{align*}
    E_g(v^\varepsilon(t))+C_0\int_\Gamma g\,d\mathcal{H}^2 &= \int_\Gamma g\left(\frac{|\nabla_\Gamma v^\varepsilon(t)|^2}{2}+\{F(v^\varepsilon(t))+C_0\}\right)\,d\mathcal{H}^2 \\
    &\geq \frac{1}{2}\|\sqrt{g}\,\nabla_\Gamma v^\varepsilon(t)\|_{L^2(\Gamma)}^2.
  \end{align*}
  Hence, it follows from \eqref{E:AEna_GL} that
  \begin{align*}
    \frac{1}{2}\|\sqrt{g}\,\nabla_\Gamma v^\varepsilon(t)\|_{L^2(\Gamma)}^2+\frac{1}{2}\int_0^t\|\sqrt{g}\,\nabla_\Gamma\tilde{\mu}^\varepsilon\|_{L^2(\Gamma)}^2\,ds \leq E_g(v_0^\varepsilon)+c_t\varepsilon^{2-25\alpha/3}+C_0\int_\Gamma g\,d\mathcal{H}^2.
  \end{align*}
  Applying \eqref{E:G_Bdd} and $E_g(v_0^\varepsilon)\leq|E_g(v_0^\varepsilon)|$, we further get
  \begin{align*}
    \|\nabla_\Gamma v^\varepsilon(t)\|_{L^2(\Gamma)}^2+\int_0^t\|\nabla_\Gamma\tilde{\mu}^\varepsilon\|_{L^2(\Gamma)}^2\,ds \leq c_t(|E_g(v_0^\varepsilon)|+\varepsilon^{2-25\alpha/3}+1).
  \end{align*}
  Also, since $\mu^\varepsilon=\tilde{\mu}^\varepsilon+\zeta^\varepsilon$ a.e. on $\Gamma\times(0,\infty)$ by \eqref{E:ACW_wSt}, we have
  \begin{align*}
    \int_0^t\|\nabla_\Gamma\mu^\varepsilon\|_{L^2(\Gamma)}^2\,ds &\leq c\left(\int_0^t\|\nabla_\Gamma\tilde{\mu}^\varepsilon\|_{L^2(\Gamma)}^2\,ds+\int_0^t\|\nabla_\Gamma\zeta^\varepsilon\|_{L^2(\Gamma)}^2\,ds\right) \\
    &\leq c\left(\int_0^t\|\nabla_\Gamma\tilde{\mu}^\varepsilon\|_{L^2(\Gamma)}^2\,ds+c_t\varepsilon^{2-25\alpha/3}\right)
  \end{align*}
  by \eqref{E:AwS_Res}.
  By the above two estimates, we obtain \eqref{E:AEna_Gr2}.
\end{proof}

\begin{proposition} \label{P:AEna_L2}
  We define
  \begin{align*}
    \mathcal{I}_{\varepsilon,\alpha} := \|v_0^\varepsilon\|_{L^2(\Gamma)}^2+|E_g(v_0^\varepsilon)|+\varepsilon^{2-25\alpha/3}+1.
  \end{align*}
  Then, for all $t\geq0$, we have
  \begin{align}
    &\|v^\varepsilon(t)\|_{L^2(\Gamma)}^2 \leq c_t\mathcal{I}_{\varepsilon,\alpha}, \label{E:AEna_v2} \\
    &\int_0^t\Bigl(\|F'(v^\varepsilon(s))\|_{L^2(\Gamma)}^2+\|\mu^\varepsilon(s)\|_{L^2(\Gamma)}^2\Bigr)\,ds \leq c_t\mathcal{I}_{\varepsilon,\alpha}^3. \label{E:AEna_mu2}
  \end{align}
\end{proposition}

\begin{proof}
  Let $\eta=v^\varepsilon$ in \eqref{E:ACW_WFv} with $T$ replaced by $t$.
  Then, by \eqref{E:DtL2_gSu}, we have
  \begin{align} \label{Pf_AE2:Equ}
    \frac{1}{2}\|\sqrt{g}\,v^\varepsilon(t)\|_{L^2(\Gamma)}^2-\frac{1}{2}\|\sqrt{g}\,v_0^\varepsilon\|_{L^2(\Gamma)}^2+\int_0^t(g\nabla\mu^\varepsilon,\nabla_\Gamma v^\varepsilon)_{L^2(\Gamma)}\,ds = R_\varepsilon(v^\varepsilon;t).
  \end{align}
  Moreover, we see by \eqref{E:G_Bdd}, H\"{o}lder's inequality, and \eqref{E:AEna_Gr2} that
  \begin{align} \label{Pf_AE2:gmv}
    \begin{aligned}
      \left|\int_0^t(g\nabla\mu^\varepsilon,\nabla_\Gamma v^\varepsilon)_{L^2(\Gamma)}\,ds\right| &\leq c\|\nabla_\Gamma\mu^\varepsilon\|_{L^2(0,t;L^2(\Gamma))}\|\nabla_\Gamma v^\varepsilon\|_{L^2(0,t;L^2(\Gamma))} \\
      &\leq c_t(|E_g(v_0^\varepsilon)|+\varepsilon^{2-25\alpha/3}+1).
    \end{aligned}
  \end{align}
  Also, by \eqref{E:AWv_Res}, \eqref{E:AEna_Gr2}, and $\varepsilon^{1-3\alpha}\leq\varepsilon^{1-25\alpha/6}\leq(|E_g(v_0^\varepsilon)|+\varepsilon^{2-25\alpha/3}+1)^{1/2}$,
  \begin{align*}
    |R_\varepsilon(v^\varepsilon;t)| \leq c_t\varepsilon^{1-3\alpha}\|\nabla_\Gamma v^\varepsilon\|_{L^2(0,t;L^2(\Gamma))} \leq c_t(|E_g(v_0^\varepsilon)|+\varepsilon^{2-25\alpha/3}+1).
  \end{align*}
  Applying the above estimates and \eqref{E:G_Bdd} to \eqref{Pf_AE2:Equ}, we obtain \eqref{E:AEna_v2}.

  Let us show \eqref{E:AEna_mu2}.
  By \eqref{E:SoSu_H1}, \eqref{E:AEna_Gr2}, and \eqref{E:AEna_v2}, we have
  \begin{align*}
    \|[(v^\varepsilon)^3](t)\|_{L^2(\Gamma)}^2 = \|v^\varepsilon(t)\|_{L^6(\Gamma)}^6 \leq c\|v^\varepsilon(t)\|_{H^1(\Gamma)}^6 \leq c_t\mathcal{I}_{\varepsilon,\alpha}^3
  \end{align*}
  for all $t\geq0$.
  Since $|F'(z)|\leq c(|z|^3+1)$ by Remark \ref{R:Po_Grow}, it follows that
  \begin{align} \label{Pf_AE2:F}
    \int_0^t\|F'(v^\varepsilon)\|_{L^2(\Gamma)}^2\,ds \leq c\int_0^t\Bigl(\|(v^\varepsilon)^3\|_{L^2(\Gamma)}^2+|\Gamma|\Bigr)\,ds \leq c_t\mathcal{I}_{\varepsilon,\alpha}^3.
  \end{align}
  Here, $|\Gamma|$ is the area of $\Gamma$ and we also used $\mathcal{I}_{\varepsilon,\alpha}\geq1$ in the last inequality.
  To estimate $\mu^\varepsilon$, let $\eta=\mu^\varepsilon$ in \eqref{E:ACW_wWe} with $T$ replaced by $t$.
  Then,
  \begin{align} \label{Pf_AE2:mu}
    \int_0^t\|\sqrt{g}\,\mu^\varepsilon\|_{L^2(\Gamma)}^2\,ds = \int_0^t(g\nabla_\Gamma v^\varepsilon,\nabla_\Gamma\mu^\varepsilon)_{L^2(\Gamma)}\,ds+\int_0^t(g\{F'(v^\varepsilon)+\zeta^\varepsilon\},\mu^\varepsilon)_{L^2(\Gamma)}\,ds.
  \end{align}
  Moreover, it follows from H\"{o}lder's and Young's inequalities and \eqref{E:G_Bdd} that
  \begin{multline*}
    \left|\int_0^t(g\{F'(v^\varepsilon)+\zeta^\varepsilon\},\mu^\varepsilon)_{L^2(\Gamma)}\,ds\right| \\
    \leq \frac{1}{2}\int_0^t\|\sqrt{g}\,\mu^\varepsilon\|_{L^2(\Gamma)}^2\,ds+c\int_0^t\Bigl(\|F'(v^\varepsilon)\|_{L^2(\Gamma)}^2+\|\zeta^\varepsilon\|_{L^2(\Gamma)}^2\Bigr)\,ds.
  \end{multline*}
  We apply this inequality to \eqref{Pf_AE2:mu}, and use \eqref{E:AwS_Res}, \eqref{Pf_AE2:gmv}, and \eqref{Pf_AE2:F}.
  Then, since
  \begin{align} \label{Pf_AE2:Iea}
    \varepsilon^{2-6\alpha} \leq |E_g(v_0^\varepsilon)|+\varepsilon^{2-25\alpha/3}+1 \leq \mathcal{I}_{\varepsilon,\alpha} \leq \mathcal{I}_{\varepsilon,\alpha}^3,
  \end{align}
  where the last inequality is due to $\mathcal{I}_{\varepsilon,\alpha}\geq1$, we find that
  \begin{align*}
    \frac{1}{2}\int_0^t\|\sqrt{g}\,\mu^\varepsilon\|_{L^2(\Gamma)}^2\,ds \leq c_t\mathcal{I}_{\varepsilon,\alpha}^3.
  \end{align*}
  By this inequality, \eqref{E:G_Bdd}, and \eqref{Pf_AE2:F}, we obtain \eqref{E:AEna_mu2}.
\end{proof}

\begin{proposition} \label{P:AEna_Dt}
  For all $T\geq0$, we have
  \begin{align} \label{E:AEna_Dt}
    \|\partial_tv^\varepsilon\|_{L^2(0,T;[H^1(\Gamma)]')} \leq c_T\mathcal{I}_{\varepsilon,\alpha}^{1/2}.
  \end{align}
\end{proposition}

\begin{proof}
  Let $\eta\in L^2(0,T;H^1(\Gamma))$.
  Since \eqref{E:ACW_WFv} holds, we have
  \begin{align*}
    \left|\int_0^T\langle\partial_tv^\varepsilon,g\eta\rangle_{H^1(\Gamma)}\,dt\right| &\leq \left|\int_0^T(g\nabla_\Gamma\mu^\varepsilon,\nabla_\Gamma\eta)_{L^2(\Gamma)}\,dt\right|+|R_\varepsilon(\eta;T)| \\
    &\leq c_T\mathcal{I}_{\varepsilon,\alpha}^{1/2}\|\nabla_\Gamma\eta\|_{L^2(0,T;L^2(\Gamma))}
  \end{align*}
  by \eqref{E:G_Bdd}, H\"{o}lder's inequality, \eqref{E:AWv_Res}, \eqref{E:AEna_Gr2}, and \eqref{Pf_AE2:Iea}.
  Hence, replacing $\eta$ by $g^{-1}\eta$, and using the regularity of $g$ and \eqref{E:G_Bdd}, we find that
  \begin{align*}
    \left|\int_0^T\langle\partial_tv^\varepsilon,\eta\rangle_{H^1(\Gamma)}\,dt\right| \leq c_T\mathcal{I}_{\varepsilon,\alpha}^{1/2}\|\eta\|_{L^2(0,T;H^1(\Gamma))}
  \end{align*}
  for all $\eta\in L^2(0,T;H^1(\Gamma))$, which yields \eqref{E:AEna_Dt}.
\end{proof}

\subsection{Weak solution to the limit problem} \label{SS:TFL_Lim}
Now, we are almost ready to prove Theorem \ref{T:TFL_Weak}, but we need to define a weak solution to the limit problem \eqref{E:CH_Lim}.

Let $T>0$ and $\eta$ be a test function on $\Gamma\times(0,T)$.
We multiply \eqref{E:CH_Lim} by $g\eta$, integrate over $\Gamma\times(0,T)$, and carry out integration by parts \eqref{E:IbP_Ag}.
Then, we get the following weak formulation of \eqref{E:CH_Lim}, as we expect from \eqref{E:ACW_WFv} and \eqref{E:ACW_wWe}.

\begin{definition} \label{D:CLi_WeSo}
  For a given $v_0\in H^1(\Gamma)$ and $T>0$, we say that a pair $(v,\mu)$ is a weak solution to \eqref{E:CH_Lim} on $[0,T)$ if it satisfies
  \begin{align*}
    v\in\mathcal{E}_T(\Gamma)\cap L^\infty(0,T;H^1(\Gamma)), \quad \mu\in L^2(0,T;H^1(\Gamma)),
  \end{align*}
  and for all $\eta\in L^2(0,T;H^1(\Gamma))$,
  \begin{align}
    &\int_0^T\langle\partial_tv,g\eta\rangle_{H^1(\Gamma)}\,dt+\int_0^T(g\nabla_\Gamma\mu,\nabla_\Gamma\eta)_{L^2(\Gamma)}\,dt = 0, \label{E:CLi_Wv} \\
    &\int_0^T(g\mu,\eta)_{L^2(\Gamma)}\,dt = \int_0^T(g\nabla_\Gamma v,\nabla_\Gamma\eta)_{L^2(\Gamma)}\,dt+\int_0^T(gF'(v),\eta)_{L^2(\Gamma)}\,dt, \label{E:CLi_Wmu}
  \end{align}
  and $v(0)=v_0$ a.e. on $\Gamma$.
\end{definition}

\begin{definition} \label{D:CLi_GlW}
  For a given $v_0\in H^1(\Gamma)$, we say that a pair $(v,\mu)$ is a global weak solution to \eqref{E:CH_Lim} if it is a weak solution to \eqref{E:CH_Lim} on $[0,T)$ for all $T>0$.
\end{definition}

As in the thin-domain problem \eqref{E:CH_CTD}, we can show the existence of a weak solution to \eqref{E:CH_Lim} by the Galerkin method, but it can be also obtained from Theorem \ref{T:TFL_Weak} (see Corollary \ref{C:CLi_Ex} below).
Let us prove the uniqueness of a weak solution to \eqref{E:CH_Lim}, which is also used in the proof of Theorem \ref{T:TFL_Weak}.

\begin{proposition} \label{P:CLi_Uni}
  Suppose that Assumption \ref{A:Poten} is satisfied.
  Then, for each $T>0$ and $v_0\in H^1(\Gamma)$, there exists at most one weak solution to \eqref{E:CH_Lim} on $[0,T)$.
\end{proposition}

\begin{proof}
  Let $(v_1,\mu_1)$ and $(v_2,\mu_2)$ be weak solutions to \eqref{E:CH_Lim} on $(0,T)$ with same initial data $v_0$.
  We define $\chi_v:=v_1-v_2$ and $\chi_\mu:=\mu_1-\mu_2$ on $\Gamma\times(0,T)$.
  Then,
  \begin{align*}
    \chi_v\in\mathcal{E}_T(\Gamma)\cap L^\infty(0,T;H^1(\Gamma)), \quad \chi_\mu\in L^2(0,T;H^1(\Gamma)),
  \end{align*}
  and $\chi_v(0)=0$ a.e. on $\Gamma$, and for all $\eta\in L^2(0,T;H^1(\Gamma))$,
  \begin{align} \label{Pf_CLU:v}
    \int_0^T\langle\partial_t\chi_v,g\eta\rangle_{H^1(\Gamma)}\,dt+\int_0^T(g\nabla_\Gamma\chi_\mu,\nabla_\Gamma\eta)_{L^2(\Gamma)}\,dt = 0
  \end{align}
  and, with $N(v_1,v_2):=F'(v_1)-F'(v_2)$,
  \begin{align} \label{Pf_CLU:mu}
    \int_0^T(g\chi_\mu,\eta)_{L^2(\Gamma)}\,dt = \int_0^T(g\nabla_\Gamma\chi_v,\nabla_\Gamma\eta)_{L^2(\Gamma)}\,dt+\int_0^T(gN(v_1,v_2),\eta)_{L^2(\Gamma)}\,dt.
  \end{align}
  Let $L_g$ be the operator given in Lemma \ref{L:InA_Sur}.
  For each $t\in[0,T]$, we have
  \begin{align} \label{Pf_CLU:TA}
    \int_\Gamma g\chi_v(t)\,d\mathcal{H}^2 = \int_\Gamma g\chi_v(t)\,d\mathcal{H}^2-\int_\Gamma g\chi_v(0)\,d\mathcal{H}^2 = \int_0^t\langle\partial_t\chi_v,g\rangle_{H^1(\Gamma)}\,ds = 0
  \end{align}
  by $\chi_v(0)=0$ a.e. on $\Gamma$, \eqref{E:DtL2_gSu}, and \eqref{Pf_CLU:v} with $\eta\equiv1$ and $T$ replaced by $t$.
  Hence, we can apply $L_g$ to $\chi_v(t)$, and we use \eqref{E:DtIA_Sur} and then \eqref{Pf_CLU:v} with $\eta=L_g\chi_v$ to get
  \begin{align*}
    \frac{1}{2}\|\sqrt{g}\,\nabla_\Gamma L_g\chi_v(t)\|_{L^2(\Gamma)}^2-\frac{1}{2}\|\sqrt{g}\,\nabla_\Gamma L_g\chi_v(0)\|_{L^2(\Gamma)}^2 &= \int_0^t\langle\partial_t\chi_v,gL_g\chi_v\rangle_{H^1(\Gamma)}\,ds \\
    &= -\int_0^t(g\nabla_\Gamma\chi_\mu,\nabla_\Gamma L_g\chi_v)_{L^2(\Gamma)}\,ds.
  \end{align*}
  Using \eqref{E:InA_Sur} with $f=\chi_v$ and $\eta=\chi_\mu$, we further get
  \begin{align} \label{Pf_CLU:Lg}
    \frac{1}{2}\|\sqrt{g}\,\nabla_\Gamma L_g\chi_v(t)\|_{L^2(\Gamma)}^2-\frac{1}{2}\|\sqrt{g}\,\nabla_\Gamma L_g\chi_v(0)\|_{L^2(\Gamma)}^2 = -\int_0^t(g\chi_\mu,\chi_v)_{L^2(\Gamma)}\,ds.
  \end{align}
  To estimate the right-hand side, let $\eta=\chi_v$ in \eqref{Pf_CLU:mu}.
  Then, since
  \begin{align*}
    N(v_1,v_2) = F'(v_1)-F'(v_2) = F''(\lambda v_1+(1-\lambda)v_2)(v_1-v_2)
  \end{align*}
  a.e. on $\Gamma\times(0,T)$ with some constant $\lambda\in(0,1)$ by the mean value theorem,
  \begin{align*}
    \int_0^t(gN(v_1,v_2),\chi_v)_{L^2(\Gamma)}\,ds &= \int_0^t\left(\int_\Gamma gF''(\lambda v_1+(1-\lambda)v_2)|\chi_v|^2\,d\mathcal{H}^2\right)\,ds \\
    &\geq -C_2\int_0^t\left(\int_\Gamma g|\chi_v|^2\,d\mathcal{H}^2\right)\,ds
  \end{align*}
  by \eqref{E:Poten} (note that $g$ is positive on $\Gamma$).
  Hence, by \eqref{Pf_CLU:mu} with $\eta=\chi_v$,
  \begin{align*}
    \int_0^t(g\chi_\mu,\chi_v)_{L^2(\Gamma)}\,ds \geq \int_0^t\|\sqrt{g}\,\nabla_\Gamma\chi_v\|_{L^2(\Gamma)}^2\,ds-C_2\int_0^t\|\sqrt{g}\,\chi_v\|_{L^2(\Gamma)}^2\,ds.
  \end{align*}
  Applying this inequality to \eqref{Pf_CLU:Lg}, we find that
  \begin{multline} \label{Pf_CLU:Ine}
    \frac{1}{2}\|\sqrt{g}\,\nabla_\Gamma L_g\chi_v(t)\|_{L^2(\Gamma)}^2+\int_0^t\|\sqrt{g}\,\nabla_\Gamma\chi_v\|_{L^2(\Gamma)}^2\,ds \\
    \leq \frac{1}{2}\|\sqrt{g}\,\nabla_\Gamma L_g\chi_v(0)\|_{L^2(\Gamma)}^2+C_2\int_0^t\|\sqrt{g}\,\chi_v\|_{L^2(\Gamma)}^2\,ds.
  \end{multline}
  Setting $f=\eta=\chi_v$ in \eqref{E:InA_Sur} and using H\"{o}lder's inequality, we further get
  \begin{align*}
    \|\sqrt{g}\,\chi_v\|_{L^2(\Gamma)}^2 = (g\nabla_\Gamma L_g\chi_v,\nabla_\Gamma\chi_v)_{L^2(\Gamma)} \leq \|\sqrt{g}\,\nabla_\Gamma L_g\chi_v\|_{L^2(\Gamma)}\|\sqrt{g}\,\nabla_\Gamma\chi_v\|_{L^2(\Gamma)}
  \end{align*}
  and thus it follows from Young's inequality that
  \begin{align} \label{Pf_CLU:L2}
    C_2\int_0^t\|\sqrt{g}\,\chi_v\|_{L^2(\Gamma)}^2\,ds \leq \frac{1}{2}\int_0^t\|\sqrt{g}\,\nabla_\Gamma\chi_v\|_{L^2(\Gamma)}^2\,ds+c\int_0^t\|\sqrt{g}\,\nabla_\Gamma L_g\chi_v\|_{L^2(\Gamma)}^2\,ds.
  \end{align}
  We apply this inequality and $\chi_v(0)=0$ a.e. on $\Gamma$ to \eqref{Pf_CLU:Ine} to get
  \begin{align*}
    \frac{1}{2}\|\sqrt{g}\,\nabla_\Gamma L_g\chi_v(t)\|_{L^2(\Gamma)}^2+\frac{1}{2}\int_0^t\|\sqrt{g}\,\nabla_\Gamma\chi_v\|_{L^2(\Gamma)}^2\,ds \leq c\int_0^t\|\sqrt{g}\,\nabla_\Gamma L_g\chi_v\|_{L^2(\Gamma)}^2\,ds
  \end{align*}
  for all $t\in[0,T]$.
  Thus, Gronwall's inequality and \eqref{E:G_Bdd} imply that
  \begin{align*}
    \int_0^T\|\sqrt{g}\,\nabla_\Gamma\chi_v\|_{L^2(\Gamma)}^2\,dt = 0, \quad \nabla_\Gamma\chi_v = 0 \quad\text{a.e. on}\quad \Gamma\times(0,T),
  \end{align*}
  and we get $\chi_v=0$, i.e. $v_1=v_2$ a.e. on $\Gamma\times(0,T)$ by \eqref{E:Poi_Sur}, since \eqref{Pf_CLU:TA} holds.
  Also,
  \begin{align*}
    \int_0^T\|\sqrt{g}\,\chi_\mu\|_{L^2(\Gamma)}\,dt = 0, \quad \chi_\mu = 0 \quad\text{a.e. on}\quad \Gamma\times(0,T)
  \end{align*}
  by \eqref{Pf_CLU:mu} with $\eta=\chi_\mu$ and by \eqref{E:G_Bdd}.
  Hence, $\mu_1=\mu_2$ a.e. on $\Gamma\times(0,T)$.
\end{proof}

\subsection{Weak convergence and characterization of the limit} \label{SS:TFL_WC}
Let us give the proof of Theorem \ref{T:TFL_Weak}.
Recall that we write $c_t=c(1+t)^\gamma$ for $t\geq0$ with some general constants $c\geq1$ and $\gamma>0$ independent of $t$.

\begin{proof}[Proof of Theorem \ref{T:TFL_Weak}]
  Suppose that the assumptions of Theorem \ref{T:TFL_Weak} are satisfied.
  Let
  \begin{align*}
    v^\varepsilon = \mathcal{M}_\varepsilon u^\varepsilon, \quad \mu^\varepsilon = \mathcal{M}_\varepsilon w^\varepsilon \quad\text{on}\quad \Gamma\times(0,\infty), \quad v_0^\varepsilon = \mathcal{M}_\varepsilon u_0^\varepsilon \quad\text{on}\quad \Gamma
  \end{align*}
  for the unique global weak solution $(u^\varepsilon,w^\varepsilon)$ to \eqref{E:CH_CTD} with initial data $u_0^\varepsilon\in H^1(\Omega_\varepsilon)$.
  Since the assumptions of Proposition \ref{P:TFL_Est} are satisfied by the condition (a) of Theorem \ref{T:TFL_Weak}, we can use the results in the previous subsections.
  Moreover,
  \begin{align} \label{Pf_TW:Iea}
    \mathcal{I}_{\varepsilon,\alpha} = \|v_0^\varepsilon\|_{L^2(\Gamma)}^2+|E_g(v_0^\varepsilon)|+\varepsilon^{2-25\alpha/3}+1 \leq c
  \end{align}
  by the condition (a), $\alpha\leq6/25$, and $0<\varepsilon<1$.
  Hence, for each $T>0$,
  \begin{itemize}
    \item $\{v^\varepsilon\}_\varepsilon$ is bounded in $\mathcal{E}_T(\Gamma)\cap L^\infty(0,T;H^1(\Gamma))$,
    \item $\{\mu^\varepsilon\}_\varepsilon$ is bounded in $L^2(0,T;H^1(\Gamma))$
  \end{itemize}
  by Propositions \ref{P:AEna_Gr2}--\ref{P:AEna_Dt} (recall that $\mathcal{E}_T(\Gamma)$ is defined by \eqref{E:Def_ET}).
  Therefore, there exist a sequence $\{\varepsilon_k\}_k$ of positive numbers and functions
  \begin{align*}
    v_T \in \mathcal{E}_T(\Gamma)\cap L^\infty(0,T;H^1(\Gamma)), \quad \mu_T\in L^2(0,T;H^1(\Gamma))
  \end{align*}
  such that $\varepsilon_k\to0$ as $k\to\infty$ and
  \begin{align} \label{Pf_TW:WC}
    \begin{alignedat}{3}
      \lim_{k\to\infty}v^{\varepsilon_k} &= v_T &\quad &\text{weakly-$\ast$ in} &\quad &L^\infty(0,T;H^1(\Gamma)), \\
      \lim_{k\to\infty}\partial_tv^{\varepsilon_k} &= \partial_tv_T &\quad &\text{weakly in} &\quad &L^2(0,T;[H^1(\Gamma)]'), \\
      \lim_{k\to\infty}\mu^{\varepsilon_k} &= \mu_T &\quad &\text{weakly in} &\quad &L^2(0,T;H^1(\Gamma)).
    \end{alignedat}
  \end{align}
  Due to the Aubin--Lions lemma (see \cite[Theorem II.5.16]{BoyFab13}), we may also have
  \begin{align*}
    \lim_{k\to\infty}v^{\varepsilon_k} = v_T \quad\text{strongly in}\quad L^2(0,T;L^2(\Gamma))
  \end{align*}
  and thus $v^{\varepsilon_k}\to v_T$ a.e. on $\Gamma\times(0,T)$ by taking a subsequence.
  Hence,
  \begin{align*}
    \lim_{k\to\infty}F'(v^{\varepsilon_k}) = F'(v_T) \quad\text{a.e. on}\quad \Gamma\times(0,T)
  \end{align*}
  by $F\in C^3(\mathbb{R})$.
  Moreover, it follows from \eqref{E:AEna_mu2} and \eqref{Pf_TW:Iea} that
  \begin{align*}
    \|F'(v^{\varepsilon_k})\|_{L^2(0,T;L^2(\Gamma))} \leq c_T \quad\text{for all}\quad k\in\mathbb{N}.
  \end{align*}
  Thus, we can apply Lemma \ref{L:WeDo} to find that
  \begin{align} \label{Pf_TW:Non}
    \lim_{k\to\infty}F'(v^{\varepsilon_k}) = F'(v_T) \quad\text{weakly in}\quad L^2(0,T;L^2(\Gamma)).
  \end{align}
  Now, let $\eta\in L^2(0,T;H^1(\Gamma))$.
  By \eqref{E:ACW_WFv} and \eqref{E:ACW_wWe}, we have
  \begin{align} \label{Pf_TW:Wk}
    \int_0^T\langle\partial_tv^{\varepsilon_k},g\eta\rangle_{H^1(\Gamma)}\,dt+\int_0^T(g\nabla_\Gamma\mu^{\varepsilon_k},\nabla_\Gamma\eta)_{L^2(\Gamma)}\,dt = R_{\varepsilon_k}(\eta;T)
  \end{align}
  and
  \begin{align*}
    \int_0^T(g\mu^{\varepsilon_k},\eta)_{L^2(\Gamma)}\,dt = \int_0^T(g\nabla_\Gamma v^{\varepsilon_k},\nabla_\Gamma\eta)_{L^2(\Gamma)}\,dt+\int_0^T(g\{F'(v^{\varepsilon_k})+\zeta^{\varepsilon_k}\},\eta)_{L^2(\Gamma)}\,dt.
  \end{align*}
  We send $\varepsilon_k\to0$ in these equalities and apply \eqref{Pf_TW:WC}, \eqref{Pf_TW:Non}, and
  \begin{align} \label{Pf_TW:Res}
    \begin{aligned}
      |R_{\varepsilon_k}(\eta;T)| &\leq c_T\varepsilon_k^{1-3\alpha}\|\nabla_\Gamma\eta\|_{L^2(0,T;L^2(\Gamma))} \to 0, \\
      \|\zeta^{\varepsilon_k}\|_{L^2(0,T;L^2(\Gamma))} &\leq c_T\varepsilon_k^{1-3\alpha} \to 0
    \end{aligned}
  \end{align}
  by \eqref{E:AWv_Res}, \eqref{E:AwS_Res}, and $\alpha\leq6/25$.
  Then, we find that $(v_T,\mu_T)$ satisfies \eqref{E:CLi_Wv} and \eqref{E:CLi_Wmu}.
  We also need to verify the initial condition on $v_T$.
  To this end, let
  \begin{align*}
    \eta(y,t) = f(t)\eta_0(y), \quad (y,t) \in \Gamma\times[0,T],
  \end{align*}
  where $f\in C^1([0,T])$ satisfies $f(0)=1$ and $f(T)=0$, and $\eta_0\in H^1(\Gamma)$ is any function.
  We take this $\eta$ in \eqref{Pf_TW:Wk} and apply \eqref{E:DtL2_gSu} and $v^{\varepsilon_k}(0)=v_0^{\varepsilon_k}$ a.e. on $\Gamma$ to get
  \begin{align*}
    -(v_0^{\varepsilon_k},g\eta_0)_{L^2(\Gamma)}-\int_0^T(v^{\varepsilon_k},g\partial_t\eta)_{L^2(\Gamma)}\,dt+\int_0^T(g\nabla_\Gamma\mu^\varepsilon,\nabla_\Gamma\eta)_{L^2(\Gamma)}\,dt = R_{\varepsilon_k}(\eta;T).
  \end{align*}
  Let $\varepsilon_k\to0$.
  Then, by the condition (b) of Theorem \ref{T:TFL_Weak}, \eqref{Pf_TW:WC}, and \eqref{Pf_TW:Res},
  \begin{align*}
    -(v_0,g\eta_0)_{L^2(\Gamma)}-\int_0^T(v_T,g\partial_t\eta)_{L^2(\Gamma)}\,dt+\int_0^T(g\nabla_\Gamma\mu_T,\nabla_\Gamma\eta)_{L^2(\Gamma)}\,dt = 0.
  \end{align*}
  On the other hand, we take the above $\eta$ in \eqref{E:CLi_Wv} and use \eqref{E:DtL2_gSu} to get
  \begin{align*}
    -(v_T(0),g\eta_0)_{L^2(\Gamma)}-\int_0^T(v_T,g\partial_t\eta)_{L^2(\Gamma)}\,dt+\int_0^T(g\nabla_\Gamma\mu_T,\nabla_\Gamma\eta)_{L^2(\Gamma)}\,dt = 0.
  \end{align*}
  From the above two relations, we deduce that
  \begin{align*}
    (v_T(0),g\eta_0)_{L^2(\Gamma)} = (v_0,g\eta_0)_{L^2(\Gamma)} \quad\text{for all}\quad \eta_0\in H^1(\Gamma),
  \end{align*}
  and $v_T(0)=v_0$ a.e. on $\Gamma$ by the density of $H^1(\Gamma)$ in $L^2(\Gamma)$ and \eqref{E:G_Bdd}.
  Hence, $(v_T,\mu_T)$ is a unique weak solution to \eqref{E:CH_Lim} on $[0,T)$, where the uniqueness is due to Proposition \ref{P:CLi_Uni}.
  The above discussions also show that, for any sequence $\{\varepsilon_\ell\}_\ell$ of positive numbers convergent to zero, there exists a subsequence $\{\varepsilon_k\}_k$ such that $\{(v^{\varepsilon_k},\mu^{\varepsilon_k})\}_k$ converges to the unique weak solution $(v_T,\mu_T)$ to \eqref{E:CH_Lim} on $[0,T)$ in the sense of \eqref{Pf_TW:WC}.
  Thus, we have the convergence \eqref{E:TFL_Weak} of the full sequence $\{(v^\varepsilon,\mu^\varepsilon)\}_\varepsilon$ to $(v_T,\mu_T)$.

  If $T<T'$, then $v_T=v_{T'}$ and $\mu_T=\mu_{T'}$ a.e. on $\Gamma\times(0,T)$ by the uniqueness of a weak solution to \eqref{E:CH_Lim} on $[0,T)$.
  Hence, setting $v:=v_T$ and $\mu:=\mu_T$ on $\Gamma\times(0,T)$ for each $T>0$, we find that the convergence result \eqref{E:TFL_Weak} holds for all $T>0$ and $(v,\mu)$ is a unique global weak solution to \eqref{E:CH_Lim}.
\end{proof}

\begin{remark} \label{R:TFLW}
  If we use the estimate \eqref{E:AwR_Wor} for $\nabla_\Gamma\zeta^\varepsilon$ instead of \eqref{E:AwS_Res}, then we have the estimates in Section \ref{SS:TFL_AEn} with $\varepsilon^{2-25\alpha/3}$ replaced by $\varepsilon^{2-10\alpha}$.
  Thus, in this case, it is necessary to assume $\alpha\leq1/5$ to obtain the results of Theorem \ref{T:TFL_Weak}, which is weaker than the condition $\alpha\leq6/25$ in Theorem \ref{T:TFL_Weak}.
  Recall that \eqref{E:AwR_Wor} comes from a natural idea to estimate $\nabla_\Gamma\zeta_G^\varepsilon$, as explained in Remarks \ref{R:ANL_TGr} and \ref{R:ACW_wSt}.
  Also, we note that it is not clear whether the bound $6/25$ is optimal, but we do not touch this problem in this paper.
\end{remark}

As a consequence of Theorem \ref{T:TFL_Weak}, we also have the following result.

\begin{corollary} \label{C:CLi_Ex}
  Suppose that Assumption \ref{A:Poten} is satisfied.
  Then, for each $v_0\in H^1(\Gamma)$, there exists a unique global weak solution to \eqref{E:CH_Lim}.
\end{corollary}

\begin{proof}
  Let $v_0\in H^1(\Gamma)$ be any function.
  For $\varepsilon>0$, we define
  \begin{align*}
    u_0^\varepsilon(x) := \frac{\bar{v}_0(x)}{J\bigl(\pi(x),d(x)\bigr)} = \frac{v_0(y)}{J(y,r)}, \quad x = y+r\bm{\nu}(y) \in \Omega_\varepsilon.
  \end{align*}
  Then, $\mathcal{M}_\varepsilon u_0^\varepsilon=v_0$ on $\Gamma$.
  Since $|F(z)|\leq c(|z|^4+1)$ by Remark \ref{R:Po_Grow}, we also have
  \begin{align*}
    \|u_0^\varepsilon\|_{L^2(\Omega_\varepsilon)}^2+|E_\varepsilon(u_0^\varepsilon)| &\leq c\Bigl(\|u_0^\varepsilon\|_{L^2(\Omega_\varepsilon)}^2+\|\nabla u_0^\varepsilon\|_{L^2(\Omega_\varepsilon)}^2+\|u_0^\varepsilon\|_{L^4(\Omega_\varepsilon)}^4+|\Omega_\varepsilon|\Bigr) \\
    &\leq c\varepsilon\Bigl(\|v_0\|_{H^1(\Gamma)}^2+\|v_0\|_{L^4(\Gamma)}^4+1\Bigr) \leq c\varepsilon\Bigl(\|v_0\|_{H^1(\Gamma)}^2+1\Bigr)
  \end{align*}
  by \eqref{E:CEGr_Bd}, \eqref{E:J_Bdd}, \eqref{E:Lp_CE}, and \eqref{E:SoSu_H1}.
  Hence, the conditions (a) and (b) in Theorem \ref{T:TFL_Weak} hold and we get a unique global weak solution to \eqref{E:CH_Lim} by Theorem \ref{T:TFL_Weak}.
\end{proof}

\subsection{Difference estimates} \label{SS:TFL_DE}
Let us estimate the difference of $\mathcal{M}_\varepsilon u^\varepsilon$ and $v$ on $\Gamma$.
Recall that $L_g$ is the operator given in Lemma \ref{L:InA_Sur}.

\begin{theorem} \label{T:DES_Pre}
  Under the assumptions of Theorem \ref{T:DiE_Sur}, we have
  \begin{multline} \label{E:DES_Pre}
    \|\nabla_\Gamma L_g(\mathcal{M}_\varepsilon u^\varepsilon-v)\|_{C([0,T];L^2(\Gamma))}+\|\mathcal{M}_\varepsilon u^\varepsilon-v\|_{L^2(0,T;H^1(\Gamma))} \\
    \leq ce^{cT}\Bigl\{\|\nabla_\Gamma L_g(\mathcal{M}_\varepsilon u_0^\varepsilon-v_0)\|_{L^2(\Gamma)}+\varepsilon^{1-3\alpha}(1+T)^3\Bigr\}
  \end{multline}
  for all $T>0$, where $c>0$ is a constant independent of $\varepsilon$ and $T$.
\end{theorem}

\begin{proof}
  Let $v^\varepsilon=\mathcal{M}_\varepsilon u^\varepsilon$, $\mu^\varepsilon=\mathcal{M}_\varepsilon w^\varepsilon$ on $\Gamma\times(0,\infty)$, $v_0^\varepsilon=\mathcal{M}_\varepsilon u_0^\varepsilon$ on $\Gamma$, and
  \begin{align*}
    \chi_v^\varepsilon := v^\varepsilon-v, \quad \chi_\mu^\varepsilon := \mu^\varepsilon-\mu \quad\text{on}\quad \Gamma\times(0,\infty), \quad \chi_{v,0}^\varepsilon := v_0^\varepsilon-v_0 \quad\text{on}\quad \Gamma.
  \end{align*}
  For each $T>0$, we see by Propositions \ref{P:ACW_Reg}, \ref{P:ACW_WFv}, \ref{P:ACW_wWe}, and Definition \ref{D:CLi_GlW} that
  \begin{align*}
    \chi_v^\varepsilon\in\mathcal{E}_T(\Gamma)\cap L^\infty(0,T;H^1(\Gamma)), \quad \chi_\mu^\varepsilon\in L^2(0,T;H^1(\Gamma)),
  \end{align*}
  and $\chi_v^\varepsilon(0)=\chi_{v,0}^\varepsilon$ a.e. on $\Gamma$, and for all $\eta\in L^2(0,T;H^1(\Gamma))$,
  \begin{align} \label{Pf_DS:v}
    \int_0^T\langle\partial_t\chi_v^\varepsilon,g\eta\rangle_{H^1(\Gamma)}\,dt+\int_0^T(g\nabla_\Gamma\chi_\mu^\varepsilon,\nabla_\Gamma\eta)_{L^2(\Gamma)}\,dt = R_\varepsilon(\eta;T)
  \end{align}
  and
  \begin{multline} \label{Pf_DS:mu}
    \int_0^T(g\chi_\mu^\varepsilon,\eta)_{L^2(\Gamma)}\,dt = \int_0^T(g\nabla_\Gamma\chi_v^\varepsilon,\nabla_\Gamma\eta)_{L^2(\Gamma)}\,dt\\
    +\int_0^T(g\{F'(v^\varepsilon)-F'(v)\},\eta)_{L^2(\Gamma)}\,dt+\int_0^T(g\zeta^\varepsilon,\eta)_{L^2(\Gamma)}\,dt.
  \end{multline}
  By \eqref{E:DES_Ini}, we have $\int_\Gamma g\chi_{v,0}^\varepsilon\,d\mathcal{H}^2=0$.
  Thus, noting that $R_\varepsilon(\eta;t)=0$ for $\eta\equiv1$ by \eqref{E:AWv_Res}, we see by \eqref{E:DtL2_gSu} and \eqref{Pf_DS:v} with $\eta\equiv1$ and $T$ replaced by $t\in[0,T]$ that
  \begin{align} \label{Pf_DS:TA}
    \int_\Gamma g\chi_v^\varepsilon(t)\,d\mathcal{H}^2 = \int_\Gamma g\chi_v^\varepsilon(t)\,d\mathcal{H}^2-\int_\Gamma g\chi_{v,0}^\varepsilon\,d\mathcal{H}^2 = \int_0^t\langle\partial_t\chi_v^\varepsilon,g\rangle_{H^1(\Gamma)}\,ds = 0.
  \end{align}
  Hence, we can apply $L_g$ to $\chi_v^\varepsilon(t)$, and we can deduce that
  \begin{multline*}
    \frac{1}{2}\|\sqrt{g}\,\nabla_\Gamma L_g\chi_v^\varepsilon(t)\|_{L^2(\Gamma)}^2+\int_0^t\|\sqrt{g}\,\nabla_\Gamma\chi_v^\varepsilon\|_{L^2(\Gamma)}^2\,ds \\
    \leq \frac{1}{2}\|\sqrt{g}\,\nabla_\Gamma L_g\chi_{v,0}^\varepsilon\|_{L^2(\Gamma)}^2+C_2\int_0^t\|\sqrt{g}\,\chi_v^\varepsilon\|_{L^2(\Gamma)}^2\,ds \\
    +|R_\varepsilon(L_g\chi_v^\varepsilon;t)|+\left|\int_0^t(g\zeta^\varepsilon,\chi_v^\varepsilon)_{L^2(\Gamma)}\,ds\right|
  \end{multline*}
  as in the proof of Proposition \ref{P:CLi_Uni} (see \eqref{Pf_CLU:Ine}), by setting $\eta=L_g\chi_v^\varepsilon$ in \eqref{Pf_DS:v} and $\eta=\chi_v^\varepsilon$ in \eqref{Pf_DS:mu} with $T$ replaced by $t$.
  Moreover,
  \begin{align*}
    |R_\varepsilon(L_g\chi_v^\varepsilon;t)| &\leq c\varepsilon^{1-3\alpha}(1+t)^3\|\sqrt{g}\,\nabla_\Gamma L_g\chi_v^\varepsilon\|_{L^2(0,t;L^2(\Gamma))}, \\
    \left|\int_0^t(g\zeta^\varepsilon,\chi_v^\varepsilon)_{L^2(\Gamma)}\,ds\right| &\leq c\varepsilon^{1-3\alpha}(1+t)^3\|\sqrt{g}\,\chi_v^\varepsilon\|_{L^2(0,t;L^2(\Gamma))}
  \end{align*}
  by \eqref{E:G_Bdd}, \eqref{E:AWv_Res}, \eqref{E:AwS_Res}, and H\"{o}lder's inequality.
  Thus, by Young's inequality,
  \begin{multline*}
    \frac{1}{2}\|\sqrt{g}\,\nabla_\Gamma L_g\chi_v^\varepsilon(t)\|_{L^2(\Gamma)}^2+\int_0^t\|\sqrt{g}\,\nabla_\Gamma\chi_v^\varepsilon\|_{L^2(\Gamma)}^2\,ds \\
    \leq \frac{1}{2}\|\sqrt{g}\,\nabla_\Gamma L_g\chi_{v,0}^\varepsilon\|_{L^2(\Gamma)}^2+c\left(\int_0^t\|\sqrt{g}\,\nabla_\Gamma L_g\chi_v^\varepsilon\|_{L^2(\Gamma)}^2\,ds+\varepsilon^{2-6\alpha}(1+t)^6\right) \\
    +C\int_0^t\|\sqrt{g}\,\chi_v^\varepsilon\|_{L^2(\Gamma)}^2\,ds
  \end{multline*}
  with constants $c,C>0$ independent of $\varepsilon$, $t$, and $T$.
  To the last term, we apply \eqref{Pf_CLU:L2} with $C_2$ and $\chi_v$ replaced by $C$ and $\chi_v^\varepsilon$, respectively.
  Then, we find that
  \begin{multline} \label{Pf_DS:Fin}
    \frac{1}{2}\|\sqrt{g}\,\nabla_\Gamma L_g\chi_v^\varepsilon(t)\|_{L^2(\Gamma)}^2+\frac{1}{2}\int_0^t\|\sqrt{g}\,\nabla_\Gamma\chi_v^\varepsilon\|_{L^2(\Gamma)}^2\,ds \\
    \leq \frac{1}{2}\|\sqrt{g}\,\nabla_\Gamma L_g\chi_{v,0}^\varepsilon\|_{L^2(\Gamma)}^2+c\left(\int_0^t\|\sqrt{g}\,\nabla_\Gamma L_g\chi_v^\varepsilon\|_{L^2(\Gamma)}^2\,ds+\varepsilon^{2-6\alpha}(1+t)^6\right).
  \end{multline}
  Since we can apply \eqref{E:Poi_Sur} to $\chi_v^\varepsilon(t)$ by \eqref{Pf_DS:TA}, we also have
  \begin{align*}
    \|\chi_v^\varepsilon(t)\|_{H^1(\Gamma)} \leq c\|\nabla_\Gamma\chi_v^\varepsilon(t)\|_{L^2(\Gamma)} \quad\text{for all}\quad t\in[0,T].
  \end{align*}
  Applying this inequality, \eqref{E:G_Bdd}, and $(1+t)^6\leq(1+T)^6$ to \eqref{Pf_DS:Fin}, we obtain
  \begin{multline*}
    \|\nabla_\Gamma L_g\chi_v^\varepsilon(t)\|_{L^2(\Gamma)}^2+\int_0^t\|\chi_v^\varepsilon\|_{H^1(\Gamma)}^2\,ds \\
    \leq c\left(\|\nabla_\Gamma L_g\chi_{v,0}^\varepsilon\|_{L^2(\Gamma)}^2+\varepsilon^{2-6\alpha}(1+T)^6+\int_0^t\|\nabla_\Gamma L_g\chi_v^\varepsilon\|_{L^2(\Gamma)}^2\,ds\right)
  \end{multline*}
  for all $t\in[0,T]$.
  Hence, \eqref{E:DES_Pre} follows from Gronwall's inequality.
\end{proof}

\begin{proof}[Proof of Theorem \ref{T:DiE_Sur}]
  Since $\int_\Gamma g(\mathcal{M}_\varepsilon u_0^\varepsilon-v_0)\,d\mathcal{H}^2=0$ by \eqref{E:DES_Ini}, we have
  \begin{align*}
    \|\nabla_\Gamma L_g(\mathcal{M}_\varepsilon u_0^\varepsilon-v_0)\|_{L^2(\Gamma)} \leq \|\mathcal{M}_\varepsilon u_0^\varepsilon-v_0\|_{[H^1(\Gamma)]'} \leq \|\mathcal{M}_\varepsilon u_0^\varepsilon-v_0\|_{L^2(\Gamma)}
  \end{align*}
  by \eqref{E:IAS_Bdd}.
  By this inequality and \eqref{E:DES_Pre}, we get \eqref{E:DiE_Sur}.
\end{proof}

We also derive the difference estimate \eqref{E:DiE_CTD} of $u^\varepsilon$ and $v$ on $\Omega_\varepsilon$.

\begin{proof}[Proof of Theorem \ref{T:DiE_CTD}]
  It follows from \eqref{E:ADif_Lp}, \eqref{E:ATG_L2D}, and \eqref{E:TFLE_H1} that
  \begin{align*}
    &\varepsilon^{-1/2}\Bigl(\left\|u^\varepsilon-\overline{\mathcal{M}_\varepsilon u^\varepsilon}\right\|_{L^2(0,T;L^2(\Omega_\varepsilon))}+\left\|\overline{\mathbf{P}}\nabla u^\varepsilon-\overline{\nabla_\Gamma\mathcal{M}_\varepsilon u^\varepsilon}\right\|_{L^2(0,T;L^2(\Omega_\varepsilon))}\Bigr) \\
    &\leq c\varepsilon^{1/2}\|u^\varepsilon\|_{L^2(0,T;H^2(\Omega_\varepsilon))} \leq c\varepsilon^{1-\alpha}(1+T).
  \end{align*}
  Also, we use \eqref{E:Lp_CE} and then \eqref{E:DiE_Sur} to find that
  \begin{align*}
    &\varepsilon^{-1/2}\Bigl(\left\|\overline{\mathcal{M}_\varepsilon u^\varepsilon}-\bar{v}\right\|_{L^2(0,T;L^2(\Omega_\varepsilon))}+\left\|\overline{\nabla_\Gamma\mathcal{M}_\varepsilon u^\varepsilon}-\overline{\nabla_\Gamma v}\right\|_{L^2(0,T;L^2(\Omega_\varepsilon))}\Bigr) \\
    &\leq c\|\mathcal{M}_\varepsilon u^\varepsilon-v\|_{L^2(0,T;H^1(\Gamma))} \leq ce^{cT}\Bigl\{\|\mathcal{M}_\varepsilon u_0^\varepsilon-v_0\|_{L^2(\Gamma)}+\varepsilon^{1-3\alpha}(1+T)^3\Bigr\}.
  \end{align*}
  Combining these estimates, we get \eqref{E:DiE_CTD}.
  Also, since $\partial_{\nu_\varepsilon}u^\varepsilon(t)=0$ on $\partial\Omega_\varepsilon$ for a.a. $t>0$ by Proposition \ref{P:CHT_Reg}, we can use \eqref{E:ND_CTD} and \eqref{E:TFLE_H1} to obtain
  \begin{align*}
    \varepsilon^{-1/2}\|\partial_\nu u^\varepsilon\|_{L^2(0,T;L^2(\Omega_\varepsilon))} \leq c\varepsilon^{1/2}\|u^\varepsilon\|_{L^2(0,T;H^2(\Omega_\varepsilon))} \leq c\varepsilon^{1-\alpha}(1+T).
  \end{align*}
  Hence, \eqref{E:DiE_ND} is valid.
\end{proof}

\section{Uniform elliptic regularity estimate} \label{S:Uni_ER}
This section is devoted to the proof of Lemma \ref{L:UER_CTD}.
As mentioned after the lemma, the proof requires long and careful discussions, since we need to show that the constant $c$ in \eqref{E:UER_CTD} is independent of $\varepsilon$.
Thus, we emphasize that the constants appearing below are all independent of $\varepsilon$.
We split the proof in several subsections below.

\subsection{Preliminaries} \label{SS:UER_Pr}
First, we give some auxiliary results.

\begin{lemma} \label{L:Op_Cov}
  Let $\mathcal{O}:=(0,1)^2$.
  There exist a finite number of $C^5$ local parametrizations $\psi_\ell\colon\mathcal{O}\to\Gamma$ with $\ell=1,\dots,L$ such that $\{\psi_\ell(\mathcal{O})\}_{\ell=1}^L$ is an open cover of $\Gamma$.
\end{lemma}

\begin{proof}
  Since $\Gamma$ is a $C^5$ surface without boundary, for each $y\in\Gamma$, there exist an open set $\mathcal{U}_y$ in $\mathbb{R}^2$ and a $C^5$ local parametrization $\psi_y\colon\mathcal{U}_y\to\Gamma$ such that $y\in\psi_y(\mathcal{U}_y)$.
  Then, since $\mathcal{U}_y$ is open in $\mathbb{R}^2$, we can take an open rectangular
  \begin{align*}
    \mathcal{O}_y = (a_y^1,b_y^1)\times(a_y^2,b_y^2) \subset \mathcal{U}_y
  \end{align*}
  such that $y\in\psi_y(\mathcal{O}_y)$.
  We restrict $\psi_y$ to $\mathcal{O}_y$ and replace it with
  \begin{align*}
    \tilde{\psi}_y(s_1,s_2) := \psi_y((1-s_1)a_y^1+s_1b_y^1,(1-s_2)a_y^2+s_2b_y^2), \quad (s_1,s_2) \in \mathcal{O} = (0,1)^2.
  \end{align*}
  Then, $\psi_y$ is defined on $\mathcal{O}$ and $y\in\psi_y(\mathcal{O})$.
  Since $\Gamma$ is compact and $\{\psi_y(\mathcal{O})\}_{y\in\Gamma}$ is an open cover of $\Gamma$, it has a finite subcover.
  Thus, by relabeling, we get a finite number of $C^5$ local parametrizations $\psi_\ell\colon\mathcal{O}\to\Gamma$ such that $\{\psi_\ell(\mathcal{O})\}_{\ell=1}^L$ is an open cover of $\Gamma$.
\end{proof}

Let $\mathcal{O}=(0,1)^2$.
We take a $C^5$ local parametrization and its inverse mapping
\begin{align*}
  \psi\colon\mathcal{O} \to \psi(\mathcal{O})\subset\Gamma, \quad \psi^{-1} = (\psi_1^{-1},\psi_2^{-1})\colon\psi(\mathcal{O}) \to \mathcal{O}.
\end{align*}
Let $\theta=(\theta_{ij})_{i,j}$ be the Riemannian metric of $\Gamma$ on $\mathcal{O}$ given by
\begin{align} \label{E:Def_Rie}
  \theta_{ij}(s') := \partial_{s_i}\psi(s')\cdot\partial_{s_j}\psi(s'), \quad s'=(s_1,s_2) \in \mathcal{O}, \quad i,j=1,2,
\end{align}
and let $\theta^{-1}=(\theta^{ij})_{i,j}$ be the inverse matrix of $\theta$.
We write $\zeta^\flat:=\zeta\circ\psi$ on $\mathcal{O}$ for a function $\zeta$ on $\Gamma$.
Let us construct a local parametrization of $\overline{\Omega_\varepsilon}$ by using $\psi$ and give related results.
Recall that $\bar{\eta}$ is the constant extension of a function $\eta$ on $\Gamma$ in the normal direction of $\Gamma$.

\begin{lemma} \label{L:LP_CTD}
  For $s=(s',s_3)\in\mathcal{O}\times[0,\varepsilon]$, we define
  \begin{align} \label{E:Def_LPC}
    \Psi_\varepsilon(s) := \psi(s')+r_\varepsilon(s)\bm{\nu}^\flat(s'), \quad r_\varepsilon(s) := (\varepsilon-s_3)g_0^\flat(s')+s_3g_1^\flat(s').
  \end{align}
  Then, $\Psi_\varepsilon\colon\mathcal{O}\times[0,\varepsilon]\to\Psi_\varepsilon(\mathcal{O}\times[0,\varepsilon])\subset\overline{\Omega_\varepsilon}$ is invertible with inverse mapping
  \begin{align} \label{E:LPC_Inv}
    \Psi_\varepsilon^{-1}(x) = \left(\bar{\psi}_1^{-1}(x),\bar{\psi}_2^{-1}(x),\frac{d(x)-\varepsilon\bar{g}_0(x)}{\bar{g}(x)}\right), \quad x \in \Psi_\varepsilon(\mathcal{O}\times[0,\varepsilon]).
  \end{align}
  Moreover, if $\mathcal{K}$ is a compact subset of $\mathcal{O}$, then
  \begin{align} \label{E:LPC_Bd}
    \begin{aligned}
      |\partial_s^\alpha\Psi_\varepsilon(s)| &\leq c, \quad s\in\mathcal{K}\times[0,\varepsilon], \quad |\alpha| = 1,2,3, \\
      |\partial_x^\alpha\Psi_\varepsilon^{-1}(x)| &\leq c, \quad x\in\Psi_\varepsilon(\mathcal{K}\times[0,\varepsilon]), \quad |\alpha| = 1,2,3.
    \end{aligned}
  \end{align}
  Here, $\alpha=(\alpha_1,\alpha_2,\alpha_3)^{\mathrm{T}}\in\mathbb{Z}_{\geq0}^3$ is a multi-index and
  \begin{align*}
    \partial_s^\alpha = \partial_{s_1}^{\alpha_1}\partial_{s_2}^{\alpha_2}\partial_{s_3}^{\alpha_3}, \quad \partial_x^\alpha = \partial_{x_1}^{\alpha_1}\partial_{x_2}^{\alpha_2}\partial_{x_3}^{\alpha_3}, \quad |\alpha| = \alpha_1+\alpha_2+\alpha_3.
  \end{align*}
\end{lemma}

\begin{proof}
  Let $x=\Psi_\varepsilon(s)$.
  Then, we see by \eqref{E:Fermi} that $\pi(x)=\psi(s')$ and
  \begin{align*}
   d(x) = r_\varepsilon(s) = (\varepsilon-s_3)g_0(\psi(s'))+s_3g_1(\psi(s')) = (\varepsilon-s_3)\bar{g}_0(x)+s_3\bar{g}_1(x).
  \end{align*}
  Solving these equations with respect to $s$, we obtain \eqref{E:LPC_Inv}.

  Let $\mathcal{K}$ be a compact subset of $\mathcal{O}$.
  Then, since $\psi$, $\bm{\nu}^\flat$, $g_0^\flat$, and $g_1^\flat$ are of class $C^3$ on $\mathcal{O}$, and since $r_\varepsilon$ is affine in $s_3$ and $0<\varepsilon<1$, we have the first line of \eqref{E:LPC_Bd}.
  Also, since $\psi^{-1}$ is of class $C^5$ on $\psi(\mathcal{O})$ and $\psi(\mathcal{K})$ is compact in $\psi(\mathcal{O})$, it follows that
  \begin{align*}
    |\nabla_\Gamma^\ell\psi_i^{-1}(y)| \leq c, \quad y\in\psi(\mathcal{K}), \quad \ell = 1,2,3, \quad i=1,2.
  \end{align*}
  Moreover, $d\in C^5(\overline{\mathcal{N}_\delta})$ and $g_0,g_1\in C^3(\Gamma)$.
  By these facts, \eqref{E:CEGr_Bd}, \eqref{E:G_Bdd}, and $0<\varepsilon<1$, we find that the second line of \eqref{E:LPC_Bd} is valid.
\end{proof}

\begin{remark} \label{R:LP_CTD}
  The reason why we do not scale $[0,\varepsilon]$ to $[0,1]$ is because we would like to avoid having the negative power of $\varepsilon$ when we differentiate $\Psi_\varepsilon^{-1}$.
  If we set
  \begin{align*}
    \zeta_\varepsilon(s) := \psi(s')+\varepsilon r_1(s)\bm{\nu}^\flat(s'), \quad r_1(s) := (1-s_3)g_0^\flat(s')+s_3g_0^\flat(s'),
  \end{align*}
  then $\zeta_\varepsilon$ maps $\mathcal{O}\times[0,1]$ into $\overline{\Omega_\varepsilon}$, but its inverse mapping is
  \begin{align*}
    \zeta_\varepsilon^{-1}(x) = \left(\bar{\psi}_1^{-1}(x),\bar{\psi}_2^{-1}(x),\frac{d(x)-\bar{g}_0(x)}{\varepsilon\bar{g}(x)}\right), \quad x \in \zeta_\varepsilon(\mathcal{O}\times[0,1]).
  \end{align*}
  Since $\nabla d=\bar{\bm{\nu}}$ is of order one in $\overline{\Omega_\varepsilon}$ (although $d$ is of order $\varepsilon$), the derivatives of the third component of $\zeta_\varepsilon^{-1}$ is of order $\varepsilon^{-1}$.
  Hence, if we use $\zeta_\varepsilon$ in the following discussions, then we are enforced to balance this $\varepsilon^{-1}$ with some positive power of $\varepsilon$ in order to derive uniform estimates.
  We can avoid such a terrible task if we do not scale $[0,\varepsilon]$.

  Of course, sometimes it is more convenient to scale $[0,\varepsilon]$ to $[0,1]$.
  For example, we use $\zeta_\varepsilon$ to prove Lemma \ref{L:Ipl_CTD} (see Section \ref{SS:PA_28} below), since we need to rely on a result on a domain with fixed thickness but it is not required to compute the derivatives of $\zeta_\varepsilon^{-1}$.
\end{remark}

\begin{lemma} \label{L:Det_LPC}
  Let $\nabla_s\Psi_\varepsilon$ be the gradient matrix of $\Psi_\varepsilon$, which is written as
  \begin{align*}
    \nabla_s\Psi_\varepsilon = (\partial_{s_i}\Psi_{\varepsilon,j})_{i,j}
    =
    \begin{pmatrix}
      \partial_{s_1}\Psi_\varepsilon & \partial_{s_2}\Psi_\varepsilon & \partial_{s_3}\Psi_\varepsilon
    \end{pmatrix}^{\mathrm{T}}
  \end{align*}
  under our notation.
  Here, $\Psi_{\varepsilon,j}$ is the $j$-th component of $\Psi_\varepsilon$ and $\partial_{s_i}\Psi_\varepsilon\in\mathbb{R}^3$ is considered as a column vector.
  Let $J$ be the function given by \eqref{E:Def_J}.
  Then,
  \begin{align} \label{E:Det_LPC}
    \det\nabla_s\Psi_\varepsilon(s) = g^\flat(s')J(\psi(s'),r_\varepsilon(s))\sqrt{\det\theta(s')}, \quad s = (s',s_3) \in \mathcal{O}\times[0,\varepsilon].
  \end{align}
  Moreover, when $\mathcal{K}$ is a compact subset of $\mathcal{O}$, we have
  \begin{align}
    c^{-1} \leq \det\nabla_s\Psi_\varepsilon(s) &\leq c, \quad s\in\mathcal{K}\times[0,\varepsilon], \label{E:DeLC_Bd} \\
    |[\partial_s^\alpha(\det\nabla_s\Psi_\varepsilon)](s)| &\leq c, \quad s\in\mathcal{K}\times[0,\varepsilon], \quad |\alpha| = 1,2,3. \label{E:DeLC_ds}
  \end{align}
\end{lemma}

\begin{proof}
  Let $\eta$ be a function on $\Gamma$.
  Then, for $s'\in\mathcal{O}$ and $i=1,2$, we have
  \begin{align*}
    \eta^\flat(s') = \eta(\psi(s')) = \bar{\eta}(\psi(s')), \quad \partial_{s_i}\eta^\flat(s') = \partial_{s_i}\psi(s')\cdot\nabla\bar{\eta}(\psi'(s)).
  \end{align*}
  Moreover, it follows from \eqref{E:CEGr_Sur} with $y=\psi(s')\in\Gamma$ that
  \begin{align*}
    \nabla\bar{\eta}(\psi'(s)) = \nabla_\Gamma\eta(\psi(s')) = (\nabla_\Gamma\eta)^\flat(s').
  \end{align*}
  Hence, $\partial_{s_i}\eta^\flat=\partial_{s_i}\psi\cdot(\nabla_\Gamma\eta)^\flat$ on $\mathcal{O}$.
  Also, $-\nabla_\Gamma\bm{\nu}=\mathbf{W}=\mathbf{W}^{\mathrm{T}}$ on $\Gamma$.
  Thus,
  \begin{align*}
    \partial_{s_i}\Psi_\varepsilon = (\mathbf{I}_3-r_\varepsilon\mathbf{W}^\flat)\partial_{s_i}\psi+(\partial_{s_i}r_\varepsilon)\bm{\nu}^\flat, \quad \partial_{s_3}\Psi_\varepsilon = g^\flat\bm{\nu}^\flat
  \end{align*}
  in $\mathcal{O}\times[0,\varepsilon]$ with $i=1,2$.
  Moreover, since $\mathbf{W}^\flat\bm{\nu}^\flat=0$, we can write
  \begin{align*}
    \partial_{s_i}\Psi_\varepsilon = (\mathbf{I}_3-r_\varepsilon\mathbf{W}^\flat)\{\partial_{s_i}\psi+(\partial_{s_i}r_\varepsilon)\bm{\nu}^\flat\}, \quad \partial_{s_3}\Psi_\varepsilon = (\mathbf{I}_3-r_\varepsilon\mathbf{W}^\flat)(g^\flat\bm{\nu}^\flat).
  \end{align*}
  Hence, $\nabla_s\Psi_\varepsilon$ is expressed as
  \begin{align*}
    \nabla_s\Psi_\varepsilon =
    \begin{pmatrix}
      \partial_{s_1}\Psi_\varepsilon & \partial_{s_2}\Psi_\varepsilon & \partial_{s_3}\Psi_\varepsilon
    \end{pmatrix}^{\mathrm{T}}
    = [(\mathbf{I}_3-r_\varepsilon\mathbf{W}^\flat)\mathbf{A}_1]^{\mathrm{T}}
  \end{align*}
  in $\mathcal{O}\times[0,\varepsilon]$, where the matrix $\mathbf{A}_1$ is given by
  \begin{align*}
    \mathbf{A}_1 :=
    \begin{pmatrix}
      \partial_{s_1}\psi+(\partial_{s_1}r_\varepsilon)\bm{\nu}^\flat & \partial_{s_2}\psi+(\partial_{s_2}r_\varepsilon)\bm{\nu}^\flat & g^\flat\bm{\nu}^\flat
    \end{pmatrix}.
  \end{align*}
  Recall that we consider vectors in $\mathbb{R}^3$ as column vectors.
  Now, we see that
  \begin{align*}
    \det\mathbf{A}_1 = \det\mathbf{A}_2, \quad \mathbf{A}_2 :=
    \begin{pmatrix}
      \partial_{s_1}\psi & \partial_{s_2}\psi & g^\flat\bm{\nu}^\flat
    \end{pmatrix}
  \end{align*}
  by elementary column operations.
  Also, letting $\mathbf{0}_2:=(0,0)^{\mathrm{T}}\in\mathbb{R}^2$, we have
  \begin{align*}
    (\det\mathbf{A}_2)^2 = \det(\mathbf{A}_2^{\mathrm{T}}\mathbf{A}_2) = \det
    \begin{pmatrix}
      \theta & \mathbf{0}_2 \\
      \mathbf{0}_2^{\mathrm{T}} & (g^\flat)^2
    \end{pmatrix}
    = (g^\flat)^2\det\theta
  \end{align*}
  by $\partial_{s_i}\psi\cdot\bm{\nu}^\flat=0$ for $i=1,2$.
  By these relations and \eqref{E:Def_J}, we find that
  \begin{align*}
    \det\nabla_s\Psi_\varepsilon = \det(\mathbf{I}_3-r_\varepsilon\mathbf{W}^\flat)\cdot\det\mathbf{A}_1 = g^\flat J(\psi,r_\varepsilon)\sqrt{\det\theta},
  \end{align*}
  i.e. \eqref{E:Det_LPC} holds.
  Also, let $\mathcal{K}$ be a compact subset of $\mathcal{O}$.
  Then,
  \begin{align} \label{Pf_DLP:det}
    c^{-1} \leq \det\theta(s') \leq c, \quad s'\in\mathcal{K},
  \end{align}
  since $\det\theta$ is continuous and positive on $\mathcal{O}$.
  Applying \eqref{E:G_Bdd}, \eqref{E:J_Bdd}, and \eqref{Pf_DLP:det} to \eqref{E:Det_LPC}, we obtain \eqref{E:DeLC_Bd}.
  Also, since $r_\varepsilon(s)$ is affine in $s_3$, and since $g^\flat(s')$, $\theta(s')$, and
  \begin{align*}
    J(\psi(s'),r_\varepsilon(s)) = \{1-r_\varepsilon(s)\kappa_1^\flat(s')\}\{1-r_\varepsilon(s)\kappa_2^\flat(s')\}
  \end{align*}
  are of class $C^3$ in $s'\in\mathcal{O}$, we have \eqref{E:DeLC_ds} by the compactness of $\mathcal{K}$ in $\mathcal{O}$ and \eqref{Pf_DLP:det}.
\end{proof}

\begin{lemma} \label{L:Hk_LPC}
  Let $k=0,1,2,3$ and $u\in H^k(\Omega_\varepsilon)$ be supported in $\Psi_\varepsilon(\mathcal{K}\times[0,\varepsilon])$ with some compact subset $\mathcal{K}$ of $\mathcal{O}$.
  We set $U:=u\circ\Psi_\varepsilon$ on $\mathcal{V}_\varepsilon:=\mathcal{O}\times(0,\varepsilon)$.
  Then,
  \begin{align} \label{E:Hk_LPC}
    U \in H^k(\mathcal{V}_\varepsilon), \quad c^{-1}\|U\|_{H^k(\mathcal{V}_\varepsilon)} \leq \|u\|_{H^k(\Omega_\varepsilon)} \leq c\|U\|_{H^k(\mathcal{V}_\varepsilon)}.
  \end{align}
\end{lemma}

\begin{proof}
  When $k=0$, we have \eqref{E:Hk_LPC} by $\mathrm{supp}\,u\subset\Psi_\varepsilon(\mathcal{K}\times[0,\varepsilon])$ and \eqref{E:DeLC_Bd}.

  Let $k=1,2,3$.
  We differentiate $U(s)=u(\Psi_\varepsilon(s))$ and use \eqref{E:LPC_Bd} to get
  \begin{align*}
    |\partial_s^\alpha U(s)| \leq \sum_{|\beta|\leq|\alpha|}|\partial_x^\beta u(\Psi_\varepsilon(s))|, \quad s\in\mathcal{K}\times(0,\varepsilon), \quad |\alpha| = 1,2,3.
  \end{align*}
  Also, since $u(x)=U(\Psi_\varepsilon^{-1}(x))$ for $x\in\Psi_\varepsilon(\mathcal{V}_\varepsilon)$, it follows from \eqref{E:LPC_Bd} that
  \begin{align*}
    |\partial_x^\alpha u(x)| \leq \sum_{|\beta|\leq|\alpha|}|\partial_s^\beta U(\Psi_\varepsilon^{-1}(x))|, \quad x\in\Psi_\varepsilon(\mathcal{K}\times(0,\varepsilon)), \quad |\alpha| = 1,2,3.
  \end{align*}
  We get \eqref{E:Hk_LPC} by these inequalities, \eqref{E:DeLC_Bd}, and $\mathrm{supp}\,u\subset\Psi_\varepsilon(\mathcal{K}\times[0,\varepsilon])$.
\end{proof}

\begin{lemma} \label{L:LPC_We}
  Let $k=0,1$.
  Suppose that the functions
  \begin{align*}
    u \in H^1(\Omega_\varepsilon), \quad f\in H^k(\Omega_\varepsilon), \quad \mathbf{v} \in H^{1+k}(\Omega_\varepsilon)^3
  \end{align*}
  are supported in $\Psi_\varepsilon(\mathcal{K}\times[0,\varepsilon])$ with a compact subset $\mathcal{K}$ of $\mathcal{O}$ and satisfy
  \begin{align} \label{E:LPW_Ori}
    \int_{\Omega_\varepsilon}\nabla u\cdot\nabla\varphi\,dx = \int_{\Omega_\varepsilon}(f\varphi+\mathbf{v}\cdot\nabla\varphi)\,dx \quad\text{for all}\quad \varphi\in H^1(\Omega_\varepsilon).
  \end{align}
  For $s\in \mathcal{V}_\varepsilon=\mathcal{O}\times(0,\varepsilon)$, let
  \begin{align} \label{E:LPW_UA}
    \begin{aligned}
      U(s) &:= u(\Psi_\varepsilon(s)), \quad \mathbf{A}_\varepsilon(s) := \det\nabla_s\Psi_\varepsilon(s)[(\nabla\Psi_\varepsilon^{-1})^{\mathrm{T}}\nabla\Psi_\varepsilon^{-1}](\Psi_\varepsilon(s)), \\
      F(s) &:= f(\Psi_\varepsilon(s))\det\nabla_s\Psi_\varepsilon(s), \quad \mathbf{V}(s) := \det\nabla_s\Psi_\varepsilon(s)[(\nabla\Psi_\varepsilon^{-1})^{\mathrm{T}}\mathbf{v}](\Psi_\varepsilon(s)),
    \end{aligned}
  \end{align}
  where $[\varphi_1\varphi_2](x)=\varphi_1(x)\varphi_2(x)$ for functions $\varphi_1$ and $\varphi_2$ on $\Psi_\varepsilon(\mathcal{V}_\varepsilon)$.
  Then, we have
  \begin{align} \label{E:LPW_Ve}
    \int_{\mathcal{V}_\varepsilon}(\mathbf{A}_\varepsilon\nabla_sU)\cdot\nabla_s\Phi\,ds = \int_{\mathcal{V}_\varepsilon}(F\Phi+\mathbf{V}\cdot\nabla_s\Phi)\,ds \quad\text{for all}\quad \Phi\in H^1(\mathcal{V}_\varepsilon),
  \end{align}
  where $\nabla_s$ is the gradient in $s$.
  Moreover,
  \begin{alignat}{2}
    &(\mathbf{A}_\varepsilon(s)\mathbf{a})\cdot\mathbf{a} \geq c|\mathbf{a}|^2, &\quad &s\in\mathcal{K}\times[0,\varepsilon], \quad \mathbf{a}\in\mathbb{R}^3, \label{E:Ae_Elp} \\
    &|\partial_s^\alpha\mathbf{A}_\varepsilon(s)| \leq c, &\quad &s\in\mathcal{K}\times[0,\varepsilon], \quad |\alpha| = 0,1,2. \label{E:dAe_Bd}
  \end{alignat}
  Also, for $k=0,1$, we have
  \begin{align} \label{E:LPW_Sou}
    \|F\|_{H^k(\mathcal{V}_\varepsilon)} \leq c\|f\|_{H^k(\Omega_\varepsilon)}, \quad \|\mathbf{V}\|_{H^{1+k}(\mathcal{V}_\varepsilon)} \leq c\|\mathbf{v}\|_{H^{1+k}(\Omega_\varepsilon)}.
  \end{align}
\end{lemma}

\begin{proof}
  When $x=\Psi_\varepsilon(s)\in\Psi_\varepsilon(\mathcal{V}_\varepsilon)$, we differentiate $u(x)=U(\Psi_\varepsilon^{-1}(x))$ to get
  \begin{align*}
    \nabla u(x) = \nabla\Psi_\varepsilon^{-1}(x)\nabla_sU(\Psi_\varepsilon^{-1}(x)), \quad \nabla u(\Psi_\varepsilon(s)) = \nabla\Psi_\varepsilon^{-1}(\Psi_\varepsilon(s))\nabla_sU(s).
  \end{align*}
  Thus, we have \eqref{E:LPW_Ve} by substituting $\varphi=\Phi\circ\Psi_\varepsilon$ for \eqref{E:LPW_Ori}, which is extended by zero outside $\Psi_\varepsilon(\mathcal{V}_\varepsilon)$ (this is allowed since $u$, $f$, and $\mathbf{v}$ are supported in $\Psi_\varepsilon(\mathcal{K}\times[0,\varepsilon])$), and by making the change of variables $x=\Psi_\varepsilon(s)$.

  Let us show \eqref{E:Ae_Elp}--\eqref{E:LPW_Sou}.
  Let $s\in\mathcal{K}\times[0,\varepsilon]$ and $\mathbf{a}\in\mathbb{R}^3$.
  Then,
  \begin{align*}
    (\mathbf{A}_\varepsilon(s)\mathbf{a})\cdot\mathbf{a} = |\mathbf{b}(s)|^2\det\nabla_s\Psi_\varepsilon(s), \quad \mathbf{b}(s) := [\nabla\Psi_\varepsilon^{-1}(\Psi_\varepsilon(s))]\mathbf{a}.
  \end{align*}
  Since $\Psi_\varepsilon^{-1}(\Psi_\varepsilon(s))=s$, we have $\nabla\Psi_\varepsilon^{-1}(\Psi_\varepsilon(s))\nabla_s\Psi_\varepsilon(s)=\mathbf{I}_3$.
  Thus,
  \begin{align*}
    \nabla\Psi_\varepsilon^{-1}(\Psi_\varepsilon(s)) = [\nabla_s\Psi_\varepsilon(s)]^{-1}, \quad \mathbf{a} = [\nabla_s\Psi_\varepsilon(s)]\mathbf{b}(s).
  \end{align*}
  Moreover, since \eqref{E:LPC_Bd} and \eqref{E:DeLC_Bd} holds for $s\in\mathcal{K}\times[0,\varepsilon]$, we find that
  \begin{align*}
    |\mathbf{a}|^2 = |[\nabla_s\Psi_\varepsilon(s)]\mathbf{b}(s)|^2 \leq c|\mathbf{b}(s)|^2\det\nabla_s\Psi_\varepsilon(s) = c(\mathbf{A}_\varepsilon(s)\mathbf{a})\cdot\mathbf{a}.
  \end{align*}
  Hence, \eqref{E:Ae_Elp} is valid.
  Also, we have \eqref{E:dAe_Bd} and \eqref{E:LPW_Sou} by \eqref{E:LPC_Bd}, \eqref{E:DeLC_Bd}, and \eqref{E:DeLC_ds}, since $f$ and $\mathbf{v}$ are supported in $\Psi_\varepsilon(\mathcal{K}\times[0,\varepsilon])$.
\end{proof}

\begin{lemma} \label{L:DQP_L2}
  Let $\{\mathbf{e}_1,\mathbf{e}_2,\mathbf{e}_3\}$ be the canonical basis of $\mathbb{R}^3$.
  For $i=1,2$, a nonzero $h\in\mathbb{R}$, and a function $\Phi$ on $\mathcal{V}_\varepsilon$ extended to $\mathbb{R}^3\setminus\mathcal{V}_\varepsilon$ by zero, we set
  \begin{align} \label{E:Def_Quo}
    \Phi_i^h(s) := \Phi(s+h\mathbf{e}_i), \quad D_i^h\Phi(s) := \frac{\Phi_i^h(s)-\Phi(s)}{h}, \quad s\in\mathcal{V}_\varepsilon.
  \end{align}
  Let $\mathcal{O}_1$ and $\mathcal{O}_2$ be open sets in $\mathbb{R}^2$ such that $\overline{\mathcal{O}_1}\subset\mathcal{O}_2$ and $\overline{\mathcal{O}_2}\subset\mathcal{O}$.
  Also, let $\Phi\in H^1(\mathcal{V}_\varepsilon)$ be supported in $\overline{\mathcal{O}_1}\times[0,\varepsilon]$.
  Then,
  \begin{align} \label{E:DQP_L2}
    \|D_i^h\Phi\|_{L^2(\mathcal{V}_\varepsilon)} \leq \|\partial_{s_i}\Phi\|_{L^2(\mathcal{V}_\varepsilon)}, \quad i=1,2
  \end{align}
  for all $h\in\mathbb{R}$ satisfying
  \begin{align} \label{E:DQP_h}
    0 < |h| < \frac{1}{2}\min\{\mathrm{dist}(\overline{\mathcal{O}_1},\mathbb{R}^2\setminus\mathcal{O}_2),\mathrm{dist}(\overline{\mathcal{O}_2},\mathbb{R}^2\setminus\mathcal{O})\},
  \end{align}
  where $\mathrm{dist}(\mathcal{A},\mathcal{B}):=\inf\{|a'-b'| \mid a'\in\mathcal{A}, \, b'\in\mathcal{B}\}$ for $\mathcal{A},\mathcal{B}\subset\mathbb{R}^2$.
\end{lemma}

Note that $i=3$ is excluded here.
Also, since $\mathcal{O}_1$, $\mathcal{O}_2$, and $\mathcal{O}$ are bounded and open in $\mathbb{R}^2$, and since $\overline{\mathcal{O}_1}\subset\mathcal{O}_2$ and $\overline{\mathcal{O}_2}\subset\mathcal{O}$, the right-hand side of \eqref{E:DQP_h} is positive.

\begin{proof}
  Let $i=1,2$ and $h$ satisfy \eqref{E:DQP_h}.
  Then, $D_i^h\Phi$ is supported in $\overline{\mathcal{O}_2}\times[0,\varepsilon]$, since $\Phi$ is supported in $\overline{\mathcal{O}_1}\times[0,\varepsilon]$.
  Moreover, when $s\in\mathcal{O}_2\times(0,\varepsilon)$, we have
  \begin{align*}
    D_i^h\Phi(s) = \frac{1}{h}\int_0^1\frac{\partial}{\partial\tau}\Bigl(\Phi(s+\tau h\mathbf{e}_i)\Bigr)\,d\tau = \int_0^1\partial_{s_i}\Phi(s+\tau h\mathbf{e}_i)\,d\tau
  \end{align*}
  by $s+\tau h\mathbf{e}_i\in\mathcal{V}_\varepsilon$ for $\tau\in[0,1]$.
  We take the square of both sides, apply H\"{o}lder's inequality to the right-hand side, and integrate with respect to $s$.
  Then,
  \begin{align*}
    \|D_i^h\Phi\|_{L^2(\mathcal{O}_2\times(0,\varepsilon))}^2 \leq \int_{\mathcal{O}_2\times(0,\varepsilon)}\left(\int_0^1|\partial_{s_i}\Phi(s+\tau h\mathbf{e}_i)|^2\,d\tau\right)\,ds.
  \end{align*}
  Hence, noting that $D_i^h\Phi$ is supported in $\overline{\mathcal{O}_2}\times[0,\varepsilon]$ and
  \begin{align*}
    \bigl(\mathcal{O}_2\times(0,\varepsilon)\bigr)+\tau h\mathbf{e}_i := \{s+\tau h\mathbf{e}_i \mid s\in\mathcal{O}_2\times(0,\varepsilon)\} \subset \mathcal{V}_\varepsilon
  \end{align*}
  for $\tau\in[0,1]$ by \eqref{E:DQP_h}, we apply Fubini's theorem and make translation in the right-hand side of the above inequality to get \eqref{E:DQP_L2}.
\end{proof}

\subsection{Localization} \label{SS:UER_Cu}
Let us start the proof of Lemma \ref{L:UER_CTD}.
For $k=0,1$, let $u\in H^{2+k}(\Omega_\varepsilon)$ satisfy $\partial_{\nu_\varepsilon}u=0$ on $\Omega_\varepsilon$.
Then, by integration by parts and H\"{o}lder's inequality,
\begin{align*}
  \|\nabla u\|_{L^2(\Omega_\varepsilon)}^2 = (-\Delta u,u)_{L^2(\Omega_\varepsilon)} \leq \|\Delta u\|_{L^2(\Omega_\varepsilon)}\|u\|_{L^2(\Omega_\varepsilon)}.
\end{align*}
Taking the square root, and applying Young's inequality, we find that
\begin{align} \label{E:IpGr_CTD}
  \|\nabla u\|_{L^2(\Omega_\varepsilon)} \leq \frac{1}{2}\Bigl(\|\Delta u\|_{L^2(\Omega_\varepsilon)}+\|u\|_{L^2(\Omega_\varepsilon)}\Bigr).
\end{align}
Let $f:=-\Delta u\in H^k(\Omega_\varepsilon)$.
Then, $u$ satisfies
\begin{align} \label{E:PoW_CTD}
  \int_{\Omega_\varepsilon}\nabla u\cdot\nabla\varphi\,dx = \int_{\Omega_\varepsilon}f\varphi\,dx \quad\text{for all}\quad \varphi\in H^1(\Omega_\varepsilon)
\end{align}
by integration by parts and $\partial_{\nu_\varepsilon}u=0$ on $\Omega_\varepsilon$.

Let $\mathcal{O}=(0,1)^2$.
By Lemma \ref{L:Op_Cov}, there exist $C^5$ local parametrizations $\psi_\ell\colon\mathcal{O}\to\Gamma$ such that $\{\psi_\ell(\mathcal{O})\}_{\ell=1}^L$ is a finite open cover of $\Gamma$.
Let $\{\eta_\ell\}_{\ell=1}^L$ be a $C^5$ partition of unity on $\Gamma$ subordinate to $\{\psi_\ell(\mathcal{O})\}_{\ell=1}^L$.
We take open sets $\mathcal{O}_{\ell,1}$, $\mathcal{O}_{\ell,2}$, and $\mathcal{O}_{\ell,3}$ in $\mathbb{R}^2$ such that
\begin{align*}
  \mathrm{supp}\,\eta_\ell \subset \psi_\ell(\overline{\mathcal{O}_{\ell,1}}), \quad \overline{\mathcal{O}_{\ell,1}} \subset \mathcal{O}_{\ell,2}, \quad \overline{\mathcal{O}_{\ell,2}} \subset \mathcal{O}_{\ell,3}, \quad \overline{\mathcal{O}_{\ell,3}} \subset \mathcal{O}.
\end{align*}
Let $\bar{\eta}_\ell=\eta_\ell$ be the constant extension of $\eta_\ell$ in the normal direction of $\Gamma$.
We set $u_\ell:=\bar{\eta}_\ell u$ on $\Omega_\varepsilon$.
Then, $u_\ell\in H^{2+k}(\Omega_\varepsilon)$ by $\eta\in C^5(\Gamma)$ and \eqref{E:CEGr_Bd}.
Also, $u=\sum_{\ell=1}^Lu_\ell$ on $\Omega_\varepsilon$.

Let $\ell=1,\dots,L$ and $\varphi\in H^1(\Omega_\varepsilon)$.
Since
\begin{align*}
  \nabla(\bar{\eta}_\ell u)\cdot\nabla\varphi = \nabla u\cdot\nabla(\bar{\eta}_\ell\varphi)+u\nabla\bar{\eta}_\ell\cdot\nabla\varphi-(\nabla u\cdot\nabla\bar{\eta}_\ell)\varphi \quad\text{in}\quad \Omega_\varepsilon
\end{align*}
and \eqref{E:PoW_CTD} holds with $\varphi$ replaced by $\bar{\eta}_\ell\varphi$, we see that $u_\ell=\bar{\eta}_\ell u$ satisfies
\begin{align*}
  \int_{\Omega_\varepsilon}\nabla u_\ell\cdot\nabla\varphi\,dx = \int_{\Omega_\varepsilon}(f_\ell\varphi+\mathbf{v}_\ell\cdot\nabla\varphi)\,dx \quad\text{for all}\quad \varphi\in H^1(\Omega_\varepsilon),
\end{align*}
where $f_\ell:=\bar{\eta}_\ell f-\nabla u\cdot\nabla\bar{\eta}_\ell$ and $\mathbf{v}_\ell:=u\nabla\bar{\eta}_\ell$ on $\Omega_\varepsilon$.
Here,
\begin{align*}
  \mathrm{supp}\,\zeta_\ell \subset \{y+r\bm{\nu}(y) \mid y\in\psi_\ell(\overline{\mathcal{O}_{\ell,1}}), \, \varepsilon g_0(y)\leq r\leq\varepsilon g_1(y)\}
\end{align*}
for $\zeta_\ell=u_\ell,f_\ell,\mathbf{v}_\ell$ by $\mathrm{supp}\,\eta_\ell\subset\psi_\ell(\overline{\mathcal{O}_{\ell,1}})$.
Also, for $k=0,1$,
\begin{align} \label{E:UER_fel}
  \|f_\ell\|_{H^k(\Omega_\varepsilon)} \leq c\Bigl(\|f\|_{H^k(\Omega_\varepsilon)}+\|u\|_{H^{1+k}(\Omega_\varepsilon)}\Bigr), \quad \|\mathbf{v}_\ell\|_{H^{1+k}(\Omega_\varepsilon)} \leq c\|u\|_{H^{1+k}(\Omega_\varepsilon)}
\end{align}
by \eqref{E:CEGr_Bd} and $\eta\in C^5(\Gamma)$.
Now, suppose first that
\begin{align} \label{E:uel_H2}
  \|u_\ell\|_{H^2(\Omega_\varepsilon)} \leq c\Bigl(\|f_\ell\|_{L^2(\Omega_\varepsilon)}+\|\mathbf{v}_\ell\|_{H^1(\Omega_\varepsilon)}\Bigr)
\end{align}
for all $\ell=1,\dots,L$.
Then, by \eqref{E:IpGr_CTD}, \eqref{E:UER_fel}, and $f=-\Delta u$, we have
\begin{align*}
  \|u_\ell\|_{H^2(\Omega_\varepsilon)} \leq c\Bigl(\|f\|_{L^2(\Omega_\varepsilon)}+\|u\|_{H^1(\Omega_\varepsilon)}\Bigr) \leq c\Bigl(\|\Delta u\|_{L^2(\Omega_\varepsilon)}+\|u\|_{L^2(\Omega_\varepsilon)}\Bigr),
\end{align*}
which gives \eqref{E:UER_CTD} with $k=0$ for $u=\sum_{\ell=1}^Lu_\ell$.
Next, suppose further that
\begin{align} \label{E:uel_H3}
  \|u_\ell\|_{H^3(\Omega_\varepsilon)} \leq c\Bigl(\|f_\ell\|_{H^1(\Omega_\varepsilon)}+\|\mathbf{v}_\ell\|_{H^2(\Omega_\varepsilon)}\Bigr)
\end{align}
for all $\ell=1,\dots,L$.
Then, we apply \eqref{E:UER_fel} and use \eqref{E:UER_CTD} with $k=0$ to get
\begin{align*}
  \|u_\ell\|_{H^3(\Omega_\varepsilon)} \leq c\Bigl(\|f\|_{H^1(\Omega_\varepsilon)}+\|u\|_{H^2(\Omega_\varepsilon)}\Bigr) \leq c\Bigl(\|\Delta u\|_{H^1(\Omega_\varepsilon)}+\|u\|_{L^2(\Omega_\varepsilon)}\Bigr).
\end{align*}
Thus, \eqref{E:UER_CTD} with $k=1$ holds for $u=\sum_{\ell=1}^Lu_\ell$.
Let us show first \eqref{E:uel_H2} and then \eqref{E:uel_H3}.
The proof relies on the method of difference quotient after the change of variables.

\subsection{Change of variables} \label{SS:UER_CoV}
Now, we fix and suppress the index $\ell$.
Let $\mathcal{O}=(0,1)^2$, and let $\mathcal{O}_1$, $\mathcal{O}_2$, and $\mathcal{O}_3$ be open sets in $\mathbb{R}^2$ such that
\begin{align} \label{E:O123}
  \overline{\mathcal{O}_1} \subset \mathcal{O}_2, \quad \overline{\mathcal{O}_2} \subset \mathcal{O}_3, \quad \overline{\mathcal{O}_3} \subset \mathcal{O}.
\end{align}
We define the mapping $\Psi_\varepsilon$ by \eqref{E:Def_LPC}, and suppose that
\begin{align} \label{E:su_zeta}
  \begin{aligned}
    &u\in H^{2+k}(\Omega_\varepsilon), \quad f\in H^k(\Omega_\varepsilon), \quad \mathbf{v} \in H^{1+k}(\Omega_\varepsilon), \quad k=0,1, \\
    &\mathrm{supp}\,\zeta \subset \Psi_\varepsilon(\overline{\mathcal{O}_1}\times[0,\varepsilon]), \quad \zeta = u,f,\mathbf{v},
  \end{aligned}
\end{align}
and \eqref{E:LPW_Ori} holds.
Let $U$, $\mathbf{A}_\varepsilon$, $F$, and $\mathbf{V}$ be given by \eqref{E:LPW_UA} on $\mathcal{V}_\varepsilon=\mathcal{O}\times(0,\varepsilon)$.
Then, the equation \eqref{E:LPW_Ve} and the estimates \eqref{E:Hk_LPC} and \eqref{E:LPW_Sou} are valid by Lemmas \ref{L:Hk_LPC} and \ref{L:LPC_We}.
Thus, to get \eqref{E:uel_H2} and \eqref{E:uel_H3} (recall that we suppress $\ell$), it is sufficient to show
\begin{align} \label{E:UER_Goal}
  \|U\|_{H^{2+k}(\mathcal{V}_\varepsilon)} \leq c\Bigl(\|F\|_{H^k(\mathcal{V}_\varepsilon)}+\|\mathbf{V}\|_{H^{1+k}(\mathcal{V}_\varepsilon)}\Bigr), \quad k=0,1.
\end{align}
Before showing \eqref{E:UER_Goal}, let us first give an $H^1$-estimate for $U$.
We see by \eqref{E:su_zeta} that
\begin{align} \label{E:supp_Z}
  \mathrm{supp}\,Z \subset \overline{\mathcal{O}_1}\times[0,\varepsilon], \quad Z = U,F,\mathbf{V}.
\end{align}
Hence, $U(\cdot,s_3)$ is supported in $\overline{\mathcal{O}_1}$ for a.a. $s_3\in(0,\varepsilon)$.
Since $\overline{\mathcal{O}_1}\subset\mathcal{O}$, we have
\begin{align*}
  \|U(\cdot,s_3)\|_{L^2(\mathcal{O})}^2 \leq \|\partial_{s_1}U(\cdot,s_3)\|_{L^2(\mathcal{O})}^2
\end{align*}
by Poincar\'{e}'s inequality on $\mathcal{O}=(0,1)^2$ (note that the constant $c=1$ is independent of $\varepsilon$).
Integrating both sides with respect to $s_3$, and taking the square root, we find that
\begin{align} \label{E:Poin_Ve}
  \|U\|_{L^2(\mathcal{V}_\varepsilon)} \leq \|\partial_{s_1}U\|_{L^2(\mathcal{V}_\varepsilon)} \leq \|\nabla_sU\|_{L^2(\mathcal{V}_\varepsilon)}.
\end{align}
Since \eqref{E:supp_Z} holds, we can use \eqref{E:Ae_Elp} and then \eqref{E:LPW_Ve} to find that
\begin{align*}
  \|\nabla_sU\|_{L^2(\mathcal{V}_\varepsilon)}^2 \leq c\int_{\mathcal{V}_\varepsilon}(\mathbf{A}_\varepsilon\nabla_sU)\cdot\nabla_sU\,ds = c\int_{\mathcal{V}_\varepsilon}(FU+\mathbf{V}\cdot\nabla_sU)\,ds.
\end{align*}
Using H\"{o}lder's and Young's inequalities and \eqref{E:Poin_Ve} to the right-hand side, we get
\begin{align*}
  \|\nabla_sU\|_{L^2(\mathcal{V}_\varepsilon)}^2 \leq \frac{1}{2}\|\nabla_sU\|_{L^2(\mathcal{V}_\varepsilon)}^2+c\Bigl(\|F\|_{L^2(\mathcal{V}_\varepsilon)}^2+\|\mathbf{V}\|_{L^2(\mathcal{V}_\varepsilon)}^2\Bigr).
\end{align*}
By \eqref{E:Poin_Ve} and the above inequality, we obtain
\begin{align} \label{E:UER_UH1}
  \|U\|_{H^1(\mathcal{V}_\varepsilon)} \leq 2\|\nabla_sU\|_{L^2(\mathcal{V}_\varepsilon)} \leq c\Bigl(\|F\|_{L^2(\mathcal{V}_\varepsilon)}+\|\mathbf{V}\|_{L^2(\mathcal{V}_\varepsilon)}\Bigr).
\end{align}
Now, let us show \eqref{E:UER_Goal}.
We deal with the cases $k=0$ and $k=1$ separately.

\subsection{Uniform \texorpdfstring{$H^2$}{H\textasciicircum2}-estimate} \label{SS:UER_H2}
First, let $k=0$.
We use the notation \eqref{E:Def_Quo} with
\begin{align} \label{E:DiQu_h}
  0 < |h| < \frac{1}{2}\min_{j=1,2,3}\delta_j, \quad \delta_j := \mathrm{dist}(\overline{\mathcal{O}_j},\mathbb{R}^2\setminus\mathcal{O}_{j+1}),
\end{align}
where $\mathcal{O}_4:=\mathcal{O}$.
Then, by Lemma \ref{L:DQP_L2}, \eqref{E:O123}, and \eqref{E:supp_Z}, we have
\begin{align} \label{E:DiQu_ZL2}
  \|D_i^hZ\|_{L^2(\mathcal{V}_\varepsilon)} \leq \|\partial_{s_i}Z\|_{L^2(\mathcal{V}_\varepsilon)}, \quad i=1,2, \quad Z = U,F,\mathbf{V}.
\end{align}
Moreover, $D_i^hU$ is supported in $\overline{\mathcal{O}_2}\times[0,\varepsilon]$ by \eqref{E:O123} and \eqref{E:supp_Z}.
Hence,
\begin{align} \label{E:DiQu_DhU}
  \|D_i^{-h}D_i^hU\|_{L^2(\mathcal{V}_\varepsilon)} \leq \|\partial_{s_i}D_i^hU\|_{L^2(\mathcal{V}_\varepsilon)} \leq \|\nabla_sD_i^hU\|_{L^2(\mathcal{V}_\varepsilon)}
\end{align}
by Lemma \ref{L:DQP_L2} with $\mathcal{O}_j$ replaced by $\mathcal{O}_{j+1}$ for $j=1,2$.

Let $i=1,2$ and $\Phi\in H^1(\mathcal{V}_\varepsilon)$ be supported in $\overline{\mathcal{O}_2}\times[0,\varepsilon]$.
Then, since $\Phi_i^{-h}$ is supported in $\overline{\mathcal{O}_3}\times[0,\varepsilon]$ by \eqref{E:DiQu_h}, we can replace $\Phi$ in \eqref{E:LPW_Ve} by $\Phi_i^{-h}$.
In the resulting equation, we make translation except for the integral of $F\Phi_i^{-h}$ to get
\begin{align*}
  \int_{\mathcal{V}_\varepsilon}([\mathbf{A}_\varepsilon]_i^h\nabla_sU_i^h)\cdot\nabla\Phi\,ds = \int_{\mathcal{V}_\varepsilon}(F\Phi_i^{-h}+\mathbf{V}_i^h\cdot\nabla_s\Phi)\,ds.
\end{align*}
We take the difference of this equation and \eqref{E:LPW_Ve}, and then divide by $h$.
Then,
\begin{align*}
  \int_{\mathcal{V}_\varepsilon}\{[\mathbf{A}_\varepsilon]_i^h\nabla_sD_i^hU+(D_i^h\mathbf{A}_\varepsilon)\nabla_sU\}\cdot\nabla_s\Phi\,ds = \int_{\mathcal{V}_\varepsilon}\{-FD_i^{-h}\Phi+(D_i^h\mathbf{V})\cdot\nabla_s\Phi\}\,ds.
\end{align*}
Since $D_i^hU$ is supported in $\overline{\mathcal{O}_2}\times[0,\varepsilon]$, we can take $\Phi=D_i^hU$ to get
\begin{multline} \label{E:DihU_Bi}
  \int_{\mathcal{V}_\varepsilon}([\mathbf{A}_\varepsilon]_i^h\nabla_sD_i^hU)\cdot\nabla_sD_i^hU\,ds \\
  = \int_{\mathcal{V}_\varepsilon}\{-FD_i^{-h}D_i^hU+[D_i^h\mathbf{V}-(D_i^h\mathbf{A}_\varepsilon)\nabla_sU]\cdot\nabla_sD_i^hU\}\,ds.
\end{multline}
Now, we recall that $[\mathbf{A}_\varepsilon]_i^h(s)=\mathbf{A}_\varepsilon(s+h\mathbf{e}_i)$ and
\begin{align*}
  D_i^h\mathbf{A}_\varepsilon(s) = \frac{1}{h}\int_0^1\frac{\partial}{\partial\tau}\Bigl(\mathbf{A}_\varepsilon(s+\tau h\mathbf{e}_i)\Bigr)\,d\tau = \int_0^1\partial_{s_i}\mathbf{A}_\varepsilon(s+\tau h\mathbf{e}_i)\,d\tau.
\end{align*}
Moreover, $s+\tau h\mathbf{e}_i\in\overline{\mathcal{O}_3}\times[0,\varepsilon]$ for $s\in\overline{\mathcal{O}_2}\times[0,\varepsilon]$ and $\tau\in[0,1]$ by \eqref{E:DiQu_h}.
Hence, we can use \eqref{E:Ae_Elp} and \eqref{E:dAe_Bd} with $\mathcal{K}=\overline{\mathcal{O}_3}$ to find that
\begin{align*}
  ([\mathbf{A}_\varepsilon]_i^h(s)\mathbf{a})\cdot\mathbf{a} \geq c|\mathbf{a}|^2, \quad |D_i^h\mathbf{A}_\varepsilon(s)| \leq c, \quad s\in\overline{\mathcal{O}_2}\times[0,\varepsilon], \quad \mathbf{a}\in\mathbb{R}^3.
\end{align*}
Since $D_i^hU$ is supported in $\overline{\mathcal{O}_2}\times[0,\varepsilon]$, we can apply these inequalities to \eqref{E:DihU_Bi}.
Then, we further use and H\"{o}lder's and Young's inequality, \eqref{E:DiQu_ZL2}, and \eqref{E:DiQu_DhU} to get
\begin{align*}
  \|\nabla_sD_i^hU\|_{L^2(\mathcal{V}_\varepsilon)}^2 &\leq c\int_{\mathcal{V}_\varepsilon}\{[\mathbf{A}_\varepsilon]_i^h\nabla_sD_i^hU\}\cdot\nabla_sD_i^hU\,ds \\
  &\leq \frac{1}{2}\|\nabla_sD_i^hU\|_{L^2(\mathcal{V}_\varepsilon)}^2+c\Bigl(\|F\|_{L^2(\mathcal{V}_\varepsilon)}^2+\|\partial_{s_i}\mathbf{V}\|_{L^2(\mathcal{V}_\varepsilon)}^2+\|\nabla_sU\|_{L^2(\mathcal{V}_\varepsilon)}^2\Bigr).
\end{align*}
By this inequality and \eqref{E:UER_UH1}, we find that (after taking the square root)
\begin{align*}
  \|\nabla_sD_i^hU\|_{L^2(\mathcal{V}_\varepsilon)} \leq c\Bigl(\|F\|_{L^2(\mathcal{V}_\varepsilon)}+\|\mathbf{V}\|_{H^1(\mathcal{V}_\varepsilon)}\Bigr),
\end{align*}
where $c>0$ is a constant independent of $\varepsilon$ and $h$.
Hence, letting $h\to0$, we get
\begin{align} \label{E:DsDiU_L2}
  \|\nabla_s\partial_{s_i}U\|_{L^2(\mathcal{V}_\varepsilon)} \leq c\Bigl(\|F\|_{L^2(\mathcal{V}_\varepsilon)}+\|\mathbf{V}\|_{H^1(\mathcal{V}_\varepsilon)}\Bigr), \quad i=1,2,
\end{align}
see also \cite[Section 5.8, Theorem 3 (ii)]{Eva10}.
It remains to estimate $\partial_{s_3}^2U$.
To this end, we test any $\Phi\in C_c^\infty(\mathcal{V}_\varepsilon)$ in \eqref{E:LPW_Ve} and carry out integration by parts.
Then, we find that
\begin{align} \label{E:div_AU}
  -\mathrm{div}_s(\mathbf{A}_\varepsilon\nabla_sU) = F-\mathrm{div}_s\mathbf{V} \quad\text{a.e. in}\quad \mathcal{V}_\varepsilon,
\end{align}
where $\mathrm{div}_s$ is the divergence in $s$.
We rewrite this equation as
\begin{align} \label{E:d33U_Ve}
  \begin{aligned}
    &-A_{\varepsilon,33}\partial_{s_3}^2U = F-\mathrm{div}_s\mathbf{V}+(\partial_{s_3}A_{\varepsilon,33})\partial_{s_3}U+G \quad\text{a.e. in}\quad \mathcal{V}_\varepsilon, \\
    &G := \sum_{(i,j)\in\Lambda}\partial_{s_i}(A_{\varepsilon,ij}\partial_{s_j}U), \quad \Lambda := \{(i,j)\in\{1,2,3\}^2 \mid (i,j)\neq(3,3)\},
  \end{aligned}
\end{align}
where $A_{\varepsilon,ij}$ is the $(i,j)$-entry of $\mathbf{A}_\varepsilon$.
Moreover, we see by \eqref{E:Ae_Elp} that
\begin{align} \label{E:Ae33_Bd}
  A_{\varepsilon,33}(s) = (\mathbf{A}_\varepsilon(s)\mathbf{e}_3)\cdot\mathbf{e}_3 \geq c|\mathbf{e}_3|^2 = c, \quad s\in\overline{\mathcal{O}_1}\times(0,\varepsilon),
\end{align}
where $\mathbf{e}_3=(0,0,1)^{\mathrm{T}}$.
Applying this and \eqref{E:dAe_Bd} with $\mathcal{K}=\overline{\mathcal{O}_1}$ to \eqref{E:d33U_Ve}, we get
\begin{align*}
  |\partial_{s_3}^2U| \leq cA_{\varepsilon,33}|\partial_{s_3}^2U| \leq c(|F|+|\mathrm{div}_s\mathbf{V}|+|\nabla_sU|+|\nabla_s\partial_{s_1}U|+|\nabla_s\partial_{s_2}U|)
\end{align*}
a.e. in $\mathcal{O}_1\times(0,\varepsilon)$.
Noting that $U$ is supported in $\overline{\mathcal{O}_1}\times[0,\varepsilon]$, we take the $L^2$-norm of the above inequality and apply \eqref{E:UER_UH1} and \eqref{E:DsDiU_L2} to find that
\begin{align} \label{E:d33U_L2}
  \|\partial_{s_3}^2U\|_{L^2(\mathcal{V}_\varepsilon)} \leq c\Bigl(\|F\|_{L^2(\mathcal{V}_\varepsilon)}+\|\mathbf{V}\|_{H^1(\mathcal{V}_\varepsilon)}\Bigr).
\end{align}
Hence, we obtain \eqref{E:UER_Goal} with $k=0$ by \eqref{E:UER_UH1}, \eqref{E:DsDiU_L2}, and \eqref{E:d33U_L2}.

\subsection{Uniform \texorpdfstring{$H^3$}{H\textasciicircum3}-estimate} \label{SS:UER_H3}
Now, let us prove \eqref{E:UER_Goal} for $k=1$.

Let $i=1,2$ and $\mathbf{X}_i:=\partial_{s_i}\mathbf{V}-(\partial_{s_i}\mathbf{A}_\varepsilon)\nabla_sU$ on $\mathcal{V}_\varepsilon$.
Then, since \eqref{E:supp_Z} holds,
\begin{align} \label{E:Xi_H1}
  \|\mathbf{X}_i\|_{H^1(\mathcal{V}_\varepsilon)} \leq c\Bigl(\|\mathbf{V}\|_{H^2(\mathcal{V}_\varepsilon)}+\|U\|_{H^2(\mathcal{V}_\varepsilon)}\Bigr) \leq c\Bigl(\|F\|_{L^2(\mathcal{V}_\varepsilon)}+\|\mathbf{V}\|_{H^2(\mathcal{V}_\varepsilon)}\Bigr)
\end{align}
by \eqref{E:dAe_Bd} with $\mathcal{K}=\overline{\mathcal{O}_1}$ and \eqref{E:UER_Goal} with $k=0$.
Let us show
\begin{align} \label{E:dsiU_WF}
  \int_{\mathcal{V}_\varepsilon}(\mathbf{A}_\varepsilon\nabla_s\partial_{s_i}U)\cdot\nabla_s\Phi\,ds = \int_{\mathcal{V}_\varepsilon}\{(\partial_{s_i}F)\Phi+\mathbf{X}_i\cdot\nabla_s\Phi\}\,ds
\end{align}
for $i=1,2$ and all $\Phi\in H^1(\mathcal{V}_\varepsilon)$.
Since \eqref{E:O123} and \eqref{E:supp_Z} hold, we may assume that
\begin{align*}
  \mathrm{supp}\,\Phi \subset \overline{\mathcal{O}_2}\times[0,\varepsilon]
\end{align*}
by multiplying $\Phi$ by a smooth cut-off function depending only on $s'=(s_1,s_2)$.
We extend $\Phi(\cdot,s_3)$ to $\mathbb{R}^2\setminus\mathcal{O}$ by zero for each $s_3\in(0,\varepsilon)$ and mollify $\Phi$ with respect to the variable $s'$.
Then, we have functions $\Phi_m\in H^1(\mathcal{V}_\varepsilon)$ with $m\in\mathbb{N}$ such that
\begin{align*}
  \partial_{s_1}\Phi_m,\partial_{s_2}\Phi_m\in H^1(\mathcal{V}_\varepsilon), \quad \mathrm{supp}\,\Phi_m \subset \overline{\mathcal{O}_3}\times[0,\varepsilon], \quad \lim_{m\to\infty}\|\Phi_m-\Phi\|_{H^1(\mathcal{V}_\varepsilon)} = 0.
\end{align*}
Note that $\partial_{s_3}\Phi_m$ also converges to $\partial_{s_3}\Phi$ in $L^2(\mathcal{V}_\varepsilon)$, since we do noting on $s_3$ in the extension and mollification with respect to $s'$.
Let $i=1,2$.
Since $\overline{\mathcal{O}_3}\subset\mathcal{O}$, we can replace $\Phi$ in \eqref{E:LPW_Ve} by $\partial_{s_i}\Phi_m$ and carry out integration by parts with respect to $s_i$.
Then, we find that \eqref{E:dsiU_WF} holds with $\Phi$ replaced by $\Phi_m$, and thus we have \eqref{E:dsiU_WF} for $\Phi$ by letting $m\to\infty$.

Now, using \eqref{E:dsiU_WF} instead of \eqref{E:LPW_Ve}, we repeat the arguments of Section \ref{SS:UER_H2} with $U$, $F$, and $\mathbf{V}$ replaced by $\partial_{s_i}U$, $\partial_{s_i}F$, and $\mathbf{X}_i$, respectively.
Then, we get
\begin{align*}
  \|\partial_{s_i}U\|_{H^2(\mathcal{V}_\varepsilon)} \leq c\Bigl(\|\partial_{s_i}F\|_{L^2(\mathcal{V}_\varepsilon)}+\|\mathbf{X}_i\|_{H^1(\mathcal{V}_\varepsilon)}\Bigr), \quad i=1,2.
\end{align*}
Thus, we further apply \eqref{E:Xi_H1} to find that
\begin{align} \label{E:dsiU_H2}
  \|\partial_{s_i}U\|_{H^2(\mathcal{V}_\varepsilon)} \leq c\Bigl(\|F\|_{H^1(\mathcal{V}_\varepsilon)}+\|\mathbf{V}\|_{H^2(\mathcal{V}_\varepsilon)}\Bigr), \quad i = 1,2.
\end{align}
The left-hand side of \eqref{E:dsiU_H2} includes all of the third order derivatives of $U$ except for $\partial_{s_3}^3U$.
To estimate this derivative, we take $\partial_{s_3}$ of \eqref{E:d33U_Ve}.
Then, we find that
\begin{align*}
  -A_{\varepsilon,33}\partial_{s_3}^3U = \partial_{s_3}F-\partial_{s_3}(\mathrm{div}_s\mathbf{V})+(\partial_{s_3}^2A_{\varepsilon,33})\partial_{s_3}U+2(\partial_{s_3}A_{\varepsilon,33})\partial_{s_3}^2U+\partial_{s_3}G
\end{align*}
a.e. in $\mathcal{V}_\varepsilon$, and it follows from \eqref{E:dAe_Bd} with $\mathcal{K}=\overline{\mathcal{O}_1}$ and \eqref{E:Ae33_Bd} that
\begin{align*}
  |\partial_{s_3}^3U| &\leq cA_{\varepsilon,33}|\partial_{s_3}^3U| \\
  &\leq c(|\partial_{s_3}F|+|\nabla_s^2\mathbf{V}|+|\nabla_sU|+|\nabla_s^2U|+|\nabla_s^2\partial_{s_1}U|+|\nabla_s^2\partial_{s_2}U|)
\end{align*}
a.e. in $\mathcal{O}_1\times(0,\varepsilon)$, where
\begin{align*}
  |\nabla_s^2Z| = (\textstyle\sum_{i,j=1}^3|\partial_{s_i}\partial_{s_j}Z|^2)^{1/2}, \quad Z = \mathbf{V},U,\partial_{s_1}U,\partial_{s_2}U.
\end{align*}
Noting that $U$ is supported in $\overline{\mathcal{O}_1}\times[0,\varepsilon]$, we take the $L^2$-norm of the above inequality and apply \eqref{E:UER_Goal} with $k=0$ and \eqref{E:dsiU_H2}.
Then, we obtain
\begin{align} \label{E:Ud3_third}
  \|\partial_{s_3}^3U\|_{L^2(\mathcal{V}_\varepsilon)} \leq c\Bigl(\|F\|_{H^1(\mathcal{V}_\varepsilon)}+\|\mathbf{V}\|_{H^2(\mathcal{V}_\varepsilon)}\Bigr).
\end{align}
Since we already verified \eqref{E:UER_Goal} for $k=0$, we conclude by \eqref{E:dsiU_H2} and \eqref{E:Ud3_third} that \eqref{E:UER_Goal} is valid for $k=1$.
This completes the proof of Lemma \ref{L:UER_CTD}.

\begin{remark} \label{R:UER_CTD}
  If $\Gamma$, $g_0$, and $g_1$ are more regular, then we can prove
  \begin{align*}
    \|\partial_{s_j}U\|_{H^3(\mathcal{V}_\varepsilon)} \leq c\Bigl(\|\partial_{s_j}F\|_{H^1(\mathcal{V}_\varepsilon)}+\|\mathbf{X}_j\|_{H^2(\mathcal{V}_\varepsilon)}\Bigr), \quad j = 1,2
  \end{align*}
  by repeating the arguments of Section \ref{SS:UER_H3} with $U$, $F$, and $\mathbf{V}$ replaced by $\partial_{s_j}U$, $\partial_{s_j}F$, and $\mathbf{X}_j$, respectively.
  Thus, estimating the last term as in \eqref{E:Xi_H1}, we get
  \begin{align*}
    \|\partial_{s_j}U\|_{H^3(\mathcal{V}_\varepsilon)} \leq c\Bigl(\|F\|_{H^2(\mathcal{V}_\varepsilon)}+\|\mathbf{V}\|_{H^3(\mathcal{V}_\varepsilon)}\Bigr), \quad j = 1,2.
  \end{align*}
  Also, taking $\partial_{s_3}^2$ of \eqref{E:d33U_Ve}, we can find as in \eqref{E:Ud3_third} that
  \begin{align*}
    \|\partial_{s_3}^4U\|_{L^2(\mathcal{V}_\varepsilon)} \leq c\Bigl(\|F\|_{H^2(\mathcal{V}_\varepsilon)}+\|\mathbf{V}\|_{H^3(\mathcal{V}_\varepsilon)}\Bigr).
  \end{align*}
  Thus, \eqref{E:UER_Goal} holds when $k=2$.
  Repeating this procedure, we can also get \eqref{E:UER_Goal} for $k\geq3$.
  Hence, the uniform elliptic regularity estimate \eqref{E:UER_CTD} is also valid for $k\geq2$, provided that $\Gamma$, $g_0$, and $g_1$ are sufficiently regular.
  We emphasize that this is highly nontrivial, since we confirmed that the constant $c$ in \eqref{E:UER_CTD} is independent of $\varepsilon$.
\end{remark}

\section{Proofs of auxiliary results} \label{S:Pf_Aux}
This section gives the proofs of some lemmas in Section \ref{S:Pre}.

\subsection{Proof of Lemma \ref{L:Ipl_CTD}} \label{SS:PA_28}
Let $\mathcal{O}=(0,1)^2$.
By Lemma \ref{L:Op_Cov}, we can take a finite number of $C^5$ local parametrizations $\psi_\ell\colon\mathcal{O}\to\Gamma$ such that $\{\psi_\ell(\mathcal{O})\}_{\ell=1}^L$ is an open cover of $\Gamma$.
Let $\{\eta_\ell\}_{\ell=1}^L$ be a $C^5$ partition of unity on $\Gamma$ subordinate to $\{\psi_\ell(\mathcal{O})\}_{\ell=1}^L$.
We may assume that $\eta_\ell$ is supported in $\psi_\ell(\mathcal{K}_\ell)$, where $\mathcal{K}_\ell$ is a compact subset of $\mathcal{O}$.

Let $u\in H^2(\Omega_\varepsilon)$ and $p\in[2,\infty]$.
We define $u_\ell:=\bar{\eta}_\ell u$ on $\Omega_\varepsilon$, where $\bar{\eta}_\ell$ is the constant extension of $\eta_\ell$ in the normal direction of $\Gamma$.
Then,
\begin{align*}
  \mathrm{supp}\,u_\ell \subset \{y+r\bm{\nu}(y) \mid y\in\psi_\ell(\mathcal{K}_\ell), \, \varepsilon g_0(y)\leq r\leq\varepsilon g_1(y)\}.
\end{align*}
Moreover, by $0\leq\eta_\ell\leq1$ and $\sum_{\ell=1}^L\eta_\ell=1$ on $\Gamma$, \eqref{E:CEGr_Bd}, and $\eta_\ell\in C^5(\Gamma)$,
\begin{align*}
  \|u\|_{L^p(\Omega_\varepsilon)} \leq \sum_{\ell=1}^L\|u_\ell\|_{L^p(\Omega_\varepsilon)}, \quad \|u_\ell\|_{\mathcal{X}(\Omega_\varepsilon)} \leq c\|u\|_{\mathcal{X}(\Omega_\varepsilon)}, \quad \mathcal{X} = L^2,H^2.
\end{align*}
Hence, to get \eqref{E:Ipl_CTD} for $u$, it is sufficient to prove the same inequality for each $u_\ell$.

Now, we fix and suppress the index $\ell$.
Let $\psi\colon\mathcal{O}\to\Gamma$ be a $C^5$ local parametrization.
Suppose that $u\in H^2(\Omega_\varepsilon)$ and there exists a compact subset $\mathcal{K}$ of $\mathcal{O}$ such that
\begin{align*}
  \mathrm{supp}\,u \subset \{y+r\bm{\nu}(y) \mid y\in\psi(\mathcal{K}), \, \varepsilon g_0(y)\leq r\leq\varepsilon g_1(y)\}.
\end{align*}
For $s=(s',s_3)\in\mathcal{O}\times[0,1]$, we define
\begin{align*}
  \zeta_\varepsilon(s) &:= \psi(s')+\varepsilon r_1(s)\bm{\nu}^\flat(s'), \quad r_1(s) := (1-s_3)g_0^\flat(s')+s_3g_1^\flat(s').
\end{align*}
Here, we write $\eta^\flat=\eta\circ\psi$ on $\mathcal{O}$ for a function $\eta$ on $\Gamma$.
Let
\begin{align*}
  \mathcal{V}_1 := \mathcal{O}\times(0,1) = (0,1)^3, \quad U := u\circ\zeta_\varepsilon \quad\text{on}\quad \mathcal{V}_1.
\end{align*}
Then, $U$ is supported in $\mathcal{K}\times[0,1]$.
Moreover, we see that
\begin{align} \label{Pf_28:Lp}
  c^{-1}\|U\|_{L^p(\mathcal{V}_1)} \leq \varepsilon^{-1/p}\|u\|_{L^p(\Omega_\varepsilon)} \leq c\|U\|_{L^p(\mathcal{V}_1)}, \quad \|U\|_{H^2(\mathcal{V}_1)} \leq c\varepsilon^{-1/2}\|u\|_{H^2(\Omega_\varepsilon)}.
\end{align}
Indeed, since $\bm{\nu}^\flat$, $g_0^\flat$, and $g_1^\flat$ are of class $C^3$ on $\mathcal{O}$, since $\mathcal{K}$ is compact in $\mathcal{O}$, and since $r_1$ is affine in $s_3$ and $0<\varepsilon<1$, it follows that
\begin{align*}
  |\partial_s^\alpha\zeta_\varepsilon(s)| \leq c, \quad s\in\mathcal{K}\times[0,1], \quad |\alpha| = 1,2.
\end{align*}
Hence, differentiating $U(s)=u(\zeta_\varepsilon(s))$ and taking $s\in\mathcal{K}\times(0,1)$, we get
\begin{align} \label{Pf_28:DU}
  |\nabla_s^kU(s)| \leq c\sum_{\ell=1}^k|\nabla^\ell u(\zeta_\varepsilon(s))|, \quad s\in\mathcal{K}\times(0,1), \quad k=1,2.
\end{align}
Also, as in the proof of \eqref{E:Det_LPC}, we can deduce that
\begin{align*}
  \det\nabla_s\zeta_\varepsilon(s) = \varepsilon g^\flat(s')J(\psi(s'),\varepsilon r_1(s))\sqrt{\det\theta(s')}, \quad s=(s',s_3) \in \mathcal{O}\times[0,1],
\end{align*}
where $J$ and $\theta$ be given by \eqref{E:Def_J} and \eqref{E:Def_Rie}, respectively.
Thus,
\begin{align} \label{Pf_28:det}
  c\varepsilon \leq \det\nabla_s\zeta_\varepsilon(s) \leq c\varepsilon, \quad s \in \mathcal{K}\times[0,1]
\end{align}
by \eqref{E:G_Bdd}, \eqref{E:J_Bdd}, and \eqref{Pf_DLP:det}.
Since $U=u\circ\zeta_\varepsilon$ is supported in $\mathcal{K}\times[0,1]$, we have \eqref{Pf_28:Lp} by change of variables, \eqref{Pf_28:DU}, and \eqref{Pf_28:det}.
In particular, $U\in H^2(\mathcal{V}_1)$ and thus
\begin{align*}
  \|U\|_{L^p(\mathcal{V}_1)} \leq c\|U\|_{L^2(\mathcal{V}_1)}^{1-\sigma}\|U\|_{H^2(\mathcal{V}_1)}^\sigma, \quad \sigma = \frac{3}{4}-\frac{3}{2p}
\end{align*}
by the interpolation inequality in $\mathcal{V}_1=(0,1)^3$ (see \cite[Theorem 5.8]{AdaFou03}).
Hence, we get \eqref{E:Ipl_CTD} by applying \eqref{Pf_28:Lp} to the above inequality.

\subsection{Proof of Lemma \ref{L:DtEn_gSu}} \label{SS:PA_215}
First, we prove (i).
Let
\begin{align} \label{Pf_215:vsp}
  v \in \mathcal{E}_T(\Gamma)\cap L^2(0,T;H^3(\Gamma)) \subset C([0,T];H^1(\Gamma)),
\end{align}
where the inclusion is due to Lemma \ref{L:ET_Emb}.
We regularize $v$ in time by cut-off, translation, and mollification as in \cite[Lemma II.5.10]{BoyFab13} to get functions
\begin{align} \label{Pf_215:vk}
  \begin{gathered}
    v_k\in C^\infty([0,T];H^3(\Gamma)) \quad \text{such that} \\
    \lim_{k\to\infty}\|v-v_k\|_{L^2(0,T;H^3(\Gamma))} = \lim_{k\to\infty}\|\partial_tv-\partial_tv_k\|_{L^2(0,T;[H^1(\Gamma)]')} = 0.
  \end{gathered}
\end{align}
For each $v_k$, we can show \eqref{E:DtGr_gSu} on $(0,T)$ by direct calculations and \eqref{E:IbP_Ag}, since $\Gamma$, $g$, and $\nabla_\Gamma$ are independent of time.
Hence, by integration by parts in time, we have
\begin{align} \label{Pf_215:Agk}
  -\frac{1}{2}\int_0^T\partial_t\zeta\|\sqrt{g}\,\nabla_\Gamma v_k\|_{L^2(\Gamma)}^2\,dt = \int_0^T\zeta\langle\partial_tv_k,-gA_gv_k\rangle_{H^1(\Gamma)}\,dt
\end{align}
for all $\zeta=\zeta(t)\in C_c^1(0,T)$.
Moreover, noting that
\begin{align*}
  \|\sqrt{g}\,\nabla_\Gamma v\|_{L^2(\Gamma)} \leq c\|v\|_{H^1(\Gamma)}, \quad \|gA_gv\|_{H^1(\Gamma)} = \|\mathrm{div}_\Gamma(g\nabla_\Gamma v)\|_{H^1(\Gamma)} \leq c\|v\|_{H^3(\Gamma)}
\end{align*}
by $g\in C^3(\Gamma)$, we deduce from H\"{o}lder's inequality and $\zeta\in C_c^1(0,T)$ that
\begin{align*}
  \left|\int_0^T\partial_t\zeta\|\sqrt{g}\,\nabla_\Gamma v_k\|_{L^2(\Gamma)}^2\,dt-\int_0^T\partial_t\zeta\|\sqrt{g}\,\nabla_\Gamma v\|_{L^2(\Gamma)}^2\,dt\right| &\leq c\sigma_{1,k}, \\
  \left|\int_0^T\zeta\langle\partial_tv_k,-gA_gv_k\rangle_{H^1(\Gamma)}\,dt-\int_0^T\zeta\langle\partial_tv,-gA_gv\rangle_{H^1(\Gamma)}\,dt\right| &\leq c\sigma_{2,k},
\end{align*}
where
\begin{align*}
  \sigma_{1,k} &:= \Bigl(\|v_k\|_{L^2(0,T;H^1(\Gamma))}+\|v\|_{L^2(0,T;H^1(\Gamma))}\Bigr)\|v_k-v\|_{L^2(0,T;H^1(\Gamma))}, \\
  \sigma_{2,k} &:= \|\partial_tv_k-\partial_tv\|_{L^2(0,T;[H^1(\Gamma)]')}\|v_k\|_{L^2(0,T;H^3(\Gamma))} \\
  &\qquad\qquad +\|\partial_tv\|_{L^2(0,T;[H^1(\Gamma)]')}\|v_k-v\|_{L^2(0,T;H^3(\Gamma))}.
\end{align*}
Since $\sigma_{1,k},\sigma_{2,k}\to0$ as $k\to\infty$ by \eqref{Pf_215:vk}, we send $k\to\infty$ in \eqref{Pf_215:Agk} to find that
\begin{align} \label{Pf_215:Agv}
  -\frac{1}{2}\int_0^T\partial_t\zeta\|\sqrt{g}\,\nabla_\Gamma v\|_{L^2(\Gamma)}^2\,dt = \int_0^T\zeta\langle\partial_tv,-gA_gv\rangle_{H^1(\Gamma)}\,dt
\end{align}
for all $\zeta\in C_c^1(0,T)$, which shows that \eqref{E:DtGr_gSu} holds for $v$ a.e. on $(0,T)$.

Next, let us prove (ii).
By the assumption on $F$ and Remark \ref{R:Po_Grow},
\begin{align} \label{Pf_215:Fbd}
  |F^{(j)}(z)| \leq c(|z|^{4-j}+1), \quad z\in\mathbb{R}, \quad j=1,2,3,
\end{align}
where $F^{(j)}=d^jF/dz^j$.
Let $v$ be in the class \eqref{Pf_215:vsp}.
We first observe that
\begin{align*}
  F(v) \in C([0,T];L^1(\Gamma)), \quad F'(v)\in L^2(0,T;H^1(\Gamma)).
\end{align*}
The second assertion can be shown by \eqref{E:SoSu_H1} and \eqref{Pf_215:Fbd} as in the proofs of \eqref{E:FL2_CTD} and \eqref{E:FH1_CTD}.
We omit details here.
To prove the first assertion, let $t,s\in[0,T]$.
Then, since
\begin{align*}
  |F(v(t))-F(v(s))| \leq c(|v(t)|^3+|v(s)|^3+1)|v(t)-v(s)| \quad\text{a.e. on}\quad \Gamma
\end{align*}
by the mean value theorem of $F$ and \eqref{Pf_215:Fbd} with $j=1$, we have
\begin{align} \label{Pf_215:Fvt}
  \begin{aligned}
    \|F(v(t))-F(v(s))\|_{L^1(\Gamma)} &\leq c\|\,|v(t)|^3+|v(s)|^3+1\,\|_{L^{4/3}(\Gamma)}\|v(t)-v(s)\|_{L^4(\Gamma)} \\
    &\leq c\Bigl(\|v(t)\|_{L^4(\Gamma)}^3+\|v(s)\|_{L^4(\Gamma)}^3+1\Bigr)\|v(t)-v(s)\|_{L^4(\Gamma)} \\
    &\leq c\Bigl(\|v(t)\|_{H^1(\Gamma)}^3+\|v(s)\|_{H^1(\Gamma)}^3+1\Bigr)\|v(t)-v(s)\|_{H^1(\Gamma)}
  \end{aligned}
\end{align}
by H\"{o}lder's inequality and \eqref{E:SoSu_H1}.
Thus, $F(v) \in C([0,T];L^1(\Gamma))$ by \eqref{Pf_215:vsp}.

Now, let us show \eqref{E:DtFu_gSu}.
We take the functions $v_k$ satisfying \eqref{Pf_215:vk}.
Then,
\begin{align} \label{Pf_215:CH1}
  \lim_{k\to\infty}\|v-v_k\|_{C([0,T];H^1(\Gamma))} = 0
\end{align}
by the continuous embedding (see Lemma \ref{L:ET_Emb})
\begin{align*}
  \mathcal{E}_T(\Gamma)\cap L^2(0,T;H^3(\Gamma)) \hookrightarrow C([0,T];H^1(\Gamma)).
\end{align*}
For each $v_k$, we have \eqref{E:DtFu_gSu} on $(0,T)$ by $F\in C^3(\mathbb{R})$.
Hence,
\begin{align} \label{Pf_215:gF}
  -\int_0^T\partial_t\zeta\left(\int_\Gamma gF(v_k)\,d\mathcal{H}^2\right)\,dt = \int_0^T\zeta\langle\partial_tv_k,gF'(v_k)\rangle_{H^1(\Gamma)}\,dt
\end{align}
for all $\zeta\in C_c^1(0,T)$ by integration by parts.
Hence, if we have
\begin{align}
  \lim_{k\to\infty}\|F(v_k)-F(v)\|_{C([0,T];L^1(\Gamma))} &= 0, \label{Pf_215:FL1} \\
  \lim_{k\to\infty}\|F'(v_k)-F'(v)\|_{L^2(0,T;H^1(\Gamma))} &= 0, \label{Pf_215:FdH1}
\end{align}
then we can show that \eqref{Pf_215:gF} holds with $v_k$ replaced by $v$ as in the proof of \eqref{Pf_215:Agv}, which implies that \eqref{E:DtFu_gSu} is valid for $v$.
Let us prove \eqref{Pf_215:FL1} and \eqref{Pf_215:FdH1}.
We see that
\begin{align*}
  &\|F(v_k)-F(v)\|_{C([0,T];L^1(\Gamma))} \leq c\sigma_{3,k}\|v_k-v\|_{C([0,T];H^1(\Gamma))}, \\
  &\sigma_{3,k} := \|v_k\|_{C([0,T];H^1(\Gamma))}^3+\|v\|_{C([0,T];H^1(\Gamma))}^3+1
\end{align*}
by replacing $v(s)$ with $v_k(t)$ in \eqref{Pf_215:Fvt}.
Thus, \eqref{Pf_215:FL1} follows from \eqref{Pf_215:CH1}.
Next,
\begin{align*}
  |F'(v_k(t))-F'(v(t))| \leq c(|v_k(t)|^2+|v(t)|^2+1)|v_k(t)-v(t)| \quad\text{a.e. on}\quad \Gamma
\end{align*}
for all $t\in[0,T]$ by the mean value theorem of $F'$ and \eqref{Pf_215:Fbd} with $j=2$.
Hence,
\begin{align*}
  \|F'(v_k(t))-F'(v(t))\|_{L^2(\Gamma)} &\leq \|\,|v_k(t)|^2+|v(t)|^2+1\,\|_{L^3(\Gamma)}\|v_k(t)-v(t)\|_{L^6(\Gamma)} \\
  &\leq c\Bigl(\|v_k(t)\|_{L^6(\Gamma)}^2+\|v(t)\|_{L^6(\Gamma)}^2+1\Bigr)\|v_k(t)-v(t)\|_{L^6(\Gamma)} \\
  &\leq c\Bigl(\|v_k(t)\|_{H^1(\Gamma)}^2+\|v(t)\|_{H^1(\Gamma)}^2+1\Bigr)\|v_k(t)-v(t)\|_{H^1(\Gamma)}
\end{align*}
by H\"{o}lder's inequality and \eqref{E:SoSu_H1}, and it follows that
\begin{align} \label{Pf_215:FdL2}
  \begin{aligned}
    &\|F'(v_k)-F'(v)\|_{L^2(0,T;L^2(\Gamma))} \leq c\sigma_{4,k}\|v_k-v\|_{L^2(0,T;H^1(\Gamma))}, \\
    &\sigma_{4,k} := \|v_k\|_{C([0,T];H^1(\Gamma))}^2+\|v\|_{C([0,T];H^1(\Gamma))}^2+1.
  \end{aligned}
\end{align}
Also, since $\nabla_\Gamma F'(v)=F''(v)\nabla_\Gamma v$, we have
\begin{align*}
  \nabla_\Gamma F'(v_k)-\nabla_\Gamma F'(v) = \{F''(v_k)-F''(v)\}\nabla_\Gamma v_k+F''(v)(\nabla_\Gamma v_k-\nabla_\Gamma v)
\end{align*}
a.e. on $\Gamma\times(0,T)$.
Thus, by the mean value theorem of $F''$ and \eqref{Pf_215:Fbd},
\begin{align*}
  |\nabla_\Gamma F'(v_k)-\nabla_\Gamma F'(v)| &\leq c\{(|v_k|+|v|+1)|v_k-v||\nabla_\Gamma v_k(t)|+(|v|^2+1)|\nabla_\Gamma v_k-\nabla_\Gamma v|\}.
\end{align*}
Using this estimate, H\"{o}lder's inequality, and \eqref{E:SoSu_H1}, we can deduce that
\begin{multline} \label{Pf_215:FdGr}
  \|\nabla_\Gamma F'(v_k)-\nabla_\Gamma F'(v)\|_{L^2(0,T;L^2(\Gamma))} \\
  \leq c\Bigl(\sigma_{5,k}\|v_k-v\|_{C([0,T];H^1(\Gamma))}+\sigma_v\|\nabla_\Gamma v_k-\nabla_\Gamma v\|_{L^2(0,T;H^1(\Gamma))}\Bigr)
\end{multline}
as in \eqref{Pf_215:FdL2}, where $\sigma_v:=\|v\|_{C(0,T;H^1(\Gamma))}^2+1$ and
\begin{align*}
  \sigma_{5,k} := \Bigl(\|v_k\|_{C([0,T];H^1(\Gamma))}+\|v\|_{C([0,T];H^1(\Gamma))}+1\Bigr)\|\nabla_\Gamma v_k\|_{L^2(0,T;H^1(\Gamma))}.
\end{align*}
Therefore, we obtain \eqref{Pf_215:FdH1} by applying \eqref{Pf_215:vk} and \eqref{Pf_215:CH1} to \eqref{Pf_215:FdL2} and \eqref{Pf_215:FdGr}.

\subsection{Proof of Lemma \ref{L:DtIA_Sur}} \label{SS:PA_216}
Let $L_g$ be the operator given in Lemma \ref{L:InA_Sur} and
\begin{align*}
  f \in H^1(0,T;[H^1(\Gamma)]'), \quad \langle f(t),g\rangle_{H^1(\Gamma)} = 0 \quad\text{for all}\quad t\in[0,T].
\end{align*}
Then, $\langle\partial_tf(t),g\rangle_{H^1(\Gamma)}=0$ for a.a. $t\in(0,T)$, since $g$ is independent of $t$.
Hence, we can apply $L_g$ to $f(t)$ and $\partial_tf(t)$, and it follows from \eqref{E:IAS_Bdd} that
\begin{align*}
  L_gf, L_g(\partial_tf) \in L^2(0,T;H^1(\Gamma)).
\end{align*}
Let $\delta\in(0,T/2)$ and $h\in\mathbb{R}$ such that $0<|h|<\delta/2$.
We set
\begin{align*}
  D^hf(t) := \frac{f(t+h)-f(t)}{h}, \quad D^hL_gf(t) := \frac{L_gf(t+h)-L_gf(t)}{h}, \quad t\in(\delta,T-\delta).
\end{align*}
Then, since $L_g$ is linear and satisfy \eqref{E:IAS_Bdd}, we have
\begin{align*}
  \|D^hL_gf\|_{L^2(\delta,T-\delta;H^1(\Gamma))} \leq c\|D^hf\|_{L^2(\delta,T-\delta;[H^1(\Gamma)]')} \leq c\|\partial_tf\|_{L^2(0,T;[H^1(\Gamma)]')},
\end{align*}
where $c>0$ is a constant independent of $h$ and $\delta$ (see e.g. \cite[Section 5.8, Theorem 3 (i)]{Eva10} for the last inequality).
Therefore,
\begin{gather}
  \partial_tL_gf \in L^2(\delta,T-\delta;H^1(\Gamma)), \notag \\
  \lim_{h\to0}D^hL_gf = \partial_tL_gf \quad\text{weakly in}\quad L^2(\delta,T-\delta;H^1(\Gamma)) \label{Pf_216:Dh}
\end{gather}
by \cite[Section 5.8, Theorem 3 (ii)]{Eva10}.
Here and in what follows, the limit $h\to0$ is taken up to a subsequence.
On the other hand, since
\begin{align*}
  \lim_{h\to0}D^hf = \partial_tf \quad\text{weakly in}\quad L^2(\delta,T-\delta;[H^1(\Gamma)]')
\end{align*}
by $f\in H^1(0,T;[H^1(\Gamma)]')$, and since
\begin{align*}
  L_g\colon L^2(0,T;[H^1(\Gamma)]') \to L^2(0,T;H^1(\Gamma))
\end{align*}
is linear and bounded by \eqref{E:IAS_Bdd}, we have
\begin{align} \label{Pf_216:LD}
  \lim_{h\to0}L_g(D^hf) = L_g(\partial_tf) \quad\text{weakly in}\quad L^2(\delta,T-\delta,H^1(\Gamma)).
\end{align}
Since $D^hL_gf=L_g(D^hf)$ a.e. on $\Gamma\times(\delta,T-\delta)$, we find by \eqref{Pf_216:Dh} and \eqref{Pf_216:LD} that
\begin{align*}
  \partial_tL_gf = L_g(\partial_tf) \quad\text{a.e. on}\quad \Gamma\times(\delta,T-\delta) \quad\text{for all}\quad \delta\in(0,T/2).
\end{align*}
Hence, the same equality holds a.e. on $\Gamma\times(0,T)$, and we conclude that
\begin{align*}
  \partial_tL_gf = L_g(\partial_tf) \in L^2(0,T;H^1(\Gamma)), \quad L_gf\in H^1(0,T;H^1(\Gamma)).
\end{align*}
Let us show \eqref{E:DtIA_Sur}.
Let $v\in\mathcal{E}_T(\Gamma)$ satisfy $\int_\Gamma gv(t)\,d\mathcal{H}^2=0$ for all $t\in[0,T]$.
Then,
\begin{align*}
  L_gv \in H^1(0,T;H^1(\Gamma)) \subset \mathcal{E}_T(\Gamma), \quad \partial_tL_gv = L_g(\partial_tv) \quad\text{a.e. on}\quad \Gamma\times(0,T)
\end{align*}
by the definition \eqref{E:Def_ET} of $\mathcal{E}_T(\Gamma)$ and the above discussions.
We see that
\begin{align*}
  \int_\Gamma g|\nabla_\Gamma L_gv|^2\,d\mathcal{H}^2 = (g\nabla_\Gamma L_gv,\nabla_\Gamma L_gv)_{L^2(\Gamma)} = \langle v,gL_gv\rangle_{H^1(\Gamma)} = (gv,L_gv)_{L^2(\Gamma)}
\end{align*}
on $[0,T]$ by \eqref{E:InA_Sur} (recall that we set $\langle f,\cdot\rangle_{H^1(\Gamma)}=(f,\cdot)_{L^2(\Gamma)}$ for $f\in L^2(\Gamma)$).
Then, noting that $v,L_gv\in\mathcal{E}_T(\Gamma)$, we take the time derivative and apply \eqref{E:DtL2_gSu} to get
\begin{align} \label{Pf_216:dtNo}
  \frac{d}{dt}\int_\Gamma g|\nabla_\Gamma L_gv|^2\,d\mathcal{H}^2 = \langle\partial_tv,gL_gv\rangle_{H^1(\Gamma)}+\langle\partial_tL_gv,gv\rangle_{H^1(\Gamma)}
\end{align}
a.e. on $(0,T)$.
Moreover, since $\partial_tL_gv=L_g(\partial_tv)\in L^2(0,T;H^1(\Gamma))$,
\begin{align*}
  \langle\partial_tL_gv,gv\rangle_{H^1(\Gamma)} = \langle L_g(\partial_tv),gv\rangle_{H^1(\Gamma)} = (L_g(\partial_tv),gv)_{L^2(\Gamma)} = (v,gL_g(\partial_tv))_{L^2(\Gamma)},
\end{align*}
and we use \eqref{E:InA_Sur} with $f=v$ and $\eta=L_g(\partial_tv)$ to get
\begin{align*}
  (v,gL_g(\partial_tv))_{L^2(\Gamma)} = \langle v,gL_g(\partial_tv)\rangle_{H^1(\Gamma)} = (g\nabla_\Gamma L_gv,\nabla_\Gamma L_g(\partial_tv))_{L^2(\Gamma)}.
\end{align*}
To the last term, we again apply \eqref{E:InA_Sur} with $f=\partial_tv$ and $\eta=v$.
Then,
\begin{align*}
  (g\nabla_\Gamma L_gv,\nabla_\Gamma L_g(\partial_tv))_{L^2(\Gamma)} = (g\nabla_\Gamma L_g(\partial_tv),\nabla_\Gamma L_gv)_{L^2(\Gamma)} = \langle\partial_tv,gL_gv\rangle_{H^1(\Gamma)}
\end{align*}
and thus $\langle\partial_tL_gv,gv\rangle_{H^1(\Gamma)}=\langle\partial_tv,gL_gv\rangle_{H^1(\Gamma)}$.
Applying this relation to \eqref{Pf_216:dtNo} and dividing both sides by two, we obtain \eqref{E:DtIA_Sur}.

\section{Galerkin method for the thin-domain problem} \label{S:Galer}
Let us briefly explain the outline of construction of a weak solution to the thin-domain problem \eqref{E:CH_CTD} by the Galerkin method.
In what follows, we fix and suppress the parameter $\varepsilon$.
For example, we write $\Omega$ and $u$ instead of $\Omega_\varepsilon$ and $u^\varepsilon$.
Note that constants in this section depend on $\varepsilon$ implicitly, but it does not matter since we do not send $\varepsilon\to0$.
Moreover, we abbreviate function spaces $\mathcal{X}(\Omega)$ on $\Omega$ to $\mathcal{X}$ for the sake of simplicity.

Let $u_0\in H^1$ be the initial data.
Suppose first that there exists a weak solution $(u_T,w_T)$ to \eqref{E:CH_CTD} on $[0,T)$ for all $T>0$.
Then, since $u_T=u_{T'}$ and $w_T=w_{T'}$ on $[0,T)$ for $T<T'$ by Proposition \ref{P:CHT_Uni}, we can get a global weak solution $(u,w)$ to \eqref{E:CH_CTD} by setting $u:=u_T$ and $w:=w_T$ on $[0,T)$ for each $T>0$.
Hence, it is enough to show the existence of a weak solution to \eqref{E:CH_CTD} on $[0,T)$ for each fixed $T>0$.

\subsection{Approximate solutions} \label{SS:Ga_App}
Let $\{\lambda_k\}_{k=1}^\infty$ and $\{\phi_k\}_{k=1}^\infty$ be the eigenvalues and corresponding eigenfunctions of the Neumann problem
\begin{align*}
  -\Delta\phi_k = \lambda_k\phi_k \quad\text{in}\quad \Omega, \quad \partial_\nu\phi_k = 0 \quad\text{on}\quad \partial\Omega
\end{align*}
such that $\lambda_1=0$ and $\phi_1$ is constant, $\lambda_k\leq\lambda_{k+1}$ for all $k\in\mathbb{N}$, and $\lambda_k\to\infty$ as $k\to\infty$.
By normalizing $\phi_k$, we may assume that (here, $\delta_{k\ell}$ is the Kronecker delta)
\begin{align} \label{E:Ga_Orth}
  (\phi_k,\phi_\ell)_{L^2} = \delta_{k\ell}, \quad (\nabla\phi_k,\nabla\phi_\ell)_{L^2} = (-\Delta\phi_k,\phi_\ell)_{L^2} = \lambda_k\delta_{k\ell},
\end{align}
so that $\{\phi_k\}_{k=1}^\infty$ is an orthonormal basis of $L^2$ and an orthogonal basis of $H^1$.
Let
\begin{align*}
  V_K := \mathrm{span}\{\phi_1,\dots,\phi_K\}, \quad \mathcal{P}_Ku := \sum_{k=1}^K(u,\phi_k)_{L^2}\phi_k, \quad u\in L^2, \quad K\in\mathbb{N},
\end{align*}
so that $\mathcal{P}_K\colon L^2\to V_K$ is an orthogonal projection.
By \eqref{E:Ga_Orth}, we see that
\begin{align} \label{E:Ga_OrPr}
  \|\mathcal{P}_Ku\|_{\mathcal{X}}\leq\|u\|_{\mathcal{X}}, \quad \lim_{K\to\infty}\|\mathcal{P}_Ku-u\|_{\mathcal{X}} = 0, \quad \mathcal{X} = L^2,H^1.
\end{align}
For each $K\in\mathbb{N}$, we look for functions
\begin{align*}
  u_K(t) = \sum_{k=1}^K\alpha_k(t)\phi_k, \quad w_K(t) = \sum_{k=1}^N\beta_k(t)\phi_k
\end{align*}
satisfying $u_K(0)=\mathcal{P}_Ku_0$ and, for all $\phi\in V_K$ and $t\in(0,T)$,
\begin{align} \label{E:Ga_ApPr}
  \begin{aligned}
    &(\partial_tu_K(t),\phi)_{L^2}+(\nabla w_K(t),\nabla\phi)_{L^2} = 0, \\
    &(w_K(t),\phi)_{L^2} = (\nabla u_K(t),\nabla\phi)_{L^2}+(F'(u_K(t)),\phi)_{L^2}.
  \end{aligned}
\end{align}
We set $\phi=\phi_k$ in \eqref{E:Ga_ApPr} and use \eqref{E:Ga_Orth} to derive equations for $\alpha_k$ and $\beta_k$.
Then, we substitute the equation of $\beta_k$ for the one of $\alpha_k$ to get the ODE system
\begin{align*}
  \frac{d\bm{\alpha}}{dt}(t) = \mathbf{f}(\bm{\alpha}(t)), \quad \bm{\alpha} = (\alpha_1,\dots,\alpha_K)^T\in\mathbb{R}^K, \quad \mathbf{f}\colon\mathbb{R}^K\to\mathbb{R}^K.
\end{align*}
Since $\mathbf{f}$ is locally Lipschitz continuous by Remark \ref{R:Po_Grow}, we can solve this system locally in time by the Cauchy--Lipschitz theorem to get a local-in-time solution to \eqref{E:Ga_ApPr}.

\subsection{Energy estimate} \label{SS:Ga_Ena}
Next, we derive estimates for $u_K$ and $w_K$.
Let
\begin{align*}
  E(u_K(t)) = \int_\Omega\left(\frac{|\nabla u_K(t)|^2}{2}+F(u_K(t))\right)\,dx
\end{align*}
be the Ginzburg--Landau free energy for $u_K(t)$.
Since $\partial_tu_K(t),w_K(t)\in V_K$,
\begin{align*}
  \frac{d}{dt}E(u_K(t)) &= \bigl(\nabla u_K(t),\nabla\partial_tu_K(t))_{L^2}+(F'(u_K(t)),\partial_tu_K(t))_{L^2} \\
  &= (w_K(t),\partial_tu_K(t))_{L^2} = -\|\nabla w_K(t)\|_{L^2}^2
\end{align*}
by \eqref{E:Ga_ApPr}.
Hence, integrating with respect to time, we obtain
\begin{align*}
  E(u_K(t))+\int_0^t\|\nabla w_K(s)\|_{L^2}^2\,ds = E(\mathcal{P}_Ku_0)
\end{align*}
as long as $u_K$ and $w_K$ exist.
Then, as in Proposition \ref{P:CHT_GLE}, we have
\begin{align*}
  \frac{1}{2}\|\nabla u_K(t)\|_{L^2}^2+\int_0^t\|\nabla w_K(s)\|_{L^2}^2\,ds \leq c\{|E(\mathcal{P}_Ku_0)|+1\}
\end{align*}
by Assumption \ref{A:Poten}.
Moreover, by Remark \ref{R:Po_Grow}, \eqref{E:Sob_CTD} with $p=4$, and \eqref{E:Ga_OrPr},
\begin{align*}
  |E(\mathcal{P}_Ku_0)| \leq c\Bigl(\|\nabla\mathcal{P}_Ku_0\|_{L^2}^2+\|\mathcal{P}_Ku_0\|_{L^4}^4+1\Bigr) \leq C(u_0).
\end{align*}
Here and in what follows, we write $C(u_0)$ for a general positive constant depending only on $\|u_0\|_{H^1}$ (recall that we fix and suppress $\varepsilon$).
Thus,
\begin{align} \label{E:Ga_EnGr}
  \|\nabla u_K(t)\|_{L^2}^2+\int_0^t\|\nabla w_K(s)\|_{L^2}^2\,ds \leq C(u_0).
\end{align}
By \eqref{E:Ga_Orth} and \eqref{E:Ga_EnGr}, the coefficients $\alpha_k(t)$ are bounded uniformly in $t$.
Hence, we can extend $u_K(t)$ and $w_K(t)$ up to $t=T$.
Also, we see by Poincar\'{e}'s inequality that
\begin{align*}
  \|u_K(t)\|_{H^1} \leq c\Bigl(\|\nabla u_K(t)\|_{L^2}+|(u_K(t),1)_{L^2}|\Bigr),
\end{align*}
and we take $\phi=1$ in \eqref{E:Ga_ApPr}, which is allowed since $\phi_1$ is constant, to get
\begin{align*}
  \frac{d}{dt}(u_K(t),1)_{L^2} = (\partial_tu_K(t),1)_{L^2} = 0, \quad (u_K(t),1)_{L^2} = (\mathcal{P}_Ku_0,1)_{L^2}.
\end{align*}
Hence, it follows from H\"{o}lder's inequality, \eqref{E:Ga_OrPr}, and \eqref{E:Ga_EnGr} that
\begin{align} \label{E:Ga_uKH1}
  \|u_K(t)\|_{H^1} \leq c\Bigl(\|\nabla u_K(t)\|_{L^2}+\|\mathcal{P}_Ku_0\|_{L^2}\Bigr) \leq C(u_0), \quad t\in[0,T].
\end{align}
Also, we set $\phi=1$ in the second line of \eqref{E:Ga_ApPr} to find that
\begin{align*}
  |(w_K(t),1)_{L^2}| = |(F'(u_K(t)),1)_{L^2}| \leq c\Bigl(\|u_K(t)\|_{L^3}^3+1\Bigr) \leq C(u_0)
\end{align*}
by Remark \ref{R:Po_Grow}, \eqref{E:Sob_CTD} with $p=3$, and \eqref{E:Ga_uKH1}.
Hence,
\begin{align} \label{E:Ga_wKH1}
  \int_0^T\|w_K(t)\|_{H^1}^2\,dt \leq c\int_0^T\Bigl(\|\nabla w_K(t)\|_{L^2}^2+|(w_K(t),1)_{L^2}|^2\Bigr)\,dt \leq (1+T)C(u_0)
\end{align}
by Poincar\'{e}'s inequality and \eqref{E:Ga_EnGr}.

We also estimate $\partial_tu_K$.
Let $\varphi\in L^2(0,T;H^1)$.
For a.a. $t\in(0,T)$, we have
\begin{align*}
  (\partial_tu_K(t),\varphi(t))_{L^2} = (\partial_tu_K(t),\mathcal{P}_K\varphi(t))_{L^2} = -(\nabla w_K(t),\nabla\mathcal{P}_K\varphi(t))_{L^2}
\end{align*}
by $\partial_tu_K(t)\in V_K$, \eqref{E:Ga_Orth}, and \eqref{E:Ga_ApPr} (note that we do not assume $\varphi(t)\in V_K$).
Hence,
\begin{align*}
  \left|\int_0^T(\partial_tu_K(t),\varphi(t))_{L^2}\,dt\right| &\leq \|\nabla w_K\|_{L^2(0,T;L^2)}\|\nabla\mathcal{P}_K\varphi\|_{L^2(0,T;L^2)} \\
  &\leq C(u_0)\|\varphi\|_{L^2(0,T;H^1)}
\end{align*}
by H\"{o}lder's inequality, \eqref{E:Ga_OrPr}, and \eqref{E:Ga_EnGr}.
This yields
\begin{align} \label{E:Ga_dtuK}
  \|\partial_tu_K\|_{L^2(0,T;[H^1]')} \leq C(u_0).
\end{align}

\subsection{Convergence to a weak solution} \label{SS:Ga_Con}
By \eqref{E:Ga_uKH1}--\eqref{E:Ga_dtuK}, there exist functions
\begin{align*}
  u\in\mathcal{E}_T\cap L^\infty(0,T;H^1), \quad w\in L^2(0,T;H^1),
\end{align*}
where $\mathcal{E}_T=\mathcal{E}_T(\Omega)$ is given in \eqref{E:Def_ET}, such that
\begin{align} \label{E:Ga_WeCo}
  \begin{alignedat}{3}
      \lim_{K\to\infty}u_K &= u &\quad &\text{weakly-$\ast$ in} &\quad &L^\infty(0,T;H^1), \\
      \lim_{K\to\infty}\partial_tu_K &= \partial_tu &\quad &\text{weakly in} &\quad &L^2(0,T;[H^1]'), \\
      \lim_{K\to\infty}w_K &= w &\quad &\text{weakly in} &\quad &L^2(0,T;H^1)
    \end{alignedat}
\end{align}
up to a subsequence.
Also, by Remark \ref{R:Po_Grow}, \eqref{E:Sob_CTD} with $p=6$, and \eqref{E:Ga_uKH1},
\begin{align*}
  \|F'(u_K(t))\|_{L^2} \leq c\Bigl(\|u_K(t)\|_{L^6}^3+1\Bigr) \leq c\Bigl(\|u_K(t)\|_{H^1}^3+1\Bigr) \leq C(u_0), \quad t\in[0,T].
\end{align*}
This gives $\|F'(u_K)\|_{L^2(0,T;L^2)}\leq T^{1/2}C(u_0)$, and thus we can prove that
\begin{align} \label{E:Ga_FCo}
  \lim_{K\to\infty}F'(u_K) = F'(u) \quad\text{weakly in}\quad L^2(0,T;L^2)
\end{align}
by applying the Aubin--Lions lemma (see \cite[Theorem II.5.16]{BoyFab13}) to  $\{u_K\}_K$ and using Lemma \ref{L:WeDo} as in the proof of Theorem \ref{T:TFL_Weak} (see Section \ref{SS:TFL_WC}).

Let us show \eqref{E:CHT_WF_u} and \eqref{E:CHT_WF_w}.
Let $\varphi\in L^2(0,T;H^1)$.
For a.a. $t\in(0,T)$,
\begin{align*}
  \|\mathcal{P}_L\varphi(t)-\varphi(t)\|_{H^1} \leq 2\|\varphi(t)\|_{H^1}, \quad \lim_{L\to\infty}\|\mathcal{P}_L\varphi(t)-\varphi(t)\|_{H^1} = 0
\end{align*}
by \eqref{E:Ga_OrPr}.
Thus, by the dominated convergence theorem, we have
\begin{align} \label{E:Ga_Test}
  \lim_{L\to\infty}\|\mathcal{P}_L\varphi-\varphi\|_{L^2(0,T;H^1)} = 0.
\end{align}
When $L\leq K$, we can take $\phi=\mathcal{P}_L\varphi(t)\in V_K$ in \eqref{E:Ga_ApPr} to get
\begin{align} \label{E:Ga_uKpL}
  \begin{aligned}
    &\int_0^T(\partial_tu_K,\mathcal{P}_L\varphi)_{L^2}\,dt+\int_0^T(\nabla w_K,\nabla\mathcal{P}_L\varphi)_{L^2}\,dt = 0, \\
    &\int_0^T(w_K,\mathcal{P}_L\varphi)_{L^2}\,dt = \int_0^T(\nabla u_K,\nabla\mathcal{P}_L\varphi)_{L^2}\,dt+\int_0^T(F'(u_K),\mathcal{P}_L\varphi)_{L^2}\,dt.
  \end{aligned}
\end{align}
In \eqref{E:Ga_uKpL}, we first send $K\to\infty$ and use \eqref{E:Ga_WeCo} and \eqref{E:Ga_FCo}.
After that, we send $L\to\infty$ and apply \eqref{E:Ga_Test}.
Then, we obtain \eqref{E:CHT_WF_u} and \eqref{E:CHT_WF_w}.

To verify the initial condition, let $f\in C^1([0,T])$ satisfy $f(0)=1$ and $f(T)=0$.
For any $\varphi_0\in H^1$, we take $\varphi(x,t):=f(t)\varphi_0(x)$ in the first line of \eqref{E:Ga_uKpL}.
Then,
\begin{align*}
  -(\mathcal{P}_Ku_0,\mathcal{P}_L\varphi_0)_{L^2}-\int_0^T(u_K,f'\mathcal{P}_L\varphi_0)_{L^2}\,dt+\int_0^T(\nabla w_K,f\nabla\mathcal{P}_L\varphi_0)_{L^2}\,dt = 0
\end{align*}
with $f'=df/dt$ by integration by parts with respect to time.
Again, we first send $K\to\infty$ and use \eqref{E:Ga_OrPr} and \eqref{E:Ga_WeCo}, and then send $L\to\infty$ and apply \eqref{E:Ga_OrPr}.
Then, we get
\begin{align*}
  -(u_0,\varphi_0)_{L^2}-\int_0^T(u,f'\varphi_0)_{L^2}\,dt+\int_0^T(\nabla w,f\nabla\varphi_0)_{L^2}\,dt = 0.
\end{align*}
On the other hand, we set $\varphi=f\varphi_0$ in \eqref{E:CHT_WF_u} and use \eqref{E:DtL2_CTD} to find that
\begin{align*}
  -(u(0),\varphi_0)_{L^2}-\int_0^T(u,f'\varphi_0)_{L^2}\,dt+\int_0^T(\nabla w,f\nabla\varphi_0)_{L^2}\,dt = 0.
\end{align*}
Thus, $(u(0),\varphi_0)_{L^2}=(u_0,\varphi_0)_{L^2}$ for all $\varphi_0\in H^1$, which yields $u(0)=u_0$ a.e. on $\Omega$ since $H^1$ is dense in $L^2$.
Hence, $(u,w)$ is a weak solution to \eqref{E:CH_CTD} on $[0,T)$.

\section*{Acknowledgments}
The work of the author was supported by JSPS KAKENHI Grant Number 23K12993.

\bibliographystyle{abbrv}
\bibliography{CH_stCTD_Ref}

\end{document}